\documentclass[letterpaper,11pt]{article}
\usepackage[backend=biber,style=alphabetic,citestyle=alphabetic,maxnames=5,maxbibnames=9, maxcitenames=6,mincitenames=4,maxcitenames=5,doi=false,isbn=false,url=false]{biblatex} 
\AtEveryBibitem{\clearlist{language}}
\usepackage{graphicx} 
\usepackage{palatino,mathrsfs}
\usepackage[mathcal]{euler}
\usepackage{xcolor}
\usepackage{epstopdf,epsfig}
\usepackage{pdfpages}

\usepackage[colorlinks=true,linkcolor=black,hypertexnames=false]{hyperref}
\usepackage{comment} 
\usepackage{amsmath,amsthm,amsfonts,amssymb,latexsym,amscd,enumerate,url,hyperref}
\usepackage{amssymb}
\usepackage{url}
\usepackage{mathtools}
\usepackage{graphicx}
\usepackage[all,knot]{xy}

\xyoption{arc}
\usepackage{multirow}
\usepackage{caption}
\usepackage{soul}

\newtheorem*{theorem*}{Theorem}
 \newtheorem{corollary}{Corollary} 
 \newtheorem*{corollary*}{Corollary}
\newtheorem{lemma}{Lemma}
\newtheorem*{lemma*}{Lemma}
\newtheorem{proposition}{Proposition}
\newtheorem*{proposition*}{Proposition}

\newtheorem*{claim*}{Claim}
\theoremstyle{definition}

\newtheorem*{definition*}{Definition}
 
\newtheorem*{example*}{Example}

\newtheorem*{question*}{Question}
  
\newtheorem*{problem*}{Problem}
  
\newtheorem*{remark*}{Remark}

\def\R{\mathbb R}
\def\C{\mathbb C}
\def\N{\mathbb N}
\def\Z{\mathbb Z}

\def\fg{\mathfrak{g}}
\def\fk{\mathfrak{k}}

\def\ft{\mathfrak{t}}

\def\fh{\mathfrak{h}}
\def\fl{\mathfrak{l}}
\def\fp{\mathfrak{p}}
\def\g{\boldsymbol{\mathfrak{g}}}
\def\k{\boldsymbol{\mathfrak{k}}}
\def\p{\boldsymbol{\mathfrak{p}}}

\def\t{\boldsymbol{\mathfrak{t}}}

\def\l{\boldsymbol{\mathfrak{l}}}

\title{Dirac operators for algebraic families}

\author{Spyridon Afentoulidis-Almpanis \& Eyal Subag
}

\newcommand{\Addresses}{{
		\bigskip
		\footnotesize
		
		Spyridon~Afentoulidis-Almpanis \par\nopagebreak\textsc{Dept. of Mathematics, Bar-Ilan University, Ramat-Gan, 5290002 Israel}\par\nopagebreak
		\textit{E-mail address}: \texttt{spyridon.almpanis@biu.ac.il}
		
		\medskip
		
Eyal~Subag \par\nopagebreak\textsc{Dept. of Mathematics, Bar-Ilan University, Ramat-Gan, 5290002 Israel}\par\nopagebreak
\textit{E-mail address}: \texttt{eyal.subag@biu.ac.il}
		
		\medskip

}}

\date{\today}

\begin{document}

\maketitle

 \begin{abstract} We introduce algebraic families of Dirac operators for the deformation family (and other related families) associated with a real reductive Lie group that interpolates the reductive group and the corresponding Cartan motion group. We prove  Vogan's conjecture in this setting, relating the infinitesimal character of an algebraic family of  Harish-Chandra modules and its Dirac cohomology. 
	\end{abstract}

\tableofcontents

\section{Introduction}

Algebraic Dirac operators were introduced by Vogan as a tool to study representations of real reductive Lie groups. 

The deformation family of Harish-Chandra pairs associated with a real reductive Lie group $G(\mathbb{R})$ interpolates $G(\mathbb{R})$ and its Cartan motion group.

In this paper, we explore algebraic Dirac operators in the context of algebraic families of Harish-Chandra modules for the deformation family (and other related families) of a real reductive Lie group. We begin by providing the necessary background and context for both algebraic families and Dirac operators, followed by a summary of our main contributions and results.

\subsection{Motivation and context}

Algebraic families of Harish-Chandra pairs and their modules were introduced by Bernstein et al.~\cite{eyalbern,MR4123111}. Their main goal was to provide an algebraic framework for studying limits of Lie groups and their representations, commonly referred to as \emph{contractions} in the physics literature (e.g., \cite{MR3077834,MR779059,MR1275599,MR55352,MR2942592}). In particular, they constructed two one-parameter algebraic families of Harish-Chandra pairs interpolating between a real reductive Lie group $G(\mathbb{R})$ and its Cartan motion group $G(\mathbb{R})_0$. One of these leads to the \emph{contraction family}, while the other gives the \emph{deformation family} associated with $G(\mathbb{R})$.

Both families have significant applications. For instance, in the case of $G(\mathbb{R}) = SO(3,1)$, the contraction family and its modules play a central role in the quantum mechanics of the hydrogen atom system \cite{MR4460278,MR3827131}. On the other hand, the deformation family has been used in the study of the \emph{Mackey-Higson bijection}, which we now briefly recall.

Inspired by certain contractions in physics, Mackey observed an analogy between many unitary irreducible representations of a non-compact semisimple Lie group $G(\mathbb{R})$ and those of its Cartan motion group $G(\mathbb{R})_0$ \cite{MR409726}. Decades later, in the case of complex semisimple groups Higson extended this analogy to a bijection, first between the tempered duals \cite{MR2391803},
\[
\mathcal{MH} : \widehat{G(\mathbb{R})}_{\text{temp}} \xrightarrow{\sim} \widehat{G(\mathbb{R})_0}_{\text{temp}},
\]
and then between the admissible duals \cite{MR2815133},
\[
\mathcal{MH} : \widehat{G(\mathbb{R})}_{\text{adms}} \xrightarrow{\sim} \widehat{G(\mathbb{R})_0}_{\text{adms}}.
\]
This bijection played a key role in Higson's new proof of the Connes-Kasparov conjecture for complex semisimple Lie groups \cite{MR2391803}. Subsequently, Afgoustidis very elegantly extended these results to the setting of real reductive Lie groups \cite{MR4079418,MR4400734}. For a different approach using $\mathcal{D}$-modules, see \cite{MR4542720}.

Underlying the works on the Mackey-Higson bijection is a  family of Lie groups $\{G_t\}_{t \in [0,1]}$ satisfying
\[
G_t \cong
\begin{cases}
G(\mathbb{R}) & t \neq 0, \\
G(\mathbb{R})_0 & t = 0.
\end{cases}
\]
This family is a real continuous  counterpart of the algebraic deformation family of Harish-Chandra pairs that was mentioned above.

In \cite{eyallietheory}, algebraic families of Harish-Chandra modules for the deformation family were used to formulate a conjectural algebraic characterization of the Mackey-Higson bijection which was verified in the case of $G(\mathbb{R}) = SL_2(\mathbb{R})$.

Following Afgoustidis, we denote by $\widehat{G(\mathbb{R})}_{\text{tempiric}}$ the set of equivalence classes of \emph{tempiric} representations of $G(\mathbb{R})$, i.e., tempered irreducible representations with real infinitesimal character. 

 Vogan showed that a tempiric representation has a unique lowest $K(\mathbb{R})$-type, where  $K(\R)$ is  a maximal compact subgroups of $G(\R)$.  He also proven that the map $L_K$ sending a tempiric representation to its unique minimal $K(\mathbb{R})$-type defines a bijection onto $\widehat{K(\mathbb{R})}$ \cite{MR2401817}.  Under the natural inclusion $\widehat{K(\mathbb{R})} \hookrightarrow \widehat{G(\mathbb{R})_0}_{\text{temp}}$, the Mackey-Higson bijection extends Vogan's:
\[
\xymatrix{
\widehat{G(\mathbb{R})}_{\text{tempiric}} \ar[d]^{\cong}_{L_K} \ar@{^{(}->}[rr] & & \widehat{G(\mathbb{R})}_{\text{adms}} \ar[d]^{\cong}_{\mathcal{MH}} \\
\widehat{K(\mathbb{R})} \ar@{^{(}->}[rr] & & \widehat{G(\mathbb{R})_0}_{\text{adms}}
}
\]

When $G(\mathbb{R})$ is an  equal-rank group, an important part of the tempiric dual of $G(\mathbb{R})$ consists of discrete series representations, i.e., unitary irreducible representations which appear discretely in the Plancherel formula of $G(\mathbb{R})$. 

In \cite{MR463358,Par}, Parthasarathy and Atiyah and Schmid utilized geometric Dirac operators to construct all discrete series representations of $G(\mathbb{R})$. The geometric Dirac operator of $G(\R)$ is a  differential operator on sections of an equivariant bundle on the homogeneous space  $G(\R)/K(\R)$.

In \cite{MR4079418}, Afgoustidis introduced a parameter $t\in[0,1]$ in Parthasarathy's abovementioned construction so that he got a one-parameter family of twisted geometric Dirac operators acting on homogeneous vector bundles defined over the homogeneous spaces $\{G_t/K(\mathbb{R})\}$. In this way, he managed to explicitly construct the Mackey-Higson bijection for the discrete series representations of $G(\mathbb{R})$ as contractions to the corresponding representations of $G(\mathbb{R})_0$.  


An important ingredient of what is treated in the present work will be the algebraic counterpart $D_{\mathfrak{g},K}$ of the abovementioned Parthasarathy's Dirac operator. It was introduced by Vogan in a series of  talks in MIT \cite{vogantalks} and since then it is known as \emph{the algebraic Dirac operator}. If  $\mathfrak{g}=\mathfrak{k}\oplus\mathfrak{p}$ stands for the Cartan decomposition of $\mathfrak{g}$, the algebraic Dirac operator $D_{\mathfrak{g},K}$ is an element of the  algebra $U(\mathfrak{g})\otimes_{\C} Cl(\mathfrak{p})$ where $Cl(\mathfrak{p})$ is the Clifford algebra of $\mathfrak{p}$. 

Whenever $V$ is a $(\mathfrak{g},K)$-module and $S$ is a spin module of $Cl(\mathfrak{p})$, then $D_{\mathfrak{g},K}$ defines a linear operator, denoted again by $D_{\mathfrak{g},K}$, acting on the tensor product $V\otimes_{\C} S$. The tensor product $V\otimes_{\C} S$ admits a natural action of the spin double cover $\widetilde{K}$ of $K$. The operator $D_{\mathfrak{g},K}$ is $\widetilde{K}$ -equivariant and  the quotient
\begin{equation*}
H_{D_{\mathfrak{g},K}}(V) := \frac{\ker D_{\mathfrak{g},K}}{\ker D_{\mathfrak{g},K} \cap \mathrm{im}\,D_{\mathfrak{g},K}},
\end{equation*}
called \emph{the Dirac cohomology of $V$}, is a representation of $\widetilde{K}$.
Vogan conjectured and Huang and Pand{\v{z}}i{\'c} proven \cite{huangpandzic} that, for an irreducible $(\mathfrak{g},K)$-module, whenever $H_{D_{\mathfrak{g},K}}(V)$ is nonzero, $H_{D_{\mathfrak{g},K}}(V)$ completely determines  the infinitesimal character of $V$.

The proof of Vogan's conjecture relies on  the fact that the square of the Dirac operator  acts as a scalar  on each $\widetilde{K}$-isotypic component.


Other features of algebraic Dirac operators include the Dirac-Parthasarathy inequality \cite{parthaineq} which provides a powerful method to determine nonunitarity for irreducible representations of $G(\mathbb{R})$, relations with coherent families and characters \cite{diracindex,familiesCroate,MPVZ}, applications to K-Theory \cite{clare2,clare1} and the Connes-Kasparov conjecture \cite{haluk}. Moreover, Dirac cohomology is related to  Lie algebra cohomology \cite{renardpandzic}, the BGG category $\mathcal{O}$ \cite{afe2}, and the theta correspondence \cite{afe3}. 

Dirac cohomology has been computed for various
families of modules, including highest weight modules in \cite{HX}, $A_q(\lambda)$-modules in \cite{kang-pandzic},
generalized Enright-Varadarajan modules in \cite{MP10}, unipotent representations in \cite{dongwong,BP2,BP1}, Jacquet modules in \cite{DH} and $(\mathfrak{g},K)$-modules with generalized infinitesimal characters in \cite{somberg}.

In light of the above discussion, several questions arise naturally.

\begin{itemize}
	\item Is there a well-defined algebraic Dirac operator  for the deformation  family of a real reductive group $G(\mathbb{R})$? 

\item If so, does its square  act as a scalar  on each $\widetilde{K}$-isotypic component, as in  the group case?

\item Is there a version of Vogan's conjecture for algebraic families of Harish-Chandra modules  for the deformation  family of a real reductive group?

\item Can Dirac operators be used to give a new  algebraic perspective for Vogan's bijection?
\end{itemize}

In this paper,  we answer all of  these question except for the very last one dealing with its relation to Vogan's bijection. We shall address this elsewhere.

It should be noted that beyond the mentioned work of Bernstein et al.  families of Harish-Chandra modules  naturally  appear in representation theory of real reductive  groups.

For example, the  algorithm underlying  the \textit{atlas} software for  determining the unitary dual of a real reductive group  uses families of Harish-Chandra modules for a constant family of Harish-Chandra pairs \cite{MR4146144}. 

 Another example is the study of   analytic families of Harish-Chandra modules   in the Thesis of van der Noort  \cite{Noort}; see also the related work of Wallach \cite{MR4389792}. 

A more general case of Harish-Chandra modules over commutative rings was  analyzed in 
\cite{MR3853058,MR4007195,MR4166966}.
 Also see \cite{Fabian}  for   related   families of $\mathcal{D}$-modules.

\subsection{Review and main results}

The main goal  of this paper is to develop a unified algebraic Dirac operator theory for algebraic families and in particular for deformation and deformation-like families of real reductive groups.

 A key result of this paper is the formulation and proof of a theorem that generalizes Vogan's conjecture for generically irreducible families of Harish-Chandra modules for deformation-like families. Localization plays a major role in the proof of the abovementioned theorem.  

 Section \ref{Spinreprese} serves as   an introduction to the theory of algebraic Dirac operators for a real reductive group  as presented in \cite{pandzic}. Namely, we recall the notion of the algebraic Dirac operator and Dirac cohomology while we discuss Vogan's conjecture.

Section \ref{algebraicf} introduces the notions of algebraic families of Harish-Chandra pairs, generalized pairs, quadratic algebras, Clifford algebras, and their modules over a smooth complex affine algebraic variety $\boldsymbol{X}$. 

Working over an affine algebraic variety allows us to replace most families by their global sections, enabling us to treat them as modules over a ring rather than a sheaf of modules. 
Moreover, we introduce the deformation-like families $(\g_{(n)}, K)$ of Harish-Chandra pairs associated with a real reductive group $G(\R)$ (see \ref{vdf}). For these families, we describe the Harish-Chandra homomorphism and we recall the notion of  infinitesimal character with respect to a $\theta$-stable Cartan subfamily.

In Section \ref{sec4}, we focus on families over a complex affine algebraic variety $\boldsymbol{X}$, whose coordinate ring $O_{\boldsymbol{X}}$ is a principal ideal domain containing $\mathbb{C}$. We assume that the family of Lie algebras $\g$ is a subfamily of the constant family $O_{\boldsymbol{X}} \otimes_{\mathbb{C}} \fg$ (where $\fg = \mathbb{C} \otimes_{\mathbb{R}} \operatorname{Lie}(G(\R))$) and that $\g$ is stable under the Cartan involution of $O_{\boldsymbol{X}} \otimes_{\mathbb{C}} \fg$. In particular,  $\g$ admits a decomposition into  eigenspaces of  $\theta$:
\begin{equation*}
\g = \k \oplus \p.
\end{equation*}
We further assume that $\p=\langle r\rangle\otimes_{\C}\fp$ for some $r\in O_{\boldsymbol{X}}$.
This class of families includes the deformation-like families mentioned above, as well as the classical case  arising from a real reductive group,  which can be considered as a family over the one point complex affine algebraic variety.

For the sake of simplicity, for the rest of this section  we shall assume that $G(\R)$ is semisimple.

In the group case, the definition of the algebraic Dirac operator $D_{\fg,K}$ of $\fg$ relies on the non-degeneracy of its Killing form. For a family   $\g$ in our class,  the non-degeneracy of the   Killing form does not suffice to define a Dirac operator. This is because the morphism from $\g$ to its dual $\g^*$ induced by the Killing form is a monomorphism but not an isomorphism. Symmetric bilinear forms that induce an isomorphism between the corresponding space and its dual are called \emph{unital}.

   We prove that for $\g$ as above, the Killing form restricted to $\p$ is proportional to a unital form $\boldsymbol{\beta}$. Using $\boldsymbol{\beta}$, we define a canonical Dirac operator $D(\g,\beta,r)$ for $\g$.
Moreover, we show that $D(\g,\beta,r)^2$ satisfies an important polynomial identity that relates it to the Casimir operator of $\g$, generalizing the known identity in the classical case of a real reductive group, while we define the Dirac cohomology for families of $(\g, K)$-modules.

In Section \ref{Vo}, we formulate and prove a generalization of Vogan's conjecture for constant families and for the deformation-like families $(\g_{(n)}, K)$. 

 In  \cite{eyallietheory}, the notion of an infinitesimal character with respect to a fundamental Cartan subfamily for modules of the deformation family  was introduced. We prove that a generically irreducible family of $(\g_{(n)}, K)$-module with non-zero Dirac cohomology must have an infinitesimal character with respect to a fundamental Cartan subfamily. This result resonates with the conjectural algebraic characterization for the Mackey-Higson bijection that was given in  \cite{eyallietheory}.

 Finally, in Section \ref{SL}, for the example of $SL_2(\mathbb{R})$, 
 we explicitly compute the Dirac cohomology of any generically irreducible family of Harish-Chandra modules for the deformation family  and we verify Vogan's conjecture.

\textbf{Acknowledgement:} 
The authors thank  Dan Barbasch  for a valuable advice. 
The helpful  suggestions of the anonymous referee are gratefully acknowledged.

This research was supported by the Israel Science Foundation (grant No. 1040/22).
 \section{The Dirac operator of a real reductive group}\label{Spinreprese} 
 In this section we recall the setup needed for describing Vogan's conjecture for Harish-Chandra modules of a connected real reductive Lie group $G(\R)$. The results stated here shall be used throughout the paper. Further information and proofs can be found in \cite{goodman,pandzic}.

\subsubsection{The reductive groups}
We let $G(\R)$ be a connected real reductive group in the sense of  \cite{MR4146144}. In particular $G(\R)$ is the fixed point set of an antiholomorphic involution $\sigma$ of a corresponding complex connected reductive algebraic group $G$. 
There is a corresponding Cartan involution $\theta$ of $G$ whose fixed point set $K$ is a complex reductive algebraic group $K$ and $K(\R):=K^{\sigma}$ is a maximal compact subgroup of $G(\R)$. 
\subsubsection{The Lie algebras}
We let $\mathfrak{g}:=\operatorname{Lie}(G)$, $\mathfrak{k}:=\operatorname{Lie}(K)$ be the complex Lie algebras of $G$ and $K$ respectively.  Similarly,  $\mathfrak{g}^{\sigma}:=\operatorname{Lie}(G^{\sigma})$, $\mathfrak{k}^{\sigma}:=\operatorname{Lie}(K^{\sigma})$ are the real Lie algebras of $G^{\sigma}$ and $K^{\sigma}$ respectively. The Cartan decompositions 
\[\mathfrak{g}=\mathfrak{k}\oplus \mathfrak{p}, \quad \mathfrak{g}^{\sigma}=\mathfrak{k}^{\sigma}\oplus \mathfrak{p}^{\sigma},\]
are the eigenspace decompositions   of $\theta$ with $\fk=\fg^{\theta}$ and $\fp=\fg^{-\theta}$.

We let 
\[\fg=Z_{\fg}\oplus \fg'\] be the decomposition of $\fg$ to the sum of its center and its derived Lie algebra $ \fg'=[\fg,\fg]$.

\subsubsection{The invariant symmetric form}\label{invform}
We let $\beta$ be a  symmetric bilinear form on $\fg$ that is $K$-invariant,  $\fg$-invariant and which is positive definite on $\fp^{\sigma}$ and negative definite on $\fk^{\sigma}$ (and in particular also non-degenerate on $\fg$). Such a form always exists and the choice of such  does not play a significant role in what follows. To make  a   canonical choice  we shall assume that the form $\beta$ upon restriction to its derived semi-simple subalgebra $\fg'$ is given by the Killing form of $\fg'$. 

\subsubsection{The Clifford algebra}
We let $\beta|_{\fp}$ stand for the restriction of $\beta$ to $\fp\times\fp$. Abusing notation, when there is no chance for ambiguity, we   shall  occasionally  let  $\beta$ also stand for  its restriction  $\beta|_{\fp}$.
 We denote  the Clifford algebra of $(\mathfrak{p},\beta|_{\fp})$ by $Cl(\mathfrak{p},\beta|_{\fp})$ 
 and we realize it as  the quotient of the tensor algebra $T(\mathfrak{p})$ by the two-sided ideal generated by all elements of the form
\begin{equation*}
	X\otimes Y+Y\otimes X-\beta (X,Y), \quad X,Y\in\mathfrak{p}.
\end{equation*} 
We denote   the canonical embedding of $\mathfrak{p}$ into $Cl(\mathfrak{p},\beta|_{\fp})$ by $\gamma=\gamma_{\fp,\beta}$. 

\subsubsection{The mapping of $\fk$ into the Clifford algebra}
For $a$ and $b$ in $\mathfrak{p}$, we let $R_{\beta,a,b}$ stands for  the element of 
\begin{equation*}
	\mathfrak{so}(\mathfrak{p},\beta):=\{T\in \mathfrak{gl}(\mathfrak{p})\mid \beta(Tu,v)+\beta(u,Tv)=0,\forall u,v\in \mathfrak{p} \}.
\end{equation*}
defined by 
\begin{equation}\label{Rmap}
	R_{\beta,a,b}(v):=\beta(b,v) a-\beta(a,v) b, \quad v\in \mathfrak{p}.
\end{equation}
The set $\{R_{\beta,a,b}\mid a,b\in\mathfrak{p}\}$ spans $\mathfrak{so}(\mathfrak{p},\beta)$ \cite[Lemma 6.2.1]{goodman} and there is a Lie algebra monomorphism $\varphi_{\beta}:\mathfrak{so}(\mathfrak{p},\beta)\rightarrow Cl(\mathfrak{p},\beta)$ given by
\begin{equation*}
	R_{\beta,a,b}\mapsto \frac{1}{2}[\gamma(a),\gamma(b)].
\end{equation*} 
The adjoint action of $\fk$ on $\fp$ and the $\fg$-invariance of $\beta$ give rise to  a Lie algebra homomorphism 
$ad:\mathfrak{k}\longrightarrow \mathfrak{so}(\mathfrak{p},\beta)$.
We let $\alpha=\alpha_{\beta}:\fk\longrightarrow Cl(\mathfrak{p},\beta)$ be the morphism of Lie algebras given by the composition
\begin{equation}
\mathfrak{k}\overset{\mathrm{ad}}{\longrightarrow} \mathfrak{so}(\mathfrak{p},\beta)\overset{\varphi_{\beta}}{\longrightarrow} Cl(\mathfrak{p},\beta).
\end{equation} 

\subsubsection{Canonical elements and Casimirs}\label{CanEl} 
For every linear subspace $\mathfrak{l}$ of $\fg$ on which the restriction of $\beta$ is non-degenerate we obtain a canonical element $\omega(\mathfrak{l},\beta)$ in $\mathfrak{l}\otimes_{\C} \mathfrak{l}\subseteq \mathfrak{g}\otimes_{\C} \mathfrak{g}$ as follows.   
There are natural linear isomorphisms 
\begin{equation}\label{Endembedd1}
	\mathrm{End}(\mathfrak{l})\cong\mathfrak{l}^*\otimes_{\C}\mathfrak{l}\cong\mathfrak{l}\otimes_{\C}\mathfrak{l},
\end{equation}
where the second equivalence is given by the isomorphism $\mathfrak{l}^*\cong \mathfrak{l}$ determined by $\beta|_{\mathfrak{l}}$.
We let $\omega(\mathfrak{l},\beta)$  be the element in $\mathfrak{l}\otimes_{\C}\mathfrak{l}$ that is obtained as the image of  the identity operator $\mathbb{I}_\mathfrak{l}\in \mathrm{End}(\mathfrak{l})$  via the abovementioned isomorphism. 
  
If $\{e_1,e_2,...,e_m\}$ is a basis for $\mathfrak{l}$ with a $\beta$-dual basis $\{e'_1,e'_2,...,e'_m\}$, that is a basis satisfying $\beta(e_i,e'_j)=\delta_{ij}$, the canonical element satisfies 
\[\omega(\mathfrak{l},\beta)=\sum_{i}^m e'_i\otimes e_i=\sum_{i,j=1}^m\beta (e'_j,e'_i) e_j\otimes e_i. \]
We denote the image of $\omega(\mathfrak{l},\beta)$ in $\mathcal{U}(\mathfrak{g})$
under  the obvious morphisms 
  \[\mathfrak{l}\otimes_{\C}\mathfrak{l}\longrightarrow T(\mathfrak{l})\longrightarrow T(\mathfrak{g})\longrightarrow \mathcal{U}(\mathfrak{g}),\]   by $\Omega(\mathfrak{l},\beta)$. When $\beta$ is 
  $\fg$-invariant,   $\Omega(\mathfrak{g},\beta)$ lies in $\mathcal{Z}(\mathfrak{g})$ (the center of the universal enveloping algebra of $\mathfrak{g}$). If in addition   $\mathfrak{g}$ is a semi-simple Lie algebra and $\beta$ its Killing form,  $\Omega(\mathfrak{g},\beta)$ is equal to the Casimir of $\fg$. 

\subsubsection{The generalized pair}\label{genpair}
The group $K$  acts on the algebra $A=A(\fg,\beta):=U(\mathfrak{g})\otimes_{\C} Cl(\mathfrak{p},\beta)$ via the adjoint action on each factor. The derivative of this action  induces a map from $U(\fk)$ into $\operatorname{End}(A(\fg,\beta))$  given by 
\begin{equation*}
X\longmapsto [\Delta_{\beta}(X),\_], \quad X\in\mathfrak{k}.
\end{equation*}
Here $\Delta_{\beta}:\mathfrak{k}\longrightarrow A$ stands for the embedding of Lie algebras, known as \textit{the diagonal embedding}, given by 
\begin{equation}\label{demb}
	 \Delta_{\beta}(X)= X\otimes1+1\otimes \alpha_{\beta}(X), \hspace{1mm}\forall  X\in\fk.
\end{equation}
 The pair $(A(\fg,\beta),{K})$ together with the diagonal embedding $ \Delta_{\beta}$ and the action of ${K}$ on $A(\fg,\beta)$ is a generalized pair in the sense of \cite[Ch. I.6]{KNV}.
 We shall denote by $\fk_{\Delta}$ the image of $\fk$ in $A$ under the diagonal embedding. Similarly we shall denote the isomorphic copy of $\mathcal{U}(\fk)$ in $A$ by  $\mathcal{U}(\fk_{\Delta})$. For   later use we introduce the notation 
 $\Omega(\mathfrak{k}_{\Delta},\beta)$ for the image of $\Omega(\mathfrak{k},\beta)$ in $A$ under the diagonal embedding.

\subsubsection{The Dirac operator}
The tensor product of the canonical maps   $\fp\longrightarrow T(\fp)\longrightarrow T(\fg)\longrightarrow \mathcal{U}(\mathfrak{g})$ and  $\fp\longrightarrow T(\fp)\longrightarrow Cl(\mathfrak{p},\beta)$ gives rise to a linear $K$-equivariant embedding  
\begin{equation}\label{Endembedd}
\mathfrak{p}\otimes_{\C}\mathfrak{p}\hookrightarrow U(\mathfrak{g})\otimes_{\C} Cl(\mathfrak{p},\beta)= A(\fg,\beta),
\end{equation}

\begin{definition*} \textit{The algebraic Dirac operator of} $(\mathfrak{g},\beta)$ is the element $D=D_{\fg,\beta}=D_{\beta}$ of $A(\fg,\beta)$ which is  the image of the canonical element $\omega(\mathfrak{p},\beta)$  via the embedding \eqref{Endembedd}.
\end{definition*}
If $\{e_i\}$ and $\{e_i'\}$ are $\beta$-dual bases of $\mathfrak{p}$,   then
\begin{equation*}
D=\sum_{i} e_i\otimes\gamma_{\fp,\beta}(e'_i)=\sum_{i,j}\beta(e_i',e_j') e_i\otimes\gamma_{\fp,\beta}(e_j).
\end{equation*}

\subsubsection{Root space decompositions}
Fix a fundamental Cartan subalgebra $\mathfrak{h}$ of $\fg$. Then $\mathfrak{h}=\mathfrak{t}\oplus \mathfrak{a}$ with $\mathfrak{t}$ being a Cartan subalgebra of $ \mathfrak{k}$ and $\mathfrak{a}\subseteq \mathfrak{p}$. 
Let $\Delta(\fg,\fh)$ be the corresponding set of roots. Fix a positive system $\Delta^+(\fg,\fh)$ and let $\rho_{\fg}$  be the  half sum of the  roots in $\Delta^+(\fg,\fh)$. The restriction $\beta|_{\fh}$  of $\beta$ to $\mathfrak{h}$ induces a non-degenerate form $\beta|_{\fh}^*$ on the dual $\fh^*$. It turns out that $\lVert\rho_{\fg}\rVert^2:=\beta|_{\fh}^*(\rho_{\fg},\rho_{\fg})$ does not depend on the form $\beta$ (as long as its restriction to $\fg'$ is given by the Killing form) nor the choice of the positive root  system. 

Similarly  we denote by $\Delta(\fk,\ft)$ the set of roots of $\fk$ with respect to $\ft$, we fix a positive root system $\Delta^+(\fk,\ft)$
and we define $\lVert\rho_{\fk}\rVert^2:=\beta|_{\ft}^*(\rho_{\fk},\rho_{\fk})$.

We choose  the set of positive roots   $\Delta^+(\fg,\fh)$ to be $\theta$-stable. Hence   the corresponding  nilpotent Lie algebras $\mathfrak{n}^{\pm}$ are $\theta$-stable and we can assume that  
\[\mathfrak{n}_{\fk}^{\pm}:= \bigoplus_{\alpha\in \Delta^+(\fk,\ft)}\fk_{\pm \alpha }\quad \subseteq \quad \mathfrak{n}^{\pm}:= \bigoplus_{\alpha\in \Delta^+(\fg,\fh)}\fg_{\pm \alpha },    \]
where $\fk_{\alpha }$ and $\fg_{ \alpha }$ are the  root subspaces of $\fk$ and $\fg$  respectively.

\subsubsection{The square of the Dirac operator}
The square of the Dirac operator takes a simple form that is explicitly given by 
 \begin{equation}\label{diracsquare}	2D_{\fg,\beta}^2=\Omega(\mathfrak{g},\beta)\otimes 1-\Omega(\mathfrak{k}_{\Delta},\beta|_{\fk})+\lVert\rho_{\fg}\rVert^2-\lVert\rho_{\mathfrak{k}}\rVert^2.
\end{equation}
For a proof, e.g.  see \cite[Prop. 3.1.6]{pandzic}.

\subsubsection{The spin module}\label{spinmodule}
By definition a \textit{spin module} (or a \textit{space of spinors}) for the Clifford algebra $Cl(\mathfrak{p},\beta|_{\fp})$ is a simple module. When $\operatorname{dim}(\fp)$ is even such a module is unique up to isomorphism and when $\operatorname{dim}(\fp)$ is odd there are exactly two different spin modules up to isomorphism \cite[Theorem 6.1.1]{goodman}.

We shall now present realizations of these spin modules. 


When  $\operatorname{dim}(\fp)$ is even, 
$\mathfrak{p}=\mathfrak{p}^+\oplus\mathfrak{p}^-$ with $\mathfrak{p}^+$ and $\mathfrak{p}^-$ being dual maximal isotropic subspaces. We realize the spin module $S=S(\fp,\beta|_{\fp},\fp^-)$   as the exterior algebra  $\bigwedge \mathfrak{p}^-$. 
The action of the Clifford algebra on $S$ is completely determined by the   action of $\gamma(\mathfrak{p}^+),\gamma(\mathfrak{p}^-)\subset Cl(\mathfrak{p},\beta|_{\fp})$. 
For $X\in \mathfrak{p}^-$ and $Y\in S$, we have
\[\gamma(X)\cdot Y:=X\wedge Y.\]
For $X\in \mathfrak{p}^+$,  $\gamma(X)$ acts on $S$ as the unique  graded derivation of degree $-1$ satisfying 
\[\gamma(X)\cdot Y:=\beta(X,Y), \quad \forall Y\in \mathfrak{p}^-\subset\bigwedge \mathfrak{p}^-. \] 
We denote the corresponding action map by 
\begin{equation}\nonumber
\gamma_{\mathfrak{p}^-,\beta}':Cl(\mathfrak{p},\beta)\longrightarrow \mathrm{End}(S(\fp,\beta,\fp^-)).
\end{equation}

When  $\operatorname{dim}(\fp)$ is odd, 
$\mathfrak{p}=\mathfrak{p}^+\oplus\mathfrak{p}^-\oplus \mathfrak{p}^0$ with $\mathfrak{p}^+$ and $\mathfrak{p}^-$ dual maximal isotropic subspaces,  and $\mathfrak{p}^0$ a one-dimensional subspace of $\fp$ orthogonal to $\mathfrak{p}^+\oplus\mathfrak{p}^-$  on which $\beta$ is nondegenerate.  In this case the spin module can still be realized on the exterior algebra $\bigwedge \mathfrak{p}^-$ with the same action of $\gamma(\fp^+\oplus\fp^-)$, as in the even-dimensional case, and choosing a   vector $e_0\in \fp_0$ satisfying $\beta(e_0,e_0)=2$, and $\epsilon\in \{-1,1\}$, the action of $\gamma(e_0)$ is via multiplication by $\epsilon $ on the even elements of $\bigwedge \mathfrak{p}^-$, and multiplication by $-\epsilon $ on the odd elements of $\bigwedge \mathfrak{p}^-$.  
The two possibilities for $\epsilon$  give non-isomorphic spin modules.

\subsubsection{The double cover}\label{dc}
Apart from the $Cl(\mathfrak{p},\beta)$-module structure, $S$ is equipped with the $\mathfrak{k}$-action  defined as the composition
\begin{equation}\label{haction}
	\mathfrak{k}\overset{\mathrm{ad}}{\longrightarrow} \mathfrak{so}(\mathfrak{p},\beta)\overset{\varphi_{\beta}}{\longrightarrow} Cl(\mathfrak{p},\beta)\overset{\gamma_{\fp^-,\beta}'}{\longrightarrow} \mathrm{End}(S(\fp,\beta,\fp^-)).
\end{equation} 
 We would like to lift this action to a group action of $K$. Although this can not always be done, the action can be lifted to the spin  double   cover of $K$ which we shall now describe. 

Recall (e.g. see \cite[Sec. 3.2.1] {pandzic}) that \textit{the spin double cover}  $\widetilde{K}(\R)=\widetilde{K}(\R,\beta)$ of $K(\R)$ is defined as the pullback of the cover $\pi:\operatorname{Spin}(\fp^{\sigma},\beta)\longrightarrow SO(\fp^{\sigma},\beta)$ with respect to $\operatorname{Ad}:K(\R)\longrightarrow SO(\fp^{\sigma},\beta)$. More precisely, 
\begin{equation*}\widetilde{K}(\R)=\{(k,s)\in K(\R)\times \operatorname{Spin}(\fp^{\sigma},\beta)|\operatorname{Ad}(k)=\pi(s) \}
\end{equation*}
and the natural projection from $\widetilde{K}(\R)$ onto $K(\R)$ is a double cover.
The projection  from  $\widetilde{K}(\R)$
into $\operatorname{Spin}(\fp^{\sigma},\beta)$ followed by the action of $\operatorname{Spin}(\fp^{\sigma},\beta)$ on $S$ defines an action of $\widetilde{K}(\R)$ on $S$. 
Let $\widetilde{K}=\widetilde{K}(\beta)$ be the complex reductive algebraic group obtained as the complexification of $\widetilde{K}(\R)$. Since the action of $\widetilde{K}(\R)$ on $S$ defines a continuous finite-dimesional representation, we obtain an algebraic action of  $\widetilde{K}$  on $S$. We denote this representation by $\pi_{\widetilde{K},S}:\widetilde{K}\longrightarrow \operatorname{Aut}_{\C}(S)$. The differential of this action coincides with $\gamma_{\fp^-,\beta}' \circ \varphi_{\beta}\circ \operatorname{ad}$ .   

Clearly  the pair $(A(\fg,\beta),\widetilde{K})$ together with the diagonal embedding $\Delta_{\beta}$ and the action of $\widetilde{K}$ on $A(\fg,\beta)$ is a generalized pair. 

\subsubsection{Dirac cohomology}
For  a $(\mathfrak{g},K)$-module $V$, $V\otimes_{\C}S$ is an $(A(\fg,\beta),\widetilde{K})$-module. In particular the Dirac operator acts on  $V\otimes_{\C}S$ via a linear operator that we denote by $D(V)=D_{\fg,\beta}(V)=D_{\beta}(V)$. Since the operator $D_{}(V)$ is $\widetilde{K}$-equivariant, its kernel $\ker D(V)$  and its image $\mathrm{im}\hspace{0.5mm}D(V)$ are $\widetilde{K}$-modules.

\begin{definition*} Let $V$ be a $(\mathfrak{g},K)$-module. The \textit{Dirac cohomology} of $V$ is the $\widetilde{K}$-module 
	\begin{equation*}
		H_D(V):=\frac{\ker D(V)}{\ker D(V)\cap \mathrm{im}\hspace{0.5mm}D(V)}.
	\end{equation*}
\end{definition*}

\subsubsection{The Harish-Chandra isomorphism}\label{The HC iso}
 We let $\operatorname{HC}_{\fh}^{\mathfrak{g}}$ be the canonical Harish-Chandra isomorphism with respect to the triangular decomposition $\fg=\mathfrak{n}^{-}\oplus \fh \oplus \mathfrak{n}^{+}$. 
 Its domain is 
  $\mathcal{Z}({\mathfrak{g}})$ and its image is  $\operatorname{S}({\mathfrak{h}})^{W(\fg,\fh)}$, the Weyl invariants in the symmetric algebra of $\mathfrak{h}$. 
 Similarly we let $\operatorname{HC}_{\ft}^{\mathfrak{k}}:\mathcal{Z}({\mathfrak{k}})\longrightarrow \operatorname{S}({\mathfrak{t}})^{{W(\fk,\ft)}}$ be the canonical Harish-Chandra isomorphism for $\mathfrak{k}$ with respect to the triangular decomposition $\fk=\mathfrak{n}_{\fk}^{-}\oplus \ft \oplus \mathfrak{n}_{\fk}^{+}$. Note that we have a canonical embedding of the Weyl group of $\mathfrak{k}$ into the Weyl group of $\mathfrak{g}$.

Recall that
a $(\fg,K)$-module $X$ has an infinitesimal character $\lambda\in \fh^*$  if for every  element $\xi \in \mathcal{Z}(\fg)$, the action of $\xi$ on $X$ is given by multiplication by $\lambda(HC^{\fg}_{\fh}(\xi))$. Here the symmetric algebra $S(\fh)$
 is identified with the algebra of polynomial functions on $\fh^*$.

\subsubsection{Vogan's conjecture}\label{VC}

The spin double cover $\widetilde{K}(\R)$ is a compact  Lie group and we can use Cartan-Weyl theory of highest weights for  disconnected groups (see \cite[Theorem 4.25]{KNV}) to parameterize its finite-dimensional irreducible representations. These representations   can be identified with irreducible finite dimensional algebraic representations of $\widetilde{K}$. 

A representation  
 $X$  of $\widetilde{K}$  contains  a $\widetilde{K}$-type  of highest weight $\mu\in \ft^*$ 
(with respect to our chosen positive system) if there is a nonzero vector $v\in X$ which is of weight $\mu$ with respect to $\ft$ and annihilated by $\mathfrak{n}_{\fk}^{+}$. See section 5.1 in \cite{MR632407}.

Recall that using the nondegenerate form $\beta$ we embed $\ft^*$ in $\fh^*$. That is, a linear functional $\mu \in \ft^*$ is  made into a linear functional $\mu^{\ft}_{\fh}\in \fh^*$ by setting 
\[\mu^{\ft}_{\fh}(X)=\begin{cases}
    \mu(X),& X\in \ft\\
    0,& X\in \ft^{\perp}=\mathfrak{a}
\end{cases}\]

\begin{theorem*}[Vogan's Conjecture {\cite{huangpandzic,zbMATH06420484}}]\label{vogansconj}\index{Vogan's conjecture}
	Let $V$ be an irreducible $(\mathfrak{g},K)$-module. Assume that $H_D(V)$ contains a $\widetilde{K}$-type of highest weight $\mu\in\mathfrak{t}^*$. Then, the infinitesimal character $\lambda\in \fh^*$ of  $X$ is $W(\fg,\fh)$-conjugate to $(\mu+\rho_\mathfrak{k})^{\ft}_{\fh}$.
\end{theorem*}

\section{Algebraic Families}\label{algebraicf}

In this section we recall some basic facts and definitions from the theory of algebraic families of Harish-Chandra pairs and their modules as given in \cite{eyalbern}. We shall work in the special case in which the base variety $\boldsymbol{X}$ is either a point $\boldsymbol{\operatorname{pt}}$,  $\mathbb{C}$ or $\mathbb{C}^{\times}$ and the families of groups are constant families with a reductive fiber. In this setup we can and we will replace a family by its global sections and  use modules over rings instead of sheaves of modules.  This shall be done throughout the paper and allows us to use simpler definitions than those that are given in \cite{eyalbern}.

We shall discuss the notion of  algebraic families  of quadratic spaces and quadratic Harish-Chandra pairs. 

We recall the construction of  the deformation family of Harish-Chandra pairs associated with a real reductive Lie group and some of its relatives. 

For deformation-like families we define the Harish-Chandra homomorphism. For modules of  deformation-like families we define and study the notion of an  infinitesimal character with respect to a fundamental Cartan
subfamily. 

Throughout this section we denote the  sheaf of regular functions on $\boldsymbol{X}$ by 
$\mathcal{O}_{\boldsymbol{X}}$ and its ring of global sections by $R_{\boldsymbol{X}}:=\mathcal{O}_{\boldsymbol{X}}(\boldsymbol{X})$. By means of choosing coordinates we  identify $R_{\boldsymbol{X}}$ with $\mathbb{C}[t]$ in the case of $\boldsymbol{X}=\mathbb{C}$,  with $\mathbb{C}[t,t^{-1}]$ in the case of $\boldsymbol{X}=\mathbb{C}^{\times}$, and with $\mathbb{C}$ in the case of $\boldsymbol{X}=\boldsymbol{\operatorname{pt}}$. . 
\subsection{Families of Lie algebras and Harish-Chandra pairs}
In this subsection we shall introduce the definitions and give examples for  algebraic families of various objects including  Lie algebras,  Harish-Chandra pairs, generalized pairs, quadratic spaces and Clifford algebras.

\subsubsection{Families of Lie algebras}
In this paper an \textit{algebraic family of complex Lie algebras $\boldsymbol{\mathfrak{g}}$ over $\boldsymbol{X}$} is a free $R_{\boldsymbol{X}}$-module  that is also a Lie algebra over $R_{\boldsymbol{X}}$.  When $\boldsymbol{X}$ is $\C$ or $\C^{\times}$,
for every $z\in \boldsymbol{X}$, the fiber of 
$\boldsymbol{\mathfrak{g}}$ over $z$ is the complex Lie algebra
\[\boldsymbol{\mathfrak{g}}|_{z}:=\boldsymbol{\mathfrak{g}}/(I_{z}\boldsymbol{\mathfrak{g}}), \]
where $I_{z}$ is the maximal ideal of $R_{\boldsymbol{X}}$ generated by $t-z$. In that way, an algebraic family of Lie algebras over $\boldsymbol{X}$ gives rise to a collection of complex Lie algebras parameterized by $\boldsymbol{X}$. When $\boldsymbol{X}=\boldsymbol{\operatorname{pt}}$, $\g|_{\boldsymbol{\operatorname{pt}}}:=\g$ is a single complex Lie algebra. 

\subsubsection{Families of Harish-Chandra pairs}

An \textit{algebraic family of Harish-Chandra pairs over ${\boldsymbol{X}}$} is a pair $(\boldsymbol{\mathfrak{g}},K)$ with $\boldsymbol{\mathfrak{g}}$ being an algebraic family of complex Lie algebras over $\boldsymbol{X}$ and $K$ being a complex algebraic group acting on $\boldsymbol{\mathfrak{g}}$ by automorphisms of Lie algebras over $R_{\boldsymbol{X}}$ 
\begin{equation*}\pi_{K,\g}:K\longrightarrow \operatorname{Aut}_{R_{\boldsymbol{X}}}(\boldsymbol{\mathfrak{g}}),
\end{equation*}
together with a $K$-equivariant embedding of complex Lie algebras \begin{equation*}j:\operatorname{Lie}(K)\longrightarrow \boldsymbol{\mathfrak{g}},
\end{equation*}
from  the Lie algebra of $K$ into  $\boldsymbol{\mathfrak{g}}$. Here the action of $K$ on its Lie algebra is the adjoint action. It is further required that the two  actions   of $\operatorname{Lie}(K)$ on $\boldsymbol{\mathfrak{g}}$ coincide in the sense that
\begin{equation*}\operatorname{ad}_{j(Y)}(Z)=d\pi_{\fk,\g}(Y)(Z), \quad \forall Y\in \operatorname{Lie}(K),Z\in \boldsymbol{\mathfrak{g}}.
\end{equation*}
Here $d\pi_{\fk,\g}$ stands for the action of $\operatorname{Lie}(K)$ that is obtained from the action $\pi_{K,\g}$ of $K$ on $\mathfrak{g}$. 

\subsubsection{Families of generalized pairs}
An \textit{algebraic family of complex (associative) algebras over   $\boldsymbol{X}$}
is a free $R_{\boldsymbol{X}}$-module  that is also an algebra over $R_{\boldsymbol{X}}$.

Following the notion of a \textit{pair} (or a \textit{generalized pair}) see \cite[Ch. I.6]{KNV},  an  \textit{algebraic family of  generalized   (Harish-Chandra) pairs }
 over $\boldsymbol{X}$ is a pair $(\boldsymbol{A},K)$, with $\boldsymbol{A}$ being an algebraic family of complex algebras over $\boldsymbol{X}$, $K$  being a complex algebraic group,  that satisfy analogous  conditions to the case of  an algebraic    family of Harish-Chandra pairs.  
\begin{example*}
Let $(\boldsymbol{\mathfrak{g}},K)$ be an algebraic family of Harish-Chandra pairs over ${\boldsymbol{X}}$. Then $\mathcal{U}(\g)$, the universal enveloping algebra of $\g$, is an algebraic family of complex algebras over ${\boldsymbol{X}}$,
 and $(\mathcal{U}(\g),K)$ is
 an  algebraic family of  generalized     pairs over ${\boldsymbol{X}}$.    
\end{example*}

\subsubsection{Families of  quadratic spaces}

\begin{definition*}
An algebraic  family of  quadratic spaces over $\boldsymbol{X}$ is a free $R_{\boldsymbol{X}}$-module $M$ together with  a quadratic form $q:M\longrightarrow R_{\boldsymbol{X}}$, that is, $q$  satisfies the following  two properties: 
\begin{enumerate}
    \item $q(rm)=r^2q(m)$, $\forall r\in R_{\boldsymbol{X}}$ and $m\in M$.
    \item The map $\beta=\beta_{q}:M\times M\longrightarrow R_{\boldsymbol{X}}$ given by \\ $\beta(m_1,m_2)=\frac{1}{2}\left(q(m_1+m_2)-q(m_1)-q(m_2) \right), \quad \forall m_1,m_2\in M,$\\ is $R_{\boldsymbol{X}}$-bilinear. 
\end{enumerate}
\end{definition*}
In our setting $2\in R_{\boldsymbol{X}}$ is invertible and the correspondence $q\longmapsto \beta_q$ is a bijection between quadratic forms and symmetric bilinear forms. We shall mainly use the symmetric form $\beta$ and refer to $(M,\beta)$ (rather than to $(M,q)$) as an algebraic family of quadratic spaces.

Given an algebraic family of quadratic spaces $(M,\beta)$, there is an induced morphism of $R_{\boldsymbol{X}}$-modules $\check{\beta}$ from $M$ into its dual 
$M^*=\operatorname{Hom}_{R_{\boldsymbol{X}}}(M,R_{\boldsymbol{X}})$, given by
\[\check{\beta}(m_1)(m_2)=\beta(m_1,m_2), \quad \forall m_1,m_2\in M. \]
The form $\beta$ is called \textit{nondegenerate} when $\check{\beta}$ is one-to-one, it is called \textit{unimodular} when $\check{\beta}$ is an isomorphism and it is called \textit{orthogonalizable} if there is a basis for $M$ that is orthonormal with respect to $\beta$. When $R_{\boldsymbol{X}}=\C$ the three properties are equivalent. In general,   orthogonalizability  implies unimodularity which implies nondegeneracy.
\subsubsection{Families of Clifford algebras }
\begin{definition*}
Let $(M,\beta)$ be an algebraic  family of  quadratic spaces over $\boldsymbol{X}$. The corresponding \textit{algebraic family of Clifford algebras over $\boldsymbol{X}$}, is a  unital $R_{\boldsymbol{X}}$-algebra $Cl(M,\beta)$ together with an $R_{\boldsymbol{X}}$-linear embedding $\gamma:M\longrightarrow Cl(M,\beta)$  such that for any 
$R_{\boldsymbol{X}}$-linear map $T:M\longrightarrow A$ into a  unital $R_{\boldsymbol{X}}$-algebra $A$ that satisfies $2T(m)^2=\beta(m,m)$ for any $m\in M$, there exists a unique morphism of unital $R_{\boldsymbol{X}}$-algebras $\widetilde{T}:Cl(M,\beta)\longrightarrow A$ such that $T=\widetilde{T}\circ \gamma $. 
 \end{definition*}
 
We shall always realize $Cl(M,\beta)$ as the quotient of the tensor algebra $T(M)$ by the two-sided ideal $I_{\beta}$  generated by 
\begin{equation*}
	X\otimes Y+Y\otimes X-\beta (X,Y), \quad X,Y\in M.
\end{equation*} 
In this realization   $\gamma$ is the natural composition $ M \longrightarrow  T(M) \longrightarrow  T(M)/I_{\beta}$.
\subsubsection{Families of quadratic Harish-Chandra pairs}
Motivated by the notion of quadratic Lie algebras over fields (see e.g., \cite{emergence}) below we define algebraic families of quadratic   Harish-Chandra pairs.
\begin{definition*}
An \textit{algebraic  family of quadratic   Harish-Chandra pairs  over $\boldsymbol{X}$} is an algebraic family of Harish-Chandra pairs $(\g,K)$ over $\boldsymbol{X}$ together with a symmetric $R_{\boldsymbol{X}}$-bilinear  form $\boldsymbol{\beta}: \g \times \g \longrightarrow  R_{\boldsymbol{X}}$ that is $(\g,K)$-invariant in the sense that
\begin{enumerate}
    \item ($\g$-invariance) For $X,Y,Z\in \g$,
    \[ \boldsymbol{\beta}([X,Y],Z)=\boldsymbol{\beta}(X,[Y,Z]). \]
    \item ($K$-invariance) 
 For $X,Y\in \g$, $k\in K$,
    \[ \boldsymbol{\beta}(\pi_{K,\g}(k)X,\pi_{K,\g}(k)Y)=\boldsymbol{\beta}(X,Y). \]
\end{enumerate}
\end{definition*}

\subsection{The  deformation family of Harish-Chandra pairs and its relatives}
In this subsection we  recall the construction of the deformation family of Harish-Chandra pairs associated with a real reductive group and discuss some natural variants of it.

Throughout this section we keep our assumptions on $G(\R)$ and use the notations for groups and Lie algebras that were introduced in    Section \ref{Spinreprese}.  In particular, the pair $(\mathfrak{g},K)$ is a  Harish-Chandra pair, where, as before, $\fg$ is the complex Lie algebra of $G$ and $K$ the fixed point set of the Cartan involution.

\subsubsection{The constant family}
By definition the constant family of  Lie algebras over $\boldsymbol{X}$ with fiber $\fg$ consists of the free $R_{\boldsymbol{X}}$-module of  all algebraic functions from $\boldsymbol{X}$ into $\fg$. Recall that a function $s:\boldsymbol{X}\longrightarrow \fg$ is algebraic if for every $\varphi\in \fg^*$, $\varphi \circ s \in R_{\boldsymbol{X}}$. The commutator of two sections $s_1,s_2$ is given by 
\[[s_1,s_2](x):=[s_1(x),s_2(x)], \quad \forall x\in \boldsymbol{X}.\] 
This family is canonically identified with $R_{\boldsymbol{X}}\otimes_{\C}\fg$. We shall freely use this identification and denote the constant family by $R_{\boldsymbol{X}}\otimes_{\C}\fg$. 
The action of $K$ on $\fg$ can be extended by  $R_{\boldsymbol{X}}$-linearity to an action on $R_{\boldsymbol{X}}\otimes_{\C}\fg$. 
By letting  $j_0:\operatorname{Lie}(K)\longrightarrow \fg$ stand for the equivariant embedding that is given with $(\fg,K)$,  the map 
$j:\operatorname{Lie}(K)\longrightarrow R_{\boldsymbol{X}}\otimes_\mathbb{C}\mathfrak{g}$ that  is defined by   $j(Z)=1\otimes j_0(Z)$  is an equivariant embedding making 
$(R_{\boldsymbol{X}}\otimes_{\C}\fg,K)$ into an algebraic family of Harish-Chandra pairs over $\boldsymbol{X}$.  We call this family \textit{the constant family of Harish-Chandra pairs over $\boldsymbol{X}$ with fiber $(\mathfrak{g},K)$}.

\subsubsection{The deformation family}
For the case of $\boldsymbol{X}=\C$ we define the subfamily $\g_d$ of the constant family of Lie algebras over $\boldsymbol{X}$ with fiber $\fg$ to be the collection of all sections  whose value at $0\in \boldsymbol{X}$ belongs to $\fk$.  By identifying $R_{\boldsymbol{X}}$ with $\C[t]$ we see that  
\[\boldsymbol{\mathfrak{g}}_d=(\mathbb{C}[t]\otimes_{\mathbb{C}}\mathfrak{k})\oplus (t\mathbb{C}[t]\otimes_{\mathbb{C}}\mathfrak{p}).\] 
The pair $(\boldsymbol{\mathfrak{g}}_d,K)$ forms an algebraic family of Harish-Chandra pairs over $\C$ that is a subfamily of the constant family. This is   \textit{the deformation family of (Harish-Chandra) pairs associated with $(\fg,K)$ (or with $G(\R)$)}.
 \subsubsection{Variants of the deformation family}\label{vdf}
Note that an algebraic  section $s:\C\longrightarrow \fg$  of the constant family $\C[t]\otimes_{\C}\fg$,  has a well defined derivative of any order. For any $n\in \mathbb{N}$, we  define $\g_{(n)}$ to be the subfamily of the constant family consisting of all sections $s$, such that the values of all their  derivatives at $0$  up to order  $(n-1)$ belong to $\fk$. Clearly 
\[\boldsymbol{\mathfrak{g}}_{(n)}=(\mathbb{C}[t]\otimes_{\mathbb{C}}\mathfrak{k})\oplus (t^{n}\mathbb{C}[t]\otimes_{\mathbb{C}}\mathfrak{p}),\] 
and in particular $\g_{(1)}=\g_d$. We define
\[\boldsymbol{\mathfrak{g}}_{(0)}:=(\mathbb{C}[t]\otimes_{\mathbb{C}}\mathfrak{k})\oplus (\mathbb{C}[t]\otimes_{\mathbb{C}}\mathfrak{p})=\C[t]\otimes_{\C}\fg.\] 
For any $n\in \N_0$, the pair $(\g_{(n)},K)$ is a subfamily of the constant family of Harish-Chandra pairs over $\boldsymbol{X}=\C$. 
For every $n\in \mathbb{N}$,
\[\boldsymbol{\mathfrak{g}}_{(n)}|_{z}\cong \begin{cases}
   \fg  &  z\neq 0\\
   \fk\ltimes \fp& z=0.
\end{cases} \]
We set $\k_{(0)}=\k:=\C[t]\otimes_{\C}\fk$, and for every $n\in \N_0$ we define $\boldsymbol{\fp}_{(n)}:=t^n\C[t]\otimes_{\C}\fp$. Hence
\[\g_{(n)}=\k\oplus \boldsymbol{\fp}_{(n)}\] as a direct sum of $K$-modules over $\C[t]$.

\subsection{Families of  modules}
In this subsection we recall the definitions of algebraic families of modules for families of Lie algebras and families of  Harish-Chandra pairs.
\subsubsection{$\g$-modules}
Let $\boldsymbol{\mathfrak{g}}$ be an algebraic family of complex Lie algebras over $\boldsymbol{X}$. An \textit{algebraic family of $\boldsymbol{\mathfrak{g}}$-modules} is a flat  $R_{\boldsymbol{X}}$-module $\boldsymbol{V}$  together with a morphism of Lie algebras over $R_{\boldsymbol{X}}$
\[\pi_{\g,\boldsymbol{V}}:\boldsymbol{\mathfrak{g}} \longrightarrow \operatorname{End}_{R_{\boldsymbol{X}}}(\boldsymbol{V}). \]
When $\boldsymbol{X}$ is $\C$ or $\C^{\times}$,
for every $z\in \boldsymbol{X}$,  the fiber of 
$\boldsymbol{V}$ over $z$ is the complex vector space 
\[\boldsymbol{V}|_{z}:=\boldsymbol{\boldsymbol{V}}/(I_{z}\boldsymbol{V}). \] 
It  carries a representation of the complex Lie algebra
$\boldsymbol{\mathfrak{g}}|_{z}$. The family    $\boldsymbol{V}$
is called \textit{generically irreducible} if all its fibers, except for at most   countably many, are irreducible. For a definition in a more general setup see \cite{eyalbern}.
It is called \textit{quasi-simple} if $\mathcal{Z}(\boldsymbol{\mathfrak{g}})$, the center of the universal enveloping algebra of $\boldsymbol{\mathfrak{g}}$, acts via multiplication by regular functions. When this is the case, there exists  a homomorphism of commutative $R_{\boldsymbol{X}}$-algebras
\[\chi_{\g,\boldsymbol{V}}:\mathcal{Z}(\boldsymbol{\mathfrak{g}})\longrightarrow R_{\boldsymbol{X}},\]
such that for every $z\in \mathcal{Z}(\boldsymbol{\mathfrak{g}})$ and $v\in\boldsymbol{V}$,
\[\pi_{\g,\boldsymbol{V}}(z)v =\chi_{\g,\boldsymbol{V}}(z)v.\] 
The homomorphism $\chi_{\g,\boldsymbol{V}}$ is called  the \textit{(central) infinitesimal  character} of $\boldsymbol{V}$.

\subsubsection{$(\boldsymbol{\mathfrak{g}},K)$-modules }
Let $(\boldsymbol{\mathfrak{g}},K)$ be an algebraic family of Harish-Chandra pairs over $\boldsymbol{X}$ with $K$ a complex reductive algebraic group.    An \textit{algebraic family of   $(\boldsymbol{\mathfrak{g}},K)$-modules} (or a \textit{$(\boldsymbol{\mathfrak{g}},K)$-module}) is an algebraic family of $\boldsymbol{\mathfrak{g}}$-modules $\boldsymbol{V}$ that carries an algebraic action of $K$ by automorphisms of $R_{\boldsymbol{X}}$-modules, that is, a homomorphism 
\[\pi_{K,\boldsymbol{V}}:K\longrightarrow \operatorname{Aut}_{R_{\boldsymbol{X}}}(\boldsymbol{V}),\]
such that 
\begin{enumerate}
	\item The action morphism $\boldsymbol{\mathfrak{g}}\times \boldsymbol{V}\longrightarrow \boldsymbol{V} $ is $K$-equivariant:
	\[ \pi_{K,\boldsymbol{V}}(k)\left(\pi_{\g,\boldsymbol{V}}(X)(v)\right)=\pi_{\g,\boldsymbol{V}}(\pi_{K,\g}(k)X)(\pi_{K,\boldsymbol{V}}(k)(v)),\quad \forall v\in \boldsymbol{V},k\in K, X\in \boldsymbol{\mathfrak{g}}. \]
	\item The two actions of $\mathfrak{k}$ on $\boldsymbol{V}$ coincide: 
	\[\pi_{\g,\boldsymbol{V}}(j(X))=d\pi_{K,\boldsymbol{V}}(X), \quad \forall X\in \mathfrak{k}, \]
\end{enumerate}
where $d\pi_{K,\boldsymbol{V}}$ stands for the action of $\mathfrak{k}$ that is obtained from the action of $K$ on $\boldsymbol{V}$. 

Such a family is called \textit{admissible} if each $K$-isotypic component  of $\boldsymbol{V}$ is a free $R_{\boldsymbol{X}}$-module of finite rank. An algebraic family of  $(\boldsymbol{\mathfrak{g}},K)$-modules is called \textit{quasi-simple}, if it is quasi-simple as a family of $\boldsymbol{\mathfrak{g}}$-modules. It is called \textit{generically irreducible} if for all   $z\in\boldsymbol{X}$, except for at most countably many, the fiber $\boldsymbol{V}|_z$ is an irreducible $(\boldsymbol{\mathfrak{g}}|_{z},K)$-module.

It follows from Dixmier's Lemma
 that a generically irreducible admissible family of  $(\boldsymbol{\mathfrak{g}},K)$-modules must be quasi-simple \cite[Lemma 4.1.1]{eyalbern}. 

\subsection{Infinitesimal character with respect to a fundamental Cartan subfamily }\label{sec3.4}

In this subsection we lift the Harish-Chandra isomorphism to a morphism of families of commutative algebras with respect to a fundamental   Cartan subalgebra   and define the related notion of   infinitesimal character  in this context of families.

\subsubsection{Harish-Chandra isomorphism for constant families}
As mentioned in Section \ref{The HC iso}, 
we fix $\theta$-stable triangular decompositions 
\[\mathfrak{g}=\mathfrak{n}^-\oplus \mathfrak{h}\oplus \mathfrak{n}^+, \quad \mathfrak{k}=\mathfrak{n}_{\mathfrak{k}}^-\oplus \mathfrak{t}\oplus \mathfrak{n}^+_{\mathfrak{k}}, \]
with $\mathfrak{h}=\mathfrak{t}\oplus \mathfrak{a}$ a fundamental Cartan subalgebra and   $\mathfrak{n}^{\pm}_{\mathfrak{k}}\subseteq \mathfrak{n}^{\pm}$.
We let $\operatorname{HC}_{\fh}^{\mathfrak{g}}:\mathcal{Z}(\mathfrak{g})\longrightarrow \operatorname{S}({\mathfrak{h}})^{W(\fg,\fh)}$ be the Harish-Chandra isomorphism of $\fg$, and let  $\operatorname{HC}_{\ft}^{\mathfrak{k}}:\mathcal{Z}({\mathfrak{k}})\longrightarrow \operatorname{S}({\mathfrak{t}})^{{W(\fk,\ft)}}$ be the  Harish-Chandra isomorphism  for $\mathfrak{k}$. 
By $R_{\boldsymbol{X}}$-linear extension we immediately obtain a Harish-Chandra isomorphism  for the constant families of commutative algebras  over ${\boldsymbol{X}}$:
\begin{eqnarray}\nonumber
&& \boldsymbol{\operatorname{HC}_{\fh}^{\mathfrak{g}}}=\boldsymbol{\operatorname{HC}_{\fh}^{\mathfrak{g}}}/\boldsymbol{X}=\mathbb{I}_{R_{\boldsymbol{X}}}\otimes \operatorname{HC}_{\fh}^{\mathfrak{g}}:\mathcal{Z}(R_{\boldsymbol{X}}\otimes_{\C}\mathfrak{g})\longrightarrow \operatorname{S}(R_{\boldsymbol{X}}\otimes_{\C}{\mathfrak{h}})^{W(\fg,\fh)}, \\ \nonumber
&& \boldsymbol{\operatorname{HC}_{\ft}^{\mathfrak{k}}}=\boldsymbol{\operatorname{HC}_{\ft}^{\mathfrak{k}}}/\boldsymbol{X}=\mathbb{I}_{R_{\boldsymbol{X}}}\otimes \operatorname{HC}_{\ft}^{\mathfrak{k}}:\mathcal{Z}(R_{\boldsymbol{X}}\otimes_{\C}\mathfrak{k})\longrightarrow \operatorname{S}(R_{\boldsymbol{X}}\otimes_{\C}{\mathfrak{t}})^{W(\fk,\ft)},    
\end{eqnarray}
where we use the canonical isomorphisms $R_{\boldsymbol{X}}\otimes_{\C}\mathcal{Z}(\mathfrak{g})\cong \mathcal{Z}(R_{\boldsymbol{X}}\otimes_{\C}\mathfrak{g})$, $R_{\boldsymbol{X}}\otimes_{\C}\operatorname{S}({\mathfrak{h}})^{W(\fg,\fh)}\cong \operatorname{S}(R_{\boldsymbol{X}}\otimes_{\C}{\mathfrak{h}})^{W(\fg,\fh)}$ and similarly for $\fk$. Also note that the actions of the two Weyl groups on the families are given by $R_{\boldsymbol{X}}$-linear extension of their actions on the complex algebras.

\subsubsection{Infinitesimal character with respect to $R_{\boldsymbol{X}}\otimes_{\C}\fh$}
We shall call the family $R_{\boldsymbol{X}}\otimes_{\C}\fh$ a \textit{fundamental Cartan subfamily of $R_{\boldsymbol{X}}\otimes_{\C}\fg$.}
\begin{definition*}
    Let  $\boldsymbol{V}$ be a quasi-simple  algebraic family of   $(R_{\boldsymbol{X}}\otimes_{\C}\fg,K)$-modules with a corresponding character 
    \[\chi_{\boldsymbol{V}}:\mathcal{Z}(R_{\boldsymbol{X}}\otimes_{\C}\fg)\longrightarrow R_{\boldsymbol{X}}.\]
    We say that $\boldsymbol{V}$ \textit{has an infinitesimal character with respect to $R_{\boldsymbol{X}}\otimes_{\C}\fh$} if there is $\boldsymbol{\lambda} \in  (R_{\boldsymbol{X}}\otimes_{\C}\fh)^*$ such that $\chi_{\boldsymbol{V}}=\widehat{\boldsymbol{\lambda}}\circ \boldsymbol{\operatorname{HC}_{\fh}^{\mathfrak{g}}}$, where $\widehat{\boldsymbol{\lambda}}$ is the unique unital  $R_{\boldsymbol{X}}$-algebra homomorphism from  $\operatorname{S}(R_{\boldsymbol{X}}\otimes_{\C}\mathfrak{h})$ into $R_{\boldsymbol{X}}$ such that $\widehat{\boldsymbol{\lambda}}(\boldsymbol{h})=\boldsymbol{\lambda}(\boldsymbol{h})$ for every $\boldsymbol{h}\in R_{{\boldsymbol{X}}}\otimes_{\C}{\mathfrak{h}}$. 
\end{definition*}

\subsubsection{The Harish-Chandra morphism of  $\g_{(n)}$}
Throughout this section we focus on the case with $\boldsymbol{X}=\C$ and $R_{\boldsymbol{X}}=\C[t]$. We shall simplify the notation and write $R$ instead of $R_{\boldsymbol{X}}$. For every complex subspace $\mathfrak{l}$ of $\fg$, we denote  the constant family over $\C$ with fiber $\mathfrak{l}$  by $\boldsymbol{\mathfrak{l}}$. 
For every $n\in \N_0$, we denote by $\boldsymbol{\mathfrak{l}}_{(n)}$  the subspace of $\boldsymbol{\mathfrak{g}}_{(n)}$ consisting of all sections with values in $\mathfrak{l}$.
If $\mathfrak{l}$ is a Lie subalgebra of $\fg$
then $\boldsymbol{\mathfrak{l}}_{(n)}$ is an algebraic family of complex Lie algebras over $\C$. Of course $\boldsymbol{\mathfrak{l}}=\boldsymbol{\mathfrak{l}}_{(0)}$.

We call the family of abelian Lie algebras $\boldsymbol{\mathfrak{h}}_{(n)}$ a \textit{fundamental Cartan subfamily of} $\boldsymbol{\mathfrak{g}}_{(n)}$. Note that if $\mathfrak{l}$ is stable under the action of one of the two mentioned  Weyl groups then so is $\boldsymbol{\mathfrak{l}}_{(n)}$. 

\begin{proposition*}
    For every $n\in \N_0$, $\boldsymbol{\operatorname{HC}_{\fh}^{\mathfrak{g}}}(\mathcal{Z}(\g_{(n)}))\subseteq  \operatorname{S}(\boldsymbol{\mathfrak{h}}_{(n)})^{W(\fg,\fh)}$.
\end{proposition*}

\begin{proof}
Each $\lambda \in \mathfrak{h}^*$ can be extended to $\mathbb{I}_{R}\otimes {\lambda} \in \boldsymbol{\mathfrak{h}}^*=\operatorname{Hom}_{\mathbb{C}[t]}(\boldsymbol{\mathfrak{h}},\mathbb{C}[t])$ that is defined via $(\mathbb{I}_{R}\otimes {\lambda})(f\otimes h)=f\lambda(h)$ for any $f\otimes h\in \boldsymbol{\mathfrak{h}}$. Each $\boldsymbol{\lambda} \in \boldsymbol{\mathfrak{h}}^*$ defines an algebra homomorphism $\widetilde{\boldsymbol{\lambda}}:\operatorname{S}(\boldsymbol{\mathfrak{h}})\longrightarrow \operatorname{S}(\boldsymbol{\mathfrak{h}}) $. It is the unique unital algebra homomorphism satisfying $\widetilde{\boldsymbol{\lambda}}(\boldsymbol{h})=\boldsymbol{h}-\boldsymbol{\lambda}(\boldsymbol{h})1$ for any $\boldsymbol{h}\in \boldsymbol{\mathfrak{h}}$.  
For every $n\in \N_0$, and $\boldsymbol{\lambda} \in \boldsymbol{\mathfrak{h}}^*$, $\widetilde{\boldsymbol{\lambda}}(\boldsymbol{\mathfrak{h}}_{(n)})\subset \boldsymbol{\mathfrak{h}}_{(n)} \oplus \C[t] \subseteq \operatorname{S}(\boldsymbol{\mathfrak{h}}_{(n)}) $ and as a result 
\begin{eqnarray}\label{func}
    &&\widetilde{\boldsymbol{\lambda}}(\operatorname{S}(\boldsymbol{\mathfrak{h}}_{(n)}))\subseteq \operatorname{S}(\boldsymbol{\mathfrak{h}}_{(n)}).
\end{eqnarray} 
As explained in \cite{eyallietheory} for  the case of the deformation family, the $\theta$-stable  triangular decomposition of $\fg$ 
 gives rise to the triangular decomposition
\[\g_{(n)}=\boldsymbol{\mathfrak{n}}_{(n)}^-\oplus \boldsymbol{\mathfrak{h}}_{(n)}\oplus \boldsymbol{\mathfrak{n}}_{(n)}^+ \]
and the corresponding decomposition of $\mathcal{U}(\boldsymbol{\mathfrak{g}}_{(n)})$
\begin{eqnarray}\nonumber
&&\mathcal{U}(\boldsymbol{\mathfrak{g}}_{(n)})=\operatorname{S}(\boldsymbol{\mathfrak{h}}_{(n)}) \oplus \left(\boldsymbol{\mathfrak{n}}_{(n)}^-\mathcal{U}(\boldsymbol{\mathfrak{g}}_{(n)})+\mathcal{U}(\boldsymbol{\mathfrak{g}}_{(n)}) \boldsymbol{\mathfrak{n}}_{(n)}^+ \right).\end{eqnarray}
We denote the corresponding  projection on the first direct summand by \begin{equation*}
    P_{\boldsymbol{\mathfrak{h}}_{(n)}}:\mathcal{U}(\boldsymbol{\mathfrak{g}}_{(n)})\longrightarrow \operatorname{S}(\boldsymbol{\mathfrak{h}}_{(n)}).
    \end{equation*}
Since 
\begin{eqnarray}\nonumber
&&\operatorname{S}(\boldsymbol{\mathfrak{h}}_{(n)})\subseteq \operatorname{S}(\boldsymbol{\mathfrak{h}}), \hspace{3mm}
\left(\boldsymbol{\mathfrak{n}}_{(n)}^-\mathcal{U}(\boldsymbol{\mathfrak{g}}_{(n)})+\mathcal{U}(\boldsymbol{\mathfrak{g}}_{(n)}) \boldsymbol{\mathfrak{n}}_{(n)}^+ \right)\subseteq \left(\boldsymbol{\mathfrak{n}}^-\mathcal{U}(\boldsymbol{\mathfrak{g}})+\mathcal{U}(\boldsymbol{\mathfrak{g}}) \boldsymbol{\mathfrak{n}}^+ \right),\end{eqnarray}
$P_{\boldsymbol{\mathfrak{h}}_{(n)}}$  is the restriction of the projection of the constant family 
$P_{\boldsymbol{\mathfrak{h}}_{(0)}}=P_{\boldsymbol{\mathfrak{h}}}$ to 
${\mathcal{U}(\boldsymbol{\mathfrak{g}}_{(n)})}$. 
By definition 
\[\boldsymbol{\operatorname{HC}_{\fh}^{\mathfrak{g}}}=\widetilde{\mathbb{I}_{R}\otimes {\rho}}  \circ P_{\boldsymbol{\mathfrak{h}}},\]
where $\rho$ is the half sum of those roots appearing in $\mathfrak{n}^+$. Hence,

\begin{eqnarray}\nonumber 
    && \boldsymbol{\operatorname{HC}_{\fh}^{\mathfrak{g}}}(\mathcal{Z}(\boldsymbol{\mathfrak{g}}_{(n)}))=\widetilde{\mathbb{I}_{R}\otimes {\rho}}  \left( P_{\boldsymbol{\mathfrak{h}}}(\mathcal{Z}(\boldsymbol{\mathfrak{g}}_{(n)})) \right)=\widetilde{\mathbb{I}_{R}\otimes {\rho}}  \left( P_{\boldsymbol{\mathfrak{h}_{(n)}}}(\mathcal{Z}(\boldsymbol{\mathfrak{g}}_{(n)})) \right)\subseteq \\ \nonumber
    && \widetilde{\mathbb{I}_{R}\otimes {\rho}}\left( \operatorname{S}(\boldsymbol{\mathfrak{h}}_{(n)})\right)\subseteq \operatorname{S}(\boldsymbol{\mathfrak{h}}_{(n)}).
\end{eqnarray}
Since $\mathcal{Z}(\boldsymbol{\mathfrak{g}}_{(n)})\subseteq \mathcal{Z}(\boldsymbol{\mathfrak{g}})$,
\begin{eqnarray}
    && \boldsymbol{\operatorname{HC}_{\fh}^{\mathfrak{g}}}(\mathcal{Z}(\boldsymbol{\mathfrak{g}}_{(n)}))\subseteq \operatorname{S}(\boldsymbol{\mathfrak{h}})^{W(\fg,\fh)}. 
\end{eqnarray}
Hence 
\begin{eqnarray}
    && \boldsymbol{\operatorname{HC}_{\fh}^{\mathfrak{g}}}(\mathcal{Z}(\boldsymbol{\mathfrak{g}}_{(n)}))\subseteq \operatorname{S}(\boldsymbol{\mathfrak{h}})^{W(\fg,\fh)}\cap  \operatorname{S}(\boldsymbol{\mathfrak{h}}_{(n)})=\operatorname{S}(\boldsymbol{\mathfrak{h}}_{(n)})^{W(\fg,\fh)}.
\end{eqnarray}
\end{proof}    

 \begin{definition*}
     For every $n\in \N_0,$ \textit{the Harish-Chandra homomorphism for $\boldsymbol{\mathfrak{g}}_{(n)}$ with respect to the fundamental Cartan subfamily 
     $\boldsymbol{\mathfrak{h}}_{(n)}$} is the map $$\boldsymbol{\operatorname{HC}_{\fh_{(n)}}^{\mathfrak{g}_{(n)}}}:=\boldsymbol{\operatorname{HC}_{\fh}^{\mathfrak{g}}}|_{\mathcal{Z}(\boldsymbol{\mathfrak{g}}_{(n)})}:\mathcal{Z}(\boldsymbol{\mathfrak{g}}_{(n)})\longrightarrow \operatorname{S}(\boldsymbol{\mathfrak{h}}_{(n)})^{W(\fg,\fh)}.$$
 \end{definition*}
As a restriction of an isomorphism of algebraic families of commutative complex algebras over  $\C$, $\boldsymbol{\operatorname{HC}_{\fh_{(n)}}^{\mathfrak{g}_{(n)}}}$ is a monomorphism of algebraic families of complex commutative  algebras over $\C$.

\subsubsection{Infinitesimal character with respect to  $\boldsymbol{\mathfrak{h}}_{(n)}$}

\begin{definition*}
    Let  $\boldsymbol{V}$ be a quasi-simple  algebraic family of   $(\boldsymbol{\mathfrak{g}}_{(n)},K)$-modules with a corresponding character 
    \[\chi_{\boldsymbol{V}}:\mathcal{Z}(\boldsymbol{\mathfrak{g}}_{(n)})\longrightarrow R_{\boldsymbol{X}}.\]
    We say that $\boldsymbol{V}$ \textit{has an infinitesimal character with respect to $\boldsymbol{\mathfrak{h}}_{(n)}$} if there is $\boldsymbol{\lambda} \in \boldsymbol{\mathfrak{h}}_{(n)}^*$ such that $\chi_{\boldsymbol{V}}=\widehat{\boldsymbol{\lambda}}\circ \boldsymbol{\operatorname{HC}_{\fh_{(n)}}^{\mathfrak{g}_{(n)}}}$, where $\widehat{\boldsymbol{\lambda}}$ is the unique unital  $\mathbb{C}[t]$-algebra homomorphism from  $\operatorname{S}(\boldsymbol{\mathfrak{h}})$ into $\mathbb{C}[t]$ such that $\widehat{\boldsymbol{\lambda}}(\boldsymbol{h})=\boldsymbol{\lambda}(\boldsymbol{h})$ for every $\boldsymbol{h}\in \boldsymbol{\mathfrak{h}}_{(n)}$. If in addition the infinitesimal character  is of the form $1\otimes \lambda$ for some $\lambda \in \mathfrak{h}^*$ we say that $\boldsymbol{V}$ has   a \textit{constant} infinitesimal character with respect to $\boldsymbol{\mathfrak{h}}_{(n)}$.
\end{definition*}
\begin{remark*}
    As demonstrated in \cite{eyallietheory} for the case of the deformation family,   not every quasi-simple algebraic family of $(\boldsymbol{\mathfrak{g}}_{(n)},K)$-modules has an infinitesimal character with respect to a fundamental Cartan subfamily.
\end{remark*}

\begin{example*}[Verma-like families of modules for a $\theta$-fixed Cartan] 
	We keep the above notation but here we shall further assume that  $\mathfrak{h}$  is a $\theta$-fixed Cartan subalgebra of $\fg$. In particular $\fh=\ft\subset \fk$.  For a fixed $n \in \N_0$, we set $\boldsymbol{\mathfrak{b}}^{\pm}_{(n)}:=\boldsymbol{\mathfrak{t}} \oplus \boldsymbol{\mathfrak{n}}_{(n)}^{\pm}$. For any $\boldsymbol{\lambda} \in \boldsymbol{\mathfrak{t}}^*$, we let $\mathbb{C}[t]_{\boldsymbol{\lambda}}$ be the free rank one $\mathbb{C}[t]$-module $\mathbb{C}[t]$ equipped with the  $\boldsymbol{\mathfrak{b}}^{+}_{(n)}$-module structure $\rho_{\boldsymbol{\lambda}}: \boldsymbol{\mathfrak{b}}^{+}_{(n)}\longrightarrow \operatorname{End}_{\mathbb{C}[t]}(\mathbb{C}[t]_{\boldsymbol{\lambda}})$  in which 
	\[\rho_{\boldsymbol{\lambda}}(X)1=\begin{cases}
		0 & X\in \boldsymbol{\mathfrak{n}}_{(n)}^+ \\
		\boldsymbol{\lambda}(X)1& X\in \boldsymbol{\mathfrak{t}}.
	\end{cases} \]
	We set $V(\boldsymbol{\lambda}):=\mathcal{U}(\boldsymbol{\mathfrak{g}}_{(n)})\otimes_{\mathcal{U}(\boldsymbol{\mathfrak{b}}_{(n)}^+)}\mathbb{C}[t]_{\boldsymbol{\lambda}}$ equipped with the $\boldsymbol{\mathfrak{g}}_{(n)}$ module structure coming from left multiplication in  $\mathcal{U}(\boldsymbol{\mathfrak{g}}_{(n)})$. By PBW-theorem, as $\mathbb{C}[t]$-modules we have $V(\boldsymbol{\lambda})\simeq \mathcal{U}(\boldsymbol{\mathfrak{n}}_{(n)}^-)$ and in particular $V(\boldsymbol{\lambda})$ is a free $\mathbb{C}[t]$-module (of infinite rank) and hence flat. 
    Since the Cartan $\fh=\ft$ is $\theta$-fixed it can be shown that  
    \begin{eqnarray}\nonumber
&&\mathcal{Z}(\boldsymbol{\mathfrak{g}}_{(n)})\subseteq \operatorname{S}(\boldsymbol{\mathfrak{t}}) \oplus \boldsymbol{\mathfrak{n}}_{(n)}^-\mathcal{U}(\boldsymbol{\mathfrak{g}}_{(n)}) \boldsymbol{\mathfrak{n}}_{(n)}^+ .\end{eqnarray}
    For any $Z\in \mathcal{Z}(\boldsymbol{\mathfrak{g}}_{(n)})$, we have 
	\begin{eqnarray}\nonumber
		&& Z( 1\otimes 1)=Z\otimes 1=(Z-P_{\boldsymbol{\mathfrak{t}}}(Z))\otimes 1+P_{\boldsymbol{\mathfrak{t}}}(Z)\otimes 1=\\ \nonumber
		&& 0+1\otimes (\rho_{\boldsymbol{\lambda}} (P_{\boldsymbol{\mathfrak{t}}}(Z))1)= \widehat{\boldsymbol{\lambda}}(P_{\boldsymbol{\mathfrak{t}}}(Z)) (1\otimes 1)= (\widehat{\boldsymbol{\lambda}}\circ P_{\boldsymbol{\mathfrak{t}}})(Z) (1\otimes 1)=\\ \nonumber
		&&(\widehat{\boldsymbol{\lambda}}\circ (\widetilde{\mathbb{I}_{R}\otimes {\rho}})^{-1} \circ \widetilde{\mathbb{I}_{R}\otimes {\rho}} \circ P_{\boldsymbol{\mathfrak{t}}})(Z) (1\otimes 1)=\\ \nonumber
		&&(\widehat{\boldsymbol{\lambda}}\circ (\widetilde{\mathbb{I}_{R}\otimes {-\rho}}) \circ \boldsymbol{\operatorname{HC}_{\fh}^{\mathfrak{g}}} )(Z) (1\otimes 1)=((\widehat{\boldsymbol{\lambda}+(\mathbb{I}_{R}\otimes \rho)})    \circ \boldsymbol{\operatorname{HC}_{\fh_{(n)}}^{\mathfrak{g}_{(n)}}} )(Z) (1\otimes 1).
	\end{eqnarray}
	Hence $V(\boldsymbol{\lambda})$ is quasi-simple with infinitesimal character with respect to $\boldsymbol{\mathfrak{h}}_{(n)}=\boldsymbol{\mathfrak{t}}$ that is given by $\boldsymbol{\lambda}+(\mathbb{I}_{R}\otimes \rho)\in \boldsymbol{\mathfrak{t}}^*$.
\end{example*}

\section{A unified algebraic approach for Dirac operators}\label{sec4}
We shall be interested  in Dirac operators of the families  $(\g_{(n)},K)$.
  For that purpose we introduce a more abstract and  unified framework that covers the cases of $(\g_{(n)},K)$ and much more. This allows us to introduce Dirac operators for families. Along the way we construct Clifford algebras and spin modules for these families and prove the expected relation between the square of the Dirac operator and the Casimir of the family of Lie algebras.

  Throughout this section we shall keep our assumptions on $G(\R)$ and keep the notations introduced before. 
  In particular $G(\R)$ is a connected real reductive group.  The invariant form $\beta$ on $\fg$ is positive definite on $\fp^{\sigma}$ and negative definite on  $\fk^{\sigma}$ and its restriction to $[\fg,\fg]$ coincides with its  Killing form. 
  We fix $\theta$-stable triangular decompositions 
  \[\mathfrak{g}=\mathfrak{n}^-\oplus \mathfrak{h}\oplus \mathfrak{n}^+, \quad \mathfrak{k}=\mathfrak{n}_{\mathfrak{k}}^-\oplus \mathfrak{t}\oplus \mathfrak{n}^+_{\mathfrak{k}}, \]
  with $\mathfrak{h}=\mathfrak{t}\oplus \mathfrak{a}$ being a fundamental Cartan subalgebra of $\fg$ and   $\mathfrak{n}^{\pm}_{\mathfrak{k}}\subseteq \mathfrak{n}^{\pm}$.
\subsection{The families of Harish-Chandra pairs}\label{subsection41}
In this subsection we introduce the class of families of Harish-Chandra pairs that shall be used throughout the paper.

From now on we let $R$ be a fixed principal ideal domain containing $\C$. In practice, we shall only use the cases in which $R$ is $\C$,  $\C[t]$ or $\C[t,t^{-1}]$.

We let $(R\otimes_{\C}\fg,K)$ be the constant family of Harish-Chandra pairs over the complex affine algebraic variety $\operatorname{Spec}_m(R)$ (the variety of maximal ideals of $R$). 

We shall assume that $(\g,K)$ is a subfamily of the constant family $(R\otimes_{\C}\fg,K)$ such that there exists a non-zero ideal $I=I_{\g}$ of $R$ and
\[\g=\k\oplus \boldsymbol{\fp}, \quad \text{with} \quad \k=(R\otimes_{\C}\fk),  \boldsymbol{\fp}=(I_{\fg} \otimes_{\C}\fp). \]
In particular $\g$ is stable under the involution $\boldsymbol{\theta}:=\mathbb{I}_R\times \theta$ of $R\otimes_{\C}\fg$, and the above decomposition is the eigenspace decomposition of the involution $\boldsymbol{\theta}|_{\g}$.

Since $R$ is a principal ideal domain, there is a non-zero $r=r_{\g}\in R$ such that $\langle r_{\g} \rangle$, the ideal generated by $r_{\g}$, is $I_{\g}$. The generator $r$ is unique up to multiplication by an invertible element of $R$. We fix such a generator  once and for all.

\begin{example*}
    The families $\g_{(n)}$ from Section \ref{vdf} satisfy the above assumptions with $R=\C[t]$, and
    \[\k=\mathbb{C}[t]\otimes_{\mathbb{C}}\mathfrak{k},\quad \p= \langle t^{n}\rangle \otimes_{\mathbb{C}}\mathfrak{p}.\] 
    
\end{example*}
\begin{remark*}
The families $(\g_{(n)},K)$ essentially exhaust all families satisfying our assumptions. If $(\g,K)$ is a subfamily of the constant family $(R\otimes_{\C}\fg,K)$ satisfying the abovementioned assumptions,  the special fibers of $\g$ are isolated and by restriction to a small enough neighborhood we can assume that there is at most one special fiber. Any subfamily of the constant family  satisfying the above assumptions  and with at most one special fiber, must be a pullback under an affine automorphism of $\mathbb{A}^1_{\C}$ of one of the families $(\g_{(n)},K)$.  
See the related work on classification  of    families of Harish-Chandra pairs within the constant family  of $SL(2,\R)$ \cite{subag2024extensions}.
\end{remark*}

\subsection{The quadratic Harish-Chandra pairs}
In this subsection we show that a family  of Harish-Chandra pairs in our class has a natural symmetric form.

\subsubsection{The symmetric form on the constant family $(R\otimes_{\C}\fg,K)$ }
We let $\beta$   be a symmetric bilinear form on $\fg$ as in Section \ref{invform}. In particular $(\fg,K)$ is a quadratic Harish-Chandra pair (over $\operatorname{Spec}(\C)$).
We denote the $R$-bilinear extension of $\beta$ to $(R\otimes_{\C} \fg)\times (R\otimes_{\C} \fg) $ by $1_R \otimes \beta$.

\begin{lemma*}
    The form $1_R\otimes {\beta}$ on $R\otimes_{\mathbb{C}}\mathfrak{g}$ is orthogonalizable. 
\end{lemma*}
\begin{proof}
The form $\beta$  is nondegenerate as a form over a field hence it is orthogonalizable. 
   Let   $\{e_{1},e_{2},...,e_n\}$ be an orthonormal  basis for   $\fg$  with respect to $\beta$.  Then clearly $\{1\otimes e_{1},1\otimes e_{2},...,1\otimes e_n\}$ is an orthonormal  basis  for $R\otimes_{\C}\fg$  with respect to $1_R\otimes {\beta}$.
\end{proof}  
\subsubsection{Restriction of the form to $\g$ }\label{s422}
Restriction of the form $1_R\otimes {\beta}$ to  $\g=\k\oplus \boldsymbol{\fp}$ is not in general orthogonalizable. Clearly, it is orthogonalizable on $\k$, and the next lemma shows that it is a multiple of an orthogonalizable form  on $\boldsymbol{\fp}=\langle r\rangle \otimes_{\C} \fp$. 
\begin{lemma*}
    There is a unique $R$-bilinear symmetric
     $(\boldsymbol{\fk},K)$-invariant   and orthogonalizable form $ \boldsymbol{\beta}_{\p}^r:\boldsymbol{\fp}\times \boldsymbol{\fp}\longrightarrow  R$ such that 
    \[ \left(1\otimes{\beta_{\fp}}\right)|_{\boldsymbol{\fp}\times \boldsymbol{\fp}}=r^2\boldsymbol{\beta}_{\p}^r,\]
where $\beta_{\mathfrak{p}}\coloneqq \beta\vert_{\mathfrak{p}}$
\end{lemma*}
\begin{proof}
Since $\left(1\otimes{\beta_{\fp}}\right)|_{\boldsymbol{\fp}\times \boldsymbol{\fp}}(\boldsymbol{\mathfrak{p}},\boldsymbol{\mathfrak{p}})\subset \langle r^2 \rangle$, there is a form $ \boldsymbol{\beta}_{\p}^r:\boldsymbol{\fp}\times \boldsymbol{\fp}\longrightarrow  R$ such that 
    \[ \left(1\otimes{\beta_{\fp}}\right)|_{\boldsymbol{\fp}\times \boldsymbol{\fp}}=r^2\boldsymbol{\beta}_{\p}^r.\]
 Since $1\otimes{\beta_{\fp}}$ is $(\boldsymbol{\fk},K)$-invariant then so is $\boldsymbol{\beta}_{\p}^r$. 
 We now  prove  orthogonalizability. If $\{e_{m+1},e_{m+2},...,e_n\}$ is an orthonormal  basis for $\fp$  with respect to $\beta_{\fp}$ then   $\{r\otimes e_{m+1},r\otimes e_{m+2},...,r\otimes e_n\}$ is  an orthonormal  basis for $\boldsymbol{\fp}$ with respect to $\boldsymbol{\beta}_{\p}^r$.
\end{proof}
 We shall make use of the form $ \boldsymbol{\beta}_{\p}^r:\boldsymbol{\fp}\times \boldsymbol{\fp}\longrightarrow  R$ throughout the text. 
 \subsection{Complex forms}
 In this subsection we show that the family of quadratic spaces $(\boldsymbol{\mathfrak{p}},\boldsymbol{\beta}_{\p}^r)$ has a natural complex form. This induces a complex form for the corresponding family of orthogonal Lie algebras and for the corresponding family of Clifford algebras.

\subsubsection{ The complex form of $(\boldsymbol{\mathfrak{p}},\boldsymbol{\beta}_{\p}^r)$}\label{compform}

\begin{proposition*} \quad
\begin{enumerate}
    \item The restriction of $ \boldsymbol{\beta}_{\p}^r$ to the complex vector subspace $\fp_r:=\C r\otimes_{\C} \fp \subseteq \boldsymbol{\mathfrak{p}}$ is non-degenerate and $(1\otimes \fk,K)$-invariant. 
    \item There is a canonical $(1\otimes \fk,K)$-equivariant isomorphism of quadratic spaces between $(\fp_r,\boldsymbol{\beta}_{\p}^r|_{\fp_r})$ and $(\fp,\beta_{\fp})$.
    \item The complex quadratic space $(\fp_r,\boldsymbol{\beta}_{\p}^r|_{\fp_r})$ is a complex form of $(\boldsymbol{\fp},\boldsymbol{\beta}_{\p}^r)$, that is,  the natural multiplication 
    \begin{eqnarray}\nonumber
    & (R\otimes_{\C}\fp_r,1\otimes \boldsymbol{\beta}_{\p}^r|_{\fp_r})& \longrightarrow (\boldsymbol{\fp},\boldsymbol{\beta}_{\p}^r)\\  \nonumber
    & f\otimes X&\longmapsto   f{X},
    \end{eqnarray}
     is a $(1\otimes \fk,K)$-equivariant isometry between the quadratic spaces. 
\end{enumerate}
\end{proposition*}
\begin{proof}
Claim 1. follows from Claim 2.  The canonical  isomorphism in Claim 2 is
  $T_r:\fp \longrightarrow \fp_r$  sending $ X$ to $r\otimes X$. The proof of Claim 3 is by direct calculations.  
\end{proof}

\begin{remark*}
    Note that $\langle r\rangle =\langle s\rangle $  if and only if there exists $z\in R^{\times}$ such that $s=zr$. If indeed $\langle r\rangle =\langle s\rangle $,  the last proposition implies that there is a canonical isomorphism of quadratic spaces between $(\fp_r,\boldsymbol{\beta}_{\p}^r|_{\fp_r})$ and  $(\fp_s,\boldsymbol{\beta}_{\p}^r|_{\fp_s})$. 
\end{remark*}
 
\subsubsection{The complex form of $\mathfrak{so}(\boldsymbol{\fp},\boldsymbol{\beta}_{\p}^r)$}
The complex form  $(\fp_r,\boldsymbol{\beta}_{\p}^r|_{\fp_r})$  of the quadratic space $(\boldsymbol{\fp},\boldsymbol{\beta}_{\p}^r)$   gives rise to a complex form of the corresponding orthogonal Lie algebra. We  summarize it in the proposition below. We omit the simple proof.
\begin{proposition*}
    The complex Lie subalgebra  $\mathfrak{so}(\fp_r,\boldsymbol{\beta}_{\p}^r|_{\fp_r})$ is a complex form of the $R$-Lie algebra $\mathfrak{so}(\boldsymbol{\fp},\boldsymbol{\beta}_{\p}^r)$. The isomorphism is explicitly given by 
\begin{eqnarray}\nonumber
    &&  E_r:R\otimes \mathfrak{so}(\fp_r,\boldsymbol{\beta}_{\p}^r|_{\fp_r})\longrightarrow \mathfrak{so}(\boldsymbol{\fp},\boldsymbol{\beta}_{\p}^r)\\ \nonumber
    && \quad \quad \quad    f\otimes  \xi  \longmapsto   f \overline{\xi},
      \end{eqnarray}
      where $\overline{\xi}$ is the $R$-linear extension of $\xi$. 
\end{proposition*}
 
The isometry $T_r:\fp \longrightarrow \fp_r$ introduced in the proof of Proposition \ref{compform}  induces an isomorphism of the corresponding orthogonal Lie algebras
\begin{eqnarray}\nonumber
    &\operatorname{Ad}(T_r):&\mathfrak{so}(\fp,\beta_{\fp})\longrightarrow  \mathfrak{so}(\fp_r,\boldsymbol{\beta}_{\p}^r|_{\fp_r})\\ \nonumber
    &&  \xi  \longmapsto    {T_r\circ \xi \circ T_r^{-1}}.
      \end{eqnarray}
As a result, we obtain  a chain of isomorphisms of Lie algebras over $R$:
    \begin{eqnarray}\nonumber
    &&R\otimes \mathfrak{so}(\fp,\beta_{\fp})\xrightarrow[]{ \mathbb{I}_R\otimes \operatorname{Ad}(T_r)} R\otimes \mathfrak{so}(\fp_r,\boldsymbol{\beta}_{\p}^r|_{\fp_r})\xrightarrow[]{ E_r} \mathfrak{so}(\boldsymbol{\fp},\boldsymbol{\beta}_{\p}^r).
      \end{eqnarray}

We note that any $L\in \mathfrak{so}(R\otimes_{\C}\fp, 1\otimes {\beta}_{\fp})$ must preserve $\boldsymbol{\fp}=I\otimes_{\C}\fp$.  Hence such an operator can be restricted to an operator on $\boldsymbol{\fp}$.  
We let \[\operatorname{Res}^{R\otimes_{\C}\fp}_{\boldsymbol{\fp}}:\mathfrak{so}(R\otimes_{\C}\fp, 1\otimes {\beta}_{\fp})\longrightarrow \mathfrak{so}(\boldsymbol{\fp}, (1\otimes {\beta})_{\boldsymbol{\fp}})=\mathfrak{so}(\boldsymbol{\fp},\boldsymbol{\beta}^r_{\boldsymbol{\fp}})\]  be the morphism of Lie algebras obtained by restriction to $\boldsymbol{\fp}$.

\begin{lemma*}\label{lem2.10}
Under the canonical isomorphism $R\otimes_{\C} \mathfrak{so}(\fp,\beta_{\fp})\cong \mathfrak{so}(R\otimes_{\C}\fp, 1\otimes {\beta}_{\fp})$, the  morphism $E_r\circ \left( \mathbb{I}_R\otimes \operatorname{Ad}(T_r)\right):R\otimes \mathfrak{so}(\fp,\beta_{\fp})\longrightarrow   \mathfrak{so}(\boldsymbol{\fp},\boldsymbol{\beta}_{\p}^r)$ coincides with $\operatorname{Res}^{R\otimes_{\C}\fp}_{\boldsymbol{\fp}}$.
\end{lemma*}
\begin{proof}
For any $\xi\in \mathfrak{so}(\fp,\beta_{\fp})$, 
 \begin{eqnarray}\nonumber
  &&E_r\circ (\mathbb{I}_R\otimes \operatorname{Ad}(T_r))(1\otimes \xi) =  E_r(1\otimes T_r\circ \xi \circ T_r^{-1} ) = \overline{T_r\circ \xi \circ T_r^{-1}}  
\end{eqnarray} 
Applying this operator to $r\otimes X\in \boldsymbol{\fp}$ gives 
\begin{eqnarray}\nonumber
  && \overline{T_r\circ \xi  \circ T_r^{-1}}(r\otimes X)=r\otimes \xi(X)=\left(\operatorname{Res}^{R\otimes_{\C}\fp}_{\boldsymbol{\fp}}(\mathbb{I}\otimes \xi)\right)(r\otimes X). 
\end{eqnarray} 
\end{proof}
\subsubsection{The complex form of $Cl(\boldsymbol{\fp},\boldsymbol{\beta}_{\p}^r)$}
It follows from Proposition \ref{compform} that the complex Clifford algebra $Cl(\fp_r,\boldsymbol{\beta}_{\p}^r|_{\fp_r})$ is a complex form of $Cl(\boldsymbol{\fp},\boldsymbol{\beta}_{\p}^r)$ as an algebra over $R$. Moreover we get a chain of isomorphisms 
\[ R\otimes_{\C}Cl(\fp,\beta) \longrightarrow R\otimes_{\C}Cl(\fp_r,\boldsymbol{\beta}_{\p}^r|_{\fp_r})\longrightarrow Cl(\boldsymbol{\fp},\boldsymbol{\beta}_{\p}^r) \]
that fit into the following  commutative diagram 
\[\xymatrix{
& R\otimes_{\C}Cl(\fp,\beta) \ar[r]
& R\otimes_{\C}Cl(\fp_r,\boldsymbol{\beta}_{\p}^r|_{\fp_r})\ar[r] & Cl(\boldsymbol{\fp},\boldsymbol{\beta}_{\p}^r)  \\
&R\otimes_{\C} \fp \ar[u]^{\mathbb{I}_R\otimes \gamma } \ar[r]^{\mathbb{I}_R\otimes T_r} & R\otimes_{\C}  \fp_r\ar[u]_{\mathbb{I}_R\otimes  \gamma_{\fp_r} } \ar[r]^{m} & \p\ar[u]^{\gamma_{\boldsymbol{\fp},\boldsymbol{\beta}_{\boldsymbol{\fp}}^r}}}\]
where the map $m$ is the natural action of $R$ on $\fp_r$, while   $\gamma_{\fp_r}$ and $\gamma_{\boldsymbol{\fp},\boldsymbol{\beta}_{\boldsymbol{\fp}}^r}$ are the canonical embeddings.

\subsection{Canonical sections}\label{CanSec}
In this subsection we show that the definition of the  canonical element associated with a complex nondegenerate quadratic space that was given in Section \ref{CanEl} makes sense  for families of quadratic spaces as long as the symmetric form  is  unimodular.

Assuming that $\boldsymbol{\beta}$ is a bilinear symmetric form on $\g$ and $\boldsymbol{\mathfrak{l}}$ is a submodule of $\g$ on which  $\boldsymbol{\beta}$  is unimodular,  
 we have the isomorphisms 
\begin{equation} 
	\mathrm{End}(\boldsymbol{\mathfrak{l}})\longrightarrow  \boldsymbol{\mathfrak{l}}^*\otimes_{R}\boldsymbol{\mathfrak{l}}\longrightarrow \boldsymbol{\mathfrak{l}}\otimes_{R}\boldsymbol{\mathfrak{l}}.
\end{equation}
The  canonical element $\omega(\boldsymbol{\mathfrak{l}},\boldsymbol{\beta})$ is defined as the image in 
$\boldsymbol{\mathfrak{l}}\otimes_{R} \boldsymbol{\mathfrak{l}}$ under the abovementioned map of  the identity operator $\mathbb{I}_{\boldsymbol{\mathfrak{l}}}\in \mathrm{End}(\boldsymbol{\mathfrak{l}})$.

We denote the image of $\omega(\boldsymbol{\mathfrak{l}},\boldsymbol{\beta})$ in $\mathcal{U}(\g)$
under  the obvious morphisms 
  \[\boldsymbol{\mathfrak{l}}\otimes_{R}\boldsymbol{\mathfrak{l}}\longrightarrow T(\boldsymbol{\mathfrak{l}})\longrightarrow T(\g)\longrightarrow \mathcal{U}(\g),\]   by $\Omega(\boldsymbol{\mathfrak{l}},\boldsymbol{\beta})$. 

\begin{definition*}\label{casd}
     Under our running  assumptions introduced in Section \ref{subsection41}, in which $(\g,K)$ is a subfamily of the constant family $(R\otimes_{\C}\fg,K)$ such that 
     	\[\g=\k\oplus \boldsymbol{\fp}, \quad \text{with} \quad \k=(R\otimes_{\C}\fk),\hspace{1mm}  \boldsymbol{\fp}=(I_{\fg} \otimes_{\C}\fp), \] 
        for every  element $r\in R$ such that $I_{\fg}=\langle r \rangle$, we define 
$$\Omega(\g,\beta,r):=r^2\otimes \Omega (\fg,\beta)\in R\otimes \mathcal{Z}(\fg)=\mathcal{Z}(R\otimes_{\C}\fg)$$  and call it \textit{the Casimir of $\g$ with respect to  $\beta,r$.}
 \end{definition*}
We note that $r^2\Omega(R\otimes_{\C}\fg,1_R\otimes {\beta})=\Omega(\g,\beta,r)$. 
 
\begin{proposition*}
    $\Omega(\g,\beta,r)=r^2\otimes  \Omega({\fk},\beta_{\fk})+\Omega({\fp}_r,\boldsymbol{\beta}_{\p}^r|_{\fp_r})$ 
    and in particular $\Omega(\g,\beta,r)\in \mathcal{Z}(\g)$.
\end{proposition*}
\begin{proof}
    Pick  an orthonormal basis  $\{e_1,e_2,...,e_m\}$ of  $\fk$ with respect to $\beta_{\fk}$  and  an orthonormal basis $\{e_{m+1},e_{m+2},...,e_n\}$ of $\fp$ with respect to $\beta_{\fp}$. There are such bases since $(\fk,\beta_{\fk})$ and $(\fp,\beta_{\fp})$ are non-degenerate complex quadratic spaces. Hence $\Omega (\fg,\beta)=\sum_{i=1}^ne_i^2$. The set $\{1\otimes e_1,1\otimes e_2,...,1\otimes e_m\}$ is an orthonormal basis of $\k$ with respect to $(1_R\otimes \beta)|_{\k}$ and $\{r\otimes e_{m+1},r\otimes e_{m+2},...,r\otimes e_n\}$ is an orthonormal basis of $\boldsymbol{\fp}$ with respect to $\boldsymbol{\beta}_{\p}^r$ and also an orthonormal basis of $\fp_r$  with respect to $\boldsymbol{\beta}_{\p}^r|_{\fp_r}$, hence
    \begin{align*}  \Omega(\g,\beta,r)&=r^2\otimes \sum_{i=1}^ne_i^2= r^2\otimes \sum_{i=1}^me_i^2+\sum_{i=m+1}^n(r\otimes e_i)^2\\
    &=r^2\otimes  \Omega({\fk},\beta_{\fk})+\Omega({\fp}_r,\boldsymbol{\beta}_{\p}^r|_{\fp_r}).
    \end{align*}
\end{proof}


\subsection{The algebra  $\boldsymbol{A}=\mathcal{U}(\g)\otimes_R {Cl}(\boldsymbol{\fp},\boldsymbol{\beta}^r_{\boldsymbol{\fp}})$}\label{The algebra A}
 
We define an $R$-algebra via
\[\boldsymbol{A}\coloneqq{A}(\g,\beta,r)\coloneqq\mathcal{U}(\g)\otimes_R {Cl}(\boldsymbol{\fp},\boldsymbol{\beta}_{\boldsymbol{\fp}}^r).\] 
Since $\g\subset R\otimes_{\C} \fg$, we have an inclusion   $\iota_{\g}:\mathcal{U}(\g) \longrightarrow R\otimes_{\C}\mathcal{U}(\fg) =\mathcal{U}(R\otimes_{\C} \fg)$. The Clifford algebra  
${Cl}(\boldsymbol{\fp},\boldsymbol{\beta}_{\boldsymbol{\fp}}^r)$ is not a subalgebra of the constant Clifford algebra   $R\otimes_{\C} {Cl}(\fp,{ \beta}_{\fp})= {Cl}(R\otimes_{\C}\fp,1\otimes \beta_{\fp})$ but there is an  isomorphism  $T_r:R\otimes_{\C} {Cl}(\fp,{ \beta}_{\fp}) \longrightarrow  {Cl}(\boldsymbol{\fp},\boldsymbol{\beta}_{\boldsymbol{\fp}}^r)  $ that is induced from the unique isometry $T_r:(R\otimes_{\C} \fp,\beta_{\fp}) \longrightarrow (\boldsymbol{\fp}=\langle r\rangle \otimes_{\C}\fp ,\boldsymbol{\beta}_{\boldsymbol{\fp}}^r)$ from the proof of Proposition \ref{compform}, satisfying $T_r(1\otimes X)=r\otimes X$ for every $X\in \fp$. 
We obtain an embedding of algebras 
\[ \iota_{\g}\otimes T_r^{-1} :A(\g,\beta,r) \longrightarrow \left(R\otimes_{\C}\mathcal{U}(\fg)\right) \otimes_R \left(R\otimes_{\C} {Cl}(\fp,{ \beta}_{\fp})  \right).  \]
We define an embedding 
\[ \iota_{\g}T_r^{-1} :A(\g,\beta,r)  \longrightarrow R\otimes_{\C} A(\fg,\beta,1).     \] 
via the commutative diagram below 
\[\xymatrix{
 A(\g,\beta,r)\ar[rrd]^{\iota_{\g}T_r^{-1}}\ar[rr]^{\hspace{-20mm} \iota_{\g}\otimes T_r^{-1}}&& \left(R\otimes_{\C}\mathcal{U}(\fg)\right) \otimes_R \left(R\otimes_{\C} {Cl}(\fp,{ \beta}_{\fp})  \right)\ar[d]\\
 &&R\otimes_{\C} A(\fg,\beta,1)}\]
where the vertical arrow is the canonical isomorphism. 

\subsection{The pair $(\boldsymbol{A},{K})$ and its diagonal embedding}
In this subsection we explain how the pair  $(\boldsymbol{A},{K})$ naturally forms  an algebraic family of generalized pairs over $\operatorname{Spec}_m(R)$ similar to the classical setup as  in Section (\ref{genpair}).

The group $K$ acts on $\mathcal{U}(\g)$ and on ${Cl}(\boldsymbol{\fp},\boldsymbol{\beta}_{\boldsymbol{\fp}}^r)$ by automorphisms of $R$-modules. Hence $K$ acts on  $A$ by automorphisms of $R$-modules. With this action on $A$, the morphism 
\[ \iota_{\g}\otimes T_r^{-1} :A(\g,\beta,r) \longrightarrow \left(R\otimes_{\C}\mathcal{U}(\fg)\right) \otimes_R \left(R\otimes_{\C} {Cl}(\fp,{ \beta}_{\fp})  \right),  \]
from Section \ref{The algebra A},  is $K$-equivariant. 
To make 
the pair  $(\boldsymbol{A},{K})$ into an algebraic family of generalized pairs over $\operatorname{Spec}_m(R)$ we need to give a ${K}$-equivariant  embedding of  the complex Lie algebra ${\fk}=\operatorname{Lie}(K)$ into $\boldsymbol{A}$. We shall do that by giving a ${K}$-equivariant  embedding of  Lie algebras over $R$,  $\Delta:\boldsymbol{\fk}\longrightarrow \boldsymbol{A}$ that we call the diagonal embedding. 
\subsubsection{Embedding of $\mathfrak{so}(\boldsymbol{\fp},\boldsymbol{\beta}^r_{\boldsymbol{\fp}})$ in $Cl(\boldsymbol{\fp},\boldsymbol{\beta}_{\boldsymbol{\fp}}^r)$}\label{451}
Similar to the case of orthogonalizable quadratic spaces over a field, one can show that if $\{e_1,e_2,....,e_n\}$ is an orthonormal basis for $\boldsymbol{\fp}$ with respect to $ \boldsymbol{\beta}^r_{\boldsymbol{\fp}}$
 then 
 $$\{     R^r_{i,j}=R_{\boldsymbol{\beta}^r_{\boldsymbol{\fp}};e_i,e_j}:=\boldsymbol{\beta}^r_{\boldsymbol{\fp}}(\_, e_j)e_i- \boldsymbol{\beta}^r_{\boldsymbol{\fp}}(\_,e_i)e_j|1\leq i<j\leq n \}$$ 
 form a basis for $\mathfrak{so}(\boldsymbol{\fp},\boldsymbol{\beta}_{\boldsymbol{\fp}}^r)$. 
Hence there is a unique  $R$-linear map  \[\varphi^r_{\boldsymbol{\fp}}:\mathfrak{so}(\boldsymbol{\fp},\boldsymbol{\beta}^r_{\boldsymbol{\fp}})\longrightarrow Cl(\boldsymbol{\fp},\boldsymbol{\beta}_{\boldsymbol{\fp}}^r)\]
satisfying $\varphi^r_{\boldsymbol{\fp}}(R^r_{i,j})= \frac{1}{2}[\gamma_{\boldsymbol{\fp},\boldsymbol{\beta}_{\boldsymbol{\fp}}^r}(e_i),\gamma_{\boldsymbol{\fp},\boldsymbol{\beta}_{\boldsymbol{\fp}}^r}(e_j)]$, 
were $\gamma_{\boldsymbol{\fp},\boldsymbol{\beta}_{\boldsymbol{\fp}}^r}:\boldsymbol{\fp}\longrightarrow Cl(\boldsymbol{\fp},\boldsymbol{\beta}_{\boldsymbol{\fp}}^r)$ is the canonical embedding. The map $\varphi^r_{\boldsymbol{\fp}}$  can be shown to be a Lie algebra embedding. 
We define a morphism of Lie algebras over $R$,  \[\alpha_{\boldsymbol{\fp},\boldsymbol{\beta}_{\boldsymbol{\fp}}^r}:\k \longrightarrow  Cl(\boldsymbol{\fp},\boldsymbol{\beta}_{\boldsymbol{\fp}}^r)\] via
$\alpha_{\boldsymbol{\fp},\boldsymbol{\beta}_{\boldsymbol{\fp}}^r}:=\varphi^r_{\boldsymbol{\fp}} \circ \operatorname{ad}^{\boldsymbol{\fk}}$, where $\operatorname{ad}^{\boldsymbol{\fk}}$ is the canonical Lie algebra morphism $\operatorname{ad}^{\boldsymbol{\fk}}:\boldsymbol{\fk}\longrightarrow \mathfrak{so}(\boldsymbol{\fp},\boldsymbol{\beta}^r_{\boldsymbol{\fp}})$ given by the adjoint action. 

\begin{proposition*}
As morphisms of Lie algebras from 
$\boldsymbol{\fk}$ into $  Cl(\boldsymbol{\fp},\boldsymbol{\beta}_{\boldsymbol{\fp}}^r)$,
\[\alpha_{\boldsymbol{\fp},\boldsymbol{\beta}_{\boldsymbol{\fp}}^r} =T_r\circ \left(\mathbb{I}_R\otimes  \alpha_{{\fp},{\beta}_{{\fp}}^1} \right).\] 
\end{proposition*}
We omit the proof which is by direct calculation.

\subsubsection{The diagonal embedding}
\begin{definition*}
  We define a map  
  \begin{eqnarray}\nonumber
  && \Delta_{(\g,\beta,r)}:\boldsymbol{\fk}\longrightarrow A(\g,\beta,r) \\ \nonumber
  && \Delta_{(\g,\beta,r)}(X):=X\otimes 1_{Cl(\boldsymbol{\fp},\boldsymbol{\beta}_{\boldsymbol{\fp}}^r)}+1_{\mathcal{U}(\g)}\otimes \alpha_{\boldsymbol{\fp},\boldsymbol{\beta}_{\boldsymbol{\fp}}^r}(X).
  \end{eqnarray}
\end{definition*}
We note that for every $X\in \fk$,
\begin{eqnarray}\nonumber
  && \Delta_{(\g,\beta,r)}(1_R\otimes X):=(1_R\otimes X)\otimes 1_{Cl(\boldsymbol{\fp},\boldsymbol{\beta}_{\boldsymbol{\fp}}^r)} +1_{\mathcal{U}(\g)}\otimes   \left(r\otimes  \alpha_{{\fp},\boldsymbol{\beta}_{{\fp}}^1}(X)  \right) 
  \end{eqnarray}  
  in 
$\left(R\otimes_{\C}\mathcal{U}(\fg)\right) \otimes_R \left(R\otimes_{\C} {Cl}(\fp,{ \beta}_{\fp})  \right)$.
  This is mapped under $\iota_{\g}\otimes T_r^{-1}$ to 
\begin{eqnarray}\nonumber
 &&(1_R\otimes X)\otimes (1_R\otimes  1_{Cl(\fp,{\beta}_{\fp})}) +(1_R\otimes 1_{\mathcal{U}(\fg)})\otimes   \left(1_R\otimes \alpha_{{\fp},\boldsymbol{\beta}_{{\fp}}^1}(X) \right)=\\ \nonumber
 &&    \Delta_{(R\otimes \fg,1\otimes \beta,1)} (1_R\otimes X)=1_R\otimes \Delta_{\beta}(X),
\end{eqnarray}
and hence
for every $X\in \fk$,
\begin{eqnarray}\nonumber
  && \iota_{\g}T_r^{-1}\circ \Delta_{(\g,\beta,r)}(1_R\otimes X)=1_R\otimes \Delta_{\beta}(X).
  \end{eqnarray} 

\begin{lemma*}
   The map $\Delta_{(\g,\beta,r)}$ is a $K$-equivariant embedding of Lie algebras.
\end{lemma*}
\begin{proof}
This is immediate   since $\Delta_{\beta}$ is a $K$-equivariant embedding of Lie algebras.
\end{proof}

\begin{proposition*}
    The pair $(A(\g,\beta,r),{K})$ together with $\Delta_{(\g,\beta,r)}:\boldsymbol{\fk}\longrightarrow A(\g,\beta,r)$ form an algebraic family of generalized pairs over $\operatorname{Spec}_m(R)$.
\end{proposition*}
\begin{proof}
    We only need to show that the two actions of $R\otimes_{\C}\fk$ coincide.  This can be calculated directly.
 \end{proof}

\subsubsection{The pair $(\boldsymbol{A},\widetilde{K})$}
Using the canonical map $\widetilde{K}\longrightarrow K$,  $(\boldsymbol{A},\widetilde{K})$ naturally becomes an algebraic family of generalized pairs over $\operatorname{Spec}_m(R)$ with the same diagonal embedding. 

\subsection{Dirac operator}\label{DO}
In this subsection we define the Dirac operator for a family of Harish-Chandra modules in our class. We prove that its square takes a simple form.

As explained in Section \ref{CanSec}, since   $\boldsymbol{\beta}_{\p}^r:\p \times \p \longrightarrow R$ is unimodular we have a well defined 
 canonical element  $\omega(\boldsymbol{\mathfrak{p}},\boldsymbol{\beta}^r)\in \p\otimes_{R} \p$.
Using the embeddings $\gamma_{\p,\boldsymbol{\beta}_{\p}^r}:\p\longrightarrow Cl(\boldsymbol{\fp},\boldsymbol{\beta}_{\boldsymbol{\fp}}^r)$ and   $\p\hookrightarrow \mathcal{U}(\g)$ we obtain an $R$-linear morphism $$\boldsymbol{\mathfrak{p}}\otimes_{R} \boldsymbol{\mathfrak{p}} \longrightarrow A(\g,\beta,r).$$

\begin{definition*} The \textit{algebraic Dirac operator $D(\g,\beta,r)$ of  $(\g,\beta,r)$} is the image 
of $\omega(\boldsymbol{\mathfrak{p}},\boldsymbol{\beta}^r)$ in $A(\g,\beta,r)$ or equivalently, the image of the identity map $\mathbb{I}_{\p}$ of $\mathrm{End}_R(\p)$   under the morphism
$\mathrm{End}_R(\p)\longrightarrow A(\g,\beta,r)$.
\end{definition*}
If $\{e_i\}$ and $\{e_i'\}$ are dual bases of $\p$ with respect to $\boldsymbol{\beta}_{\p}^r$,  then
\begin{equation*}
		D(\g,\beta,r)=\sum_{i}e_i'\otimes\gamma_{\p,\boldsymbol{\beta}_{\p}^r}(e_i).
	\end{equation*}

Note that the map $\mathrm{End}_R(\p)\longrightarrow A(\g,\beta,r)$ is $K$-equivariant and $\mathbb{I}_{\p}$ is $K$-invariant and hence so is $D(\g,\beta,r)$. 
\begin{lemma*}  
\[(\iota_{\g}T_r^{-1})D(\g,\beta,r)=r\otimes D_{\fg,\beta} \]
\end{lemma*}

\begin{proof}
Let $\{e_1,e_2,....,e_n\}$ be an orthonormal basis of  $\fp $ with respect to $\beta_{\fp}$. Then $\{r\otimes e_1,r\otimes e_2,....,r\otimes e_n\}$ is an orthonormal basis of  $\p $ with respect to $\boldsymbol{\beta}_{\p}^r$.
Hence 
\begin{eqnarray}\nonumber
&& (\iota_{\g}\otimes T_r^{-1})D(\g,\beta,r)=\sum_{i}(\iota_{\g}\otimes T_r^{-1}) \left((r\otimes e_i)\otimes\gamma_{\p,\boldsymbol{\beta}_{\p}^r}(r\otimes e_i)\right)=\\ \nonumber
&& \sum_{i} (\iota_{\g}\otimes T_r^{-1})\left((r\otimes e_i)\otimes T_r(1\otimes \gamma_{\fp,\beta_{\fp}}(e_i))\right)=\\ \nonumber
&&\sum_{i} (r\otimes e_i)\otimes (1\otimes \gamma_{\fp,\beta_{\fp}}(e_i))=rD(R\otimes_{\C}\fg,\beta,1)
\end{eqnarray}

which implies $(\iota_{\g}T_r^{-1})D(\g,\beta,r)=r\otimes D_{\fg,\beta}$.  
\end{proof}

\subsubsection{The square of the Dirac operator }
\begin{theorem*}[The square of the Dirac operator]
In $A(\g,\beta,r)$ we have 
    \[2D^2(\g,\beta,r)=\Omega(\g,\beta,r)\otimes 1_{Cl(\p,\boldsymbol{\beta}_{\p}^r)}- r^2\Delta_{(\g,\beta_{\fp},r)} (\Omega (\boldsymbol{\fk},1\otimes \beta_{\fk}))  +r^2(\| \rho \|_{\beta}^2-\| \rho_{\fk} \|_{\beta_{\fk}}^2 )1_{A(\g,\beta,r)}\]
\end{theorem*}

\begin{proof}
Using the last lemma and    Equation (\ref{diracsquare}) for the square of the Dirac operator of a reductive group (also  see \cite[Prop. 3.1.6]{pandzic}) we have
\[ 2D^2(\fg,\beta,1)=\Omega (\fg,\beta)\otimes 1_{Cl(\fp,\beta_{\fp})}-\Delta_{(\fg,\beta_{\fp},1)} (\Omega (\fk,\beta_{\fk}))+(\| \rho \|_{\beta}^2- \| \rho_{\fk} \|_{\beta_{\fk}}^2 )1_{\mathcal{U}(\fg)}\otimes 1_{Cl(\fp,\beta_{\fp})}. \]
Hence
\begin{eqnarray}\nonumber 
&&(\iota_{\g}T_r^{-1})\left(2D^2(\g,\beta,r)\right)=r^2\otimes 2D^2(\fg,\beta,1)=\\ \nonumber
&& r^2\otimes \left( \Omega (\fg,\beta)\otimes 1_{Cl(\fp,\beta_{\fp})}-\Delta_{(\fg,\beta_{\fp},1)} (\Omega (\fk,\beta_{\fk}))+(\| \rho \|_{\beta}^2- \| \rho_{\fk} \|_{\beta_{\fk}}^2 )1_{\mathcal{U}(\fg)}\otimes 1_{Cl(\fp,\beta_{\fp})}  \right)
\end{eqnarray}
which implies the stated theorem. 
\end{proof}

\subsection{Spin modules}
In this subsection we define the spin module for families. We show that a spin module for $Cl({\mathfrak{p}},\beta)$ naturally gives rise to a spin module for $Cl(\boldsymbol{\mathfrak{p}},\boldsymbol{\beta}_{\p}^r)$.

\subsubsection{Spin modules over domains} 
Let $\l$ be  a free $R$-module equipped with
a symmetric $R$-bilinear form $ \boldsymbol{\beta}:\boldsymbol{\fl}\times \boldsymbol{\fl}\longrightarrow  R$. 
\begin{definition*}
   A \textit{spin module for the Clifford algebra ${Cl}(\l,\boldsymbol{\beta})$} is a free $R$-module $S$ that is a module for  ${Cl}(\l,\boldsymbol{\beta})$ such that $k(R)\otimes_R S$ is a simple module for $k(R)\otimes_R {Cl}(\l,\boldsymbol{\beta})\cong {Cl}(k(R)\otimes_R \l,k(R)\otimes_R \boldsymbol{\beta})$, where  $k(R)$ is the fraction field of $R$.
\end{definition*}

\subsubsection{Spin modules for $Cl(\boldsymbol{\fp},\boldsymbol{\beta}_{\boldsymbol{\fp}}^r)$ via scalar extension}\label{482}
Recall that any spin module for the Clifford algebra of a real reductive group can be realized, as in Section \ref{spinmodule}, on the exterior algebra $\bigwedge \mathfrak{p}^-$ of a maximal isotropic subspace $\fp^-$ of $\fp$ with respect to the form $\beta$. We denote the corresponding action map by 
\begin{equation}\nonumber
\gamma_{\mathfrak{p}^-,\beta}':Cl(\mathfrak{p},\beta)\longrightarrow \mathrm{End}_{\C}\left(\bigwedge \mathfrak{p}^-\right).
\end{equation}
Using scalar extension,  the obvious  map 
\begin{equation}\nonumber
\mathbb{I}_R\otimes  \gamma_{\mathfrak{p}^-,\beta}':Cl(R\otimes_{\C}\mathfrak{p},1_R\otimes\beta)\longrightarrow \mathrm{End}_{R}\left(\bigwedge (R\otimes_{\C}\mathfrak{p}^-)\right),
\end{equation}
equips $\bigwedge (R\otimes_{\C}\mathfrak{p}^-)$ with the structure of  a spin module for $Cl(R\otimes_{\C}\mathfrak{p},1_R\otimes\beta)$.

Recall that $\p=\langle r\rangle\otimes_{\C}\fp $. 
As in Section \ref{spinmodule}, in the even-dimensional case we have $\fp=\fp^-\oplus \fp^+$ and in the odd-dimensional case we have $\fp=\fp^-\oplus \fp^+\oplus \fp^0$. 
We set $\p^{\pm}:=\langle r\rangle\otimes_{\C}\fp^{\pm}$, and in the odd dimensional case we also define $\p^{0}:=\langle r\rangle\otimes_{\C}\fp^{0}$. 

  The isometry $T_r:(R\otimes_{\C} \fp,\beta_{\fp}) \longrightarrow (\boldsymbol{\fp} ,\boldsymbol{\beta}_{\boldsymbol{\fp}}^r)$ introduced in  Section \ref{The algebra A},  induces an isomorphism 
$T_r:R\otimes_{\C} {Cl}(\fp,{ \beta}_{\fp}) \longrightarrow  {Cl}(\boldsymbol{\fp},\boldsymbol{\beta}_{\boldsymbol{\fp}}^r)$. It also  induces an isomorphism $T_{r,\fp^-}:R\otimes_{\C} \bigwedge \mathfrak{p}^- \longrightarrow  \bigwedge \p^-$. Using these isomorphisms, we  equip $\bigwedge \p^-$ with the   structure of a spin module for $Cl(\boldsymbol{\fp},\boldsymbol{\beta}_{\boldsymbol{\fp}}^r)$   
\begin{eqnarray}\nonumber
&& \boldsymbol{\gamma}_{\boldsymbol{\fp}^-,\boldsymbol{\beta}_{\boldsymbol{\fp}}^r}':Cl(\boldsymbol{\fp},\boldsymbol{\beta}_{\boldsymbol{\fp}}^r)\longrightarrow \mathrm{End}_{R}\left(\bigwedge \p^-\right)
\end{eqnarray}
which is the unique map of $R$-algebras satisfing 
\[\left(\boldsymbol{\gamma}_{\boldsymbol{\fp}^-,\boldsymbol{\beta}_{\boldsymbol{\fp}}^r}'(X)\right)(Y)= T_{r,\fp^-}\left(\left(\left(\mathbb{I}_R\otimes  \gamma_{\mathfrak{p}^-,\beta}'\right)(T_r^{-1}(X))\right) (T_{r,\fp^-}^{-1}(Y))\right), \]
for every  $X\in \p. Y\in \bigwedge \p^-$. 
Equivalently, the diagram  
\[\xymatrix{
& Cl(\boldsymbol{\fp},\boldsymbol{\beta}_{\boldsymbol{\fp}}^r)\otimes_R \bigwedge \p^- \ar[r]^{a_{\boldsymbol{\beta}_{\boldsymbol{\fp}}^r}}
& \bigwedge \p^-  \\
&  \left(R\otimes_{\C} Cl(\mathfrak{p},\beta)\right)\otimes_R\left( R\otimes_{\C} \bigwedge \fp^-\right)\ar[d]^{m\otimes \mathbb{I}_{Cl(\mathfrak{p},\beta)}\otimes  \mathbb{I}_{\bigwedge \fp^-}}     \ar[u]^{T_{r}\otimes T_{r,\fp^-}} &\\
&R\otimes_{\C} Cl(\mathfrak{p},\beta)\otimes_{\C} \bigwedge \fp^-       \ar[r]^{\hspace{12mm}\mathbb{I}_R\otimes a_{{\beta}}} & R\otimes_{\C} \bigwedge \fp^-\ar[uu]_{T_{r,\fp^-}} }\]
with the upper horizontal map $a_{\boldsymbol{\beta}_{\boldsymbol{\fp}}^r}$ being  the action arising from $\boldsymbol{\gamma}_{\boldsymbol{\fp}^-,\boldsymbol{\beta}_{\boldsymbol{\fp}}^r}'$, the  map $a_{\beta}$ being  the action arising from  $\gamma_{\mathfrak{p}^-,\beta}'$, and $m:R\otimes_R R\longrightarrow R$ the obvious multiplication,  is a commutative diagram. 

We shall also use the notation $\boldsymbol{S}(\p,\boldsymbol{\beta}_{\boldsymbol{\fp}}^r,\p^-)$ for the spin module $\bigwedge \p^-$.

It is immediate to check that   the following  hold.
\begin{enumerate}
    \item For any $X\in \p^-$ and $Y\in \bigwedge \p^-$,  
\[\boldsymbol{\gamma}_{\boldsymbol{\fp}^-,\boldsymbol{\beta}_{\boldsymbol{\fp}}^r}'(\boldsymbol{\gamma}_{\p,\boldsymbol{\beta}_{\p}^r}(X))(Y)=X\wedge Y.\]
\item For any $X\in \p^+$,  $\boldsymbol{\gamma}_{\p,\boldsymbol{\beta}_{\p}^r}(X)$ acts on $\bigwedge \p^-$ as the unique  graded derivation of degree $-1$ satisfying 
\[\boldsymbol{\gamma}_{\boldsymbol{\fp}^-,\boldsymbol{\beta}_{\boldsymbol{\fp}}^r}'(\boldsymbol{\gamma}_{\p,\boldsymbol{\beta}_{\p}^r}(X))(Y)=\boldsymbol{\beta}_{\boldsymbol{\fp}}^r(X,Y), \quad \forall Y\in \p^-\subset\bigwedge \p^-. \] 
\item In the odd rank case, for any $X\in \fp^0$ and $Y\in \bigwedge \fp^-$ and $f\in R$ such that $f\otimes Y\in \bigwedge \p^-$
\[\boldsymbol{\gamma}_{\boldsymbol{\fp}^-,\boldsymbol{\beta}_{\boldsymbol{\fp}}^r}'(\boldsymbol{\gamma}_{\p,\boldsymbol{\beta}_{\p}^r}(r\otimes X))(f\otimes Y)= f\otimes {\gamma}_{{\fp}^-,{\beta}}'({\gamma}_{\fp,\beta}(X))(Y). \] 
\end{enumerate}

\subsubsection{The action of $\widetilde{K}$ on $\boldsymbol{S}$}
\label{dc}

In Section \ref{451} we constructed the  morphism of Lie algebras
\[\alpha_{\boldsymbol{\fp},\boldsymbol{\beta}_{\boldsymbol{\fp}}^r}=\varphi^r_{\boldsymbol{\fp}} \circ \operatorname{ad}^{\boldsymbol{\fk}}:\k \longrightarrow  Cl(\boldsymbol{\fp},\boldsymbol{\beta}_{\boldsymbol{\fp}}^r).\] 
By composing it with $\boldsymbol{\gamma}'_{\p^-,\boldsymbol{\beta}_{\p}^r}$
we obtain a representation of $\k$ on $\bigwedge \p^-$ and by restriction we obtain a representation of $1\otimes_{\C} \fk\simeq \fk$ via  endomorphisms of $R$-modules. 
\begin{lemma*}
     The representation   $\boldsymbol{\gamma}'_{\p^-,\boldsymbol{\beta}_{\p}^r}\circ \alpha_{\boldsymbol{\fp},\boldsymbol{\beta}_{\boldsymbol{\fp}}^r}|_{\fk}:\fk\longrightarrow \mathrm{End}_{R}\left(\bigwedge \p^-\right)$ can be lifted to a representation of the spin double cover $\widetilde{K}$ via automorphisms of $R$-modules.
\end{lemma*}

 \begin{proof}
By a direct calculation, for every $X\in \fk$ and $f\otimes Y\in \bigwedge (R\otimes_{\C}\fp^-)$,
\begin{eqnarray}\nonumber
&&    T_{r,\fp^-}^{-1} \left(\left(\boldsymbol{\gamma}'_{\p^-,\boldsymbol{\beta}_{\p}^r}\circ \alpha_{\boldsymbol{\fp},\boldsymbol{\beta}_{\boldsymbol{\fp}}^r}|_{\fk}(1\otimes X)\right)(T_{r,\fp^-}(f \otimes Y))\right)=\\ \nonumber
&&T_{r,\fp^-}^{-1} \left(a_{\boldsymbol{\beta}_{\boldsymbol{\fp}}^r}(\alpha_{\boldsymbol{\fp},\boldsymbol{\beta}_{\boldsymbol{\fp}}^r}(1\otimes X)\otimes T_{r,\fp^-}(f \otimes Y ))\right)= f\otimes a_{\beta}(\alpha_{{\fp},{\beta}_{{\fp}}^1}(X),Y).
\end{eqnarray}
Since the right hand side of the last equation defines a representation of $\fk$ on $\bigwedge (R\otimes_{\C}\fp^-)$ that is integrable to a representation of $\widetilde{K}$ via $R$-linear automorphisms and  since $T_{r,\fp^-}$ is a $\mathfrak{k}$-equivariant isomorphism the lemma follows. 
 \end{proof}
  For a later use we shall denote the $\widetilde{K}$ representation on $\bigwedge \p^-$  by 
  \[\pi_{\widetilde{K},\boldsymbol{S}}:\widetilde{K}\longrightarrow \operatorname{Aut}_R\left( \boldsymbol{S}(\p,\boldsymbol{\beta}_{\boldsymbol{\fp}}^r,\p^-)\right).\]

\subsection{Dirac cohomology}
In this subsection we introduce the Dirac cohomology of algebraic families of $(\g,K)-$modules.
\subsubsection{ From  $(\g,K)-$modules to $(\boldsymbol{A},\widetilde{K})$-modules}

Let $\boldsymbol{V}$ be   a $(\g,K)$-module  with actions of $K$ and $\g$ that are denoted by
\[\pi_{K,\boldsymbol{V}}:K\longrightarrow \operatorname{Aut}_{R}(\boldsymbol{V}),\quad \text{and}\quad \pi_{\g,\boldsymbol{V}}:\boldsymbol{\mathfrak{g}} \longrightarrow \operatorname{End}_{R}(\boldsymbol{V}),\]
respectively. 
The $R$-module $\boldsymbol{V}\otimes_R \boldsymbol{S}(\p,\boldsymbol{\beta}_{\p}^r,\p^-)$ is an $ A(\g,\beta,r) $-module via
\[\pi_{A,\boldsymbol{V}}:A(\g,\beta,r)\longrightarrow  \operatorname{End}_{R}(\boldsymbol{V}\otimes_R \boldsymbol{S}(\p,\boldsymbol{\beta}_{\p}^r,\p^-)),\]
satisfying 
\[  \left(\pi_{A,\boldsymbol{V}}(X\otimes Y)\right)(v\otimes s)=\pi_{\g,\boldsymbol{V}}(X)v\otimes \boldsymbol{\gamma}'_{\p^-,\boldsymbol{\beta}_{\p}^r}(Y)s, \]
for every $X\in \g$, $Y\in Cl(\boldsymbol{\fp},\boldsymbol{\beta}_{\boldsymbol{\fp}}^r)$, $v\in \boldsymbol{V}$ and $s\in \boldsymbol{S}(\p,\boldsymbol{\beta}_{\p}^r,\p^-)$.

 Using the cover map $\widetilde{K}\mapsto K$, $\boldsymbol{V}$ carries an action  of $\widetilde{K}$. As explained in \ref{dc}, the spin module $\boldsymbol{S}(\p,\boldsymbol{\beta}_{\p}^r,\p^-)$ also carries an action  of $\widetilde{K}$. Hence 
$\boldsymbol{V}\otimes_R \boldsymbol{S}(\p,\boldsymbol{\beta}_{\p}^r,\p^-)$  is a representation of $\widetilde{K}$ via automorphisms of $R$-modules as a tensor product of such. We denote this action by 
$\pi_{\widetilde{K},\boldsymbol{V}}:\widetilde{K}\longrightarrow \operatorname{Aut}_{R}(\boldsymbol{V}\otimes_R \boldsymbol{S}(\p,\boldsymbol{\beta}_{\p}^r,\p^-)).$
The actions are compatible in the sense that for every $X\in \fk$ and $z\in \boldsymbol{V}\otimes_R \boldsymbol{S}(\p,\boldsymbol{\beta}_{\p}^r,\p^-)$,
\[  \left(\pi_{A,\boldsymbol{V}} \circ \Delta_{(\g,\beta,r)}(1\otimes X)\right)(z)=d\pi_{\widetilde{K},\boldsymbol{V}}(X)(z).  \]
This make $\boldsymbol{V}\otimes_R \boldsymbol{S}(\p,\boldsymbol{\beta}_{\p}^r,\p^-)$  into a module for the algebraic family of generalized pairs $(A(\g,\beta,r),\widetilde{K})$. 
\subsubsection{The Dirac cohomology}
 The operator $\pi_{A,\boldsymbol{V}}(D(\g,\beta,r))\in \operatorname{End}_{R}(\boldsymbol{V}\otimes_R \boldsymbol{S}(\p,\boldsymbol{\beta}_{\p}^r,\p^-))$
   is $\widetilde{K}$-equivariant so that the quotient space  
 \begin{equation*}
	H_{D(\g,\beta,r)}(\boldsymbol{V}):=\frac{\ker \pi_{A,\boldsymbol{V}}(D(\g,\beta,r))}{\ker \pi_{A,\boldsymbol{V}}(D(\g,\beta,r))\cap \mathrm{im}\hspace{0.5mm}\pi_{A,\boldsymbol{V}}(D(\g,\beta,r))}
\end{equation*} 
is an $R$-module that carries an action  of $\widetilde{K}$ via automorphisms of $R$-modules.  This space is called \textit{the Dirac cohomology of $\boldsymbol{V}$}.

\section{Vogan's conjecture}\label{Vo}
In this section we formulate and prove  Vogan's conjecture, first for constant families and then for the variants of the deformation family introduced  in Section 
\ref{vdf}.

We keep our running assumptions and notations. In particular $G(\R)$ is a connected real reductive group.  The invariant form $\beta$ on $\fg$ is positive definite on $\fp^{\sigma}$ and negative definite on  $\fk^{\sigma}$ and its restriction to $[\fg,\fg]$ coincides with its  Killing form. 
We fix $\theta$-stable triangular decompositions 
\[\mathfrak{g}=\mathfrak{n}^-\oplus \mathfrak{h}\oplus \mathfrak{n}^+, \quad \mathfrak{k}=\mathfrak{n}_{\mathfrak{k}}^-\oplus \mathfrak{t}\oplus \mathfrak{n}^+_{\mathfrak{k}}, \]
with $\mathfrak{h}=\mathfrak{t}\oplus \mathfrak{a}$ being a fundamental Cartan subalgebra of $\fg$ and   $\mathfrak{n}^{\pm}_{\mathfrak{k}}\subseteq \mathfrak{n}^{\pm}$.

The ring $R$ is a principal ideal domain containing $\C$.

The Dirac cohomology $H_{D(\g,\beta,r)}(\boldsymbol{V})$ is always a complex representation of $\widetilde{K}$. 
As in Section \ref{VC} we shall parameterize the irreducible complex algebraic representations of $\widetilde{K}$ by integral dominant weights  in $\ft^*$. 
We shall say that the Dirac cohomology of $\boldsymbol{V}$ 
contains  a $\widetilde{K}$-type of highest weight $\mu\in \ft^*$ 
(with respect to our chosen positive system) if it does so as a complex representation of $\widetilde{K}$.

\subsection{Vogan's conjecture for constant families }\label{51}
In this subsection we formulate and prove Vogan's conjecture for constant families.

\begin{theorem*}
    Let $\boldsymbol{V}$  be a generically  irreducible $(R\otimes_{\C}\fg,K)$-module with non-zero Dirac cohomology. Then 
$\boldsymbol{V}$  has   an infinitesimal character $\boldsymbol{\lambda}\in (R\otimes_{\C}\mathfrak{h})^*$ with respect to $R\otimes_{\C}\fh$. 
Moreover,  if $\hspace{0.5mm}H_{D(R\otimes_{\C}\fg,\beta,1)}(\boldsymbol{V})$ contains a $\widetilde{K}$-type of highest weight $\mu\in\mathfrak{t}^*$, then $\boldsymbol{\lambda}$   is $W(\fg,\fh)$-conjugate to $\mathbb{I}_R\otimes (\mu+ \rho_\mathfrak{k})^{\ft}_{\fh}\in R^*\otimes_{\C}\fh^*= (R\otimes_{\C}\mathfrak{h})^*$.
\end{theorem*}

We shall prove the above theorem by adapting the proof of \cite[Theorem 3.2.5]{pandzic} to our context.  
\begin{proof}   
Since $\boldsymbol{V}$ is generically irreducible it has a central infinitesimal character $\chi_{\boldsymbol{V}}$.

According to \cite[Theorems 3.2.7, 3.2.8]{pandzic}, there exists a  morphism of complex algebras 
\begin{equation*}
	\zeta: \mathcal{Z}(\mathfrak{g})\rightarrow \mathcal{Z}(\mathfrak{k}_\Delta)
\end{equation*}
which fits into the commutative diagram 
 \[\xymatrix{
 \mathcal{Z}({{\mathfrak{g}}}) \ar[d]^{\zeta}\ar[rr]_{\operatorname{HC}_{\fh}^{\mathfrak{g}}}
&& \operatorname{S}({{\mathfrak{h}}})^{W(\fg,\fh)} \ar[d]^{\operatorname{Res}} \\
\mathcal{Z}(\mathfrak{k}_{\Delta}) \ar[r]^{\Delta_{\beta}^{-1}} &\mathcal{Z}(\mathfrak{k})\ar[r]^{\hspace{-4mm}  \operatorname{HC}_{\ft}^{\mathfrak{k}}} &  \operatorname{S}({{\mathfrak{t}}})^{W(\fk,\ft)} }\]
and such that for every $z\in \mathcal{Z}(\mathfrak{g})$, there are  $a,b\in A(\fg,\beta)$ such that 
\begin{equation*}
	z\otimes 1=\zeta(z)+Da+bD.
\end{equation*}
Here $\Delta_{\beta}$ stands for the isomorphism from  
$\mathcal{U}({{\mathfrak{k}}})$ onto its image inside   
${{A}}$ and also for its restriction to $\mathcal{Z}({{\mathfrak{k}}})$ 
which is mapped onto
$\mathcal{Z}({{\mathfrak{k}}}_{{{\Delta}}})
$.  
The map $\operatorname{Res}$  is obtained from the projection from $\fh$ onto $\ft$ with respect to the decomposition $\fh=\ft\oplus \mathfrak{a}$.

By scalar extension, that is, by applying the functor $R\otimes_{\C}\hspace{-1mm}\hspace{0.3mm}\_$, we obtain the commutative diagram of $R$-algebras
\[\xymatrix{
R\otimes_{\C}\mathcal{Z}( \fg) \ar[dd]^{\mathbb{I}_R\otimes {\zeta}}\ar[rrr]_{\mathbb{I}_R\otimes \operatorname{HC}_{\fh}^{\mathfrak{g}}} 
&& &R\otimes_{\C}\operatorname{S}(\fh)^{W(\fg,\fh)} \ar[dd]^{\mathbb{I}_R\otimes \operatorname{Res}} \\
&&&\\
R\otimes_{\C}\mathcal{Z}( \fk_{\Delta})\ar[rr]^{\hspace{0mm}\mathbb{I}_R\otimes {\Delta}_{\beta}^{-1} } &&R\otimes_{\C}\mathcal{Z}(\fk)\ar[r]^{ \hspace{-4mm}\mathbb{I}_R\otimes \operatorname{HC}_{\ft}^{\mathfrak{k}}} &  R\otimes_{\C}\operatorname{S}(\ft)^{W(\fk,\ft)} 
 }\]
which leads to an isomorphic commutative diagram 
\[\xymatrix{
\mathcal{Z}( R\otimes_{\C}\fg) \ar[dd]^{\mathbb{I}_R\otimes {\zeta}}\ar[rrr]_{\mathbb{I}_R\otimes \operatorname{HC}_{\fh}^{\mathfrak{g}}} 
&& &\operatorname{S}(R\otimes_{\C}\fh)^{W(\fg,\fh)} \ar[dd]^{\mathbb{I}_R\otimes \operatorname{Res}} \\
&&&\\
\mathcal{Z}( (R\otimes_{\C}\fk)_{\Delta_{(R\otimes_{\C}\fg,\beta,1)}})\ar[rr]^{\hspace{10mm}\Delta_{(R\otimes_{\C}\fg,\beta,1)}^{-1} } &&\mathcal{Z}(R\otimes_{\C}\fk)\ar[r]^{ \hspace{-4mm}\mathbb{I}_R\otimes \operatorname{HC}_{\ft}^{\mathfrak{k}}} &  \operatorname{S}(R\otimes_{\C}\ft)^{W(\fk,\ft)} 
 }\]
It follows that for every $z\in \mathcal{Z}(\fg)$ there are $a,b\in A(R\otimes_{\C}\fg,\beta,1)$ such that 
\begin{equation}\nonumber
		(1_R\otimes z)\otimes 1_{Cl(R\otimes_{\C}\fp,1_R\otimes {\beta})}=\left(\mathbb{I}_R\otimes {\zeta}\right)(1_R\otimes z)+D(R\otimes_{\C}\fg,\beta,1)a+bD (R\otimes_{\C}\fg,\beta,1).
	\end{equation}

If $x\in \operatorname{Ker}  \pi_{A,\boldsymbol{V}}(D(R\otimes_{\C}\fg,\beta,1))\setminus \operatorname{Im}  \pi_{A,\boldsymbol{V}}(D(R\otimes_{\C}\fg,\beta,1))$ is a highest weight  vector  for the action of $\fk$ of weight $\mu \in \ft^*$, then   
\begin{eqnarray}\nonumber
&& \pi_{A,\boldsymbol{V}}\left( D(R\otimes_{\C}\fg,\beta,1)a\right)x=\\ \nonumber
&& \pi_{A,\boldsymbol{V}}\left( D(R\otimes_{\C}\fg,\beta,1)a+bD (R\otimes_{\C}\fg,\beta,1) \right)x=\\ \nonumber
&&     \pi_{A,\boldsymbol{V}}\left((1_R\otimes z)\otimes 1_{Cl(R\otimes_{\C}\fp,1_R\otimes {\beta})}-\left(\mathbb{I}_R\otimes {\zeta}\right)(1_R\otimes z)\right)x=\\ \nonumber
&&  \chi_{\boldsymbol{V}}(1_R\otimes z)x     -\pi_{A,\boldsymbol{V}}\left(\Delta_{(R\otimes_{\C}\fg,\beta,1)}\circ \Delta_{(R\otimes_{\C}\fg,\beta,1)}^{-1}\left(\mathbb{I}_R\otimes {\zeta}\right)(1_R\otimes z)\right)x=\\ \nonumber
&&    \chi_{\boldsymbol{V}}(1_R\otimes z)x   -\pi_{A,\boldsymbol{V}}\left(\Delta_{(R\otimes_{\C}\fg,\beta,1)}\circ (\mathbb{I}_R\otimes \Delta_{\beta}^{-1})\left(\mathbb{I}_R\otimes {\zeta}\right)(1_R\otimes z)\right)x =\\ \nonumber
&&    \chi_{\boldsymbol{V}}(1_R\otimes z)x     -d\pi_{\widetilde{K},\boldsymbol{V}}( \Delta_{\beta}^{-1}(\zeta(z)))x =\\ \nonumber
&&    \chi_{\boldsymbol{V}}(1_R\otimes z)x    -( \widehat{1\otimes(\mu+\rho_{\fk})})(\mathbb{I}_R\otimes \operatorname{HC}_{\ft}^{\mathfrak{k}})(1_R\otimes \Delta_{\beta}^{-1}(\zeta(z)))x=\\ \nonumber
&&    \chi_{\boldsymbol{V}}(1_R\otimes z)x   -( \widehat{1\otimes(\mu+\rho_{\fk})})(\mathbb{I}_R\otimes \operatorname{Res}\circ \operatorname{HC}_{\fh}^{\mathfrak{g}})(1_R\otimes z)x,
\end{eqnarray}
where in the last equality we used the abovementioned commutative diagram. 
Since $x\notin \operatorname{Im}  \pi_{A,\boldsymbol{V}}(D(R\otimes_{\C}\fg,\beta,1))$, for every $z\in \mathcal{Z}(\fg)$,
\begin{eqnarray}\nonumber
&&    \chi_{\boldsymbol{V}}(1_R\otimes z)=( \widehat{\mathbb{I}_R\otimes(\mu+\rho_{\fk})})(\mathbb{I}_R\otimes \operatorname{Res}\circ \operatorname{HC}_{\fh}^{\mathfrak{g}})(1_R\otimes z)
\end{eqnarray}
which is equivalent to 
\begin{eqnarray}\nonumber
&    \chi_{\boldsymbol{V}} &=( \widehat{\mathbb{I}_R\otimes(\mu+\rho_{\fk})})(\mathbb{I}_R\otimes \operatorname{Res} )(\mathbb{I}_R\otimes   \operatorname{HC}_{\fh}^{\mathfrak{g}}) \\ \nonumber
&&=( \widehat{\mathbb{I}_R\otimes(\mu+\rho_{\fk})^{\ft}_{\fh}}) (\mathbb{I}_R\otimes   \operatorname{HC}_{\fh}^{\mathfrak{g}}).
\end{eqnarray}
Hence $\mathbb{I}_R\otimes (\mu+ \rho_\mathfrak{k})^{\ft}_{\fh}$ is an infinitesimal character of $\boldsymbol{V}$ with respect to $R\otimes_{\C}\fh$. 

 For the second statement, if  
$\boldsymbol{\lambda}\in (R\otimes_{\C}\mathfrak{h})^*$ is also an infinitesimal character of $\boldsymbol{V}$ with respect to $R\otimes_{\C}\fh$
we must have 

\begin{eqnarray}\nonumber
&&   \widehat{\boldsymbol{\lambda}}(\mathbb{I}_R\otimes \operatorname{HC}_{\fh}^{\mathfrak{g}}) =( \widehat{\mathbb{I}_R\otimes(\mu+\rho_{\fk})^{\ft}_{\fh}}) (\mathbb{I}_R\otimes   \operatorname{HC}_{\fh}^{\mathfrak{g}})
\end{eqnarray}
and 
$\boldsymbol{\lambda}$  must be  $W(\fg,\fh)$-conjugate to $\mathbb{I}_R\otimes (\mu+ \rho_\mathfrak{k})^{\ft}_{\fh}$. 
\end{proof}

\subsection{Vogan's conjecture for $(\g_{(n)},K)$}\label{52}
In this subsection we formulate and prove Vogan's conjecture for the families $(\g_{(n)},K)$. 

\begin{theorem*}
    Let $\boldsymbol{V}$  be a generically  irreducible and admissible  $(\g_{(n)},K)$-module 
    with non-zero Dirac cohomology. Then 
$\boldsymbol{V}$  has   an infinitesimal character $\boldsymbol{\lambda}\in \boldsymbol{\mathfrak{h}}_{(n)}^*$ with respect to the fundamental Cartan subfamily $\boldsymbol{\fh}_{(n)}$.
Moreover,  if $H_{D(\g_{(n)},\beta,t^n)}(\boldsymbol{V})$ contains a $\widetilde{K}$-type of highest weight $\mu\in\mathfrak{t}^*$, then $\boldsymbol{\lambda}$   is $W(\fg,\fh)$-conjugate to   $\mathbb{I}_R\otimes (\mu+ \rho_\mathfrak{k})^{\ft}_{\fh}\in  \boldsymbol{\mathfrak{h}}_{(n)}^*$.
\end{theorem*}

\begin{remark*}
   Note that the last theorem, implies that a generically  irreducible $(\g_{(n)},K)$-module 
    with non-zero Dirac cohomology must have a constant infinitesimal character with respect to a fundamental Cartan subfamily. 
\end{remark*} 
 
The main idea in the proof of the last theorem is to restrict the relevant families to $\C^{\times}$ where they are isomorphic to constant families. Then to use Vogan's conjecture for constant families which was proven in the previous section. The theorem then follows by a continuity argument.

Before proving the theorem we shall discuss  restriction and relevant needed results. Restriction in our algebraic context is realized via localization.  

\subsubsection{Preliminaries}
From now on, we let  $R=\C[t]$. 
Recall that as a direct sum of $K$-modules over $R$,
\[\g_{(n)}=\k\oplus \boldsymbol{\fp}_{(n)},\] 
where $\k=R\otimes_{\C}\fk$, $\boldsymbol{\fp}_{(n)}=I_n\otimes_{\C}\fp$, and  $I_n:=\langle t^n\rangle$. We fix once and for all the generator $t^n$ of the ideal $I_n$.

We let $R_0=\mathcal{S}^{-1}R=\C[t,t^{-1}]$ be the localization of $R$ with respect to the multiplicatively closed set $\mathcal{S}:=\{1,t,t^2,...\}$. We have a canonical inclusion  of rings $e_{R_0}^R:R\longrightarrow R_0$. Below  we shall not write the inclusion morphism explicitly. 
We shall repeatedly use the localization functor $\iota^{R}_{R_0}: \prescript{}{R}{\mathcal{M}}\longrightarrow \prescript{}{R_0}{\mathcal{M}}$  from the category of left $R$-modules to the category of left $R_0$-modules given by $\iota^{R}_{R_0}(M)=R_0\otimes_{R}M$. This is an exact functor that geometrically amounts to restriction from $\C$ to $\C^{\times}$. We denote  by $q^R_{R_0}$ the canonical map from $M$ to $\iota^{R}_{R_0}(M)$ sending $m$ to $1_{R_0}\otimes m$.

We shall freely use well-known properties of localizations of free modules over a principal ideal domain such as  
\begin{itemize}
    \item $\forall M,N \in \prescript{}{R}{\mathcal{M}}$,  $\forall \varphi \in \operatorname{Hom}_R(M,N)$, $q^R_{R_0}\circ \varphi  =\iota^{R}_{R_0}(\varphi)\circ q^R_{R_0}$. 
    \item $\forall M,N \in \prescript{}{R}{\mathcal{M}}$, $\iota^{R}_{R_0}(M\otimes_R N)\cong \iota^{R}_{R_0}(M)\otimes_{R_0} \iota^{R}_{R_0}(N)$.
        \item $\forall M,N \in \prescript{}{R}{\mathcal{M}}$, $\iota^{R}_{R_0}(\operatorname{Hom}_R(M,N))\cong \operatorname{Hom}_{R_0}(\iota^{R}_{R_0}(M),\iota^{R}_{R_0}(N))$.
\end{itemize}

We shall also need the related  (scalar extension) functors $\iota^{\C}_{R_0}: \prescript{}{\C}{\mathcal{M}}\longrightarrow \prescript{}{R_0}{\mathcal{M}}$  and $\iota^{\C}_{R}: \prescript{}{\C}{\mathcal{M}}\longrightarrow \prescript{}{R}{\mathcal{M}}$. 

For any $M\in \prescript{}{\C}{\mathcal{M}}$, we let $m_{0}:\iota^{R}_{R_0}\iota^{\C}_{R}(M)\longrightarrow \iota^{\C}_{R_0}(M)$ be the isomorphism of $R_0$-modules  given by 
$m_0(x\otimes y\otimes m)=x y \otimes m $ for every $x\in R_0$, $y\in R$, and $m\in M$. 
\subsubsection{The localized Harish-Chandra pairs}\label{lhcp}

For every $n\in \N_0$, we let $\iota_n:\g_{(n)}\longrightarrow \iota^{\C}_{R}(\fg)=\g_{(0)}$ be the inclusion map (of Lie algebras over $R$). Note that the corresponding map between the universal enveloping algebras was denoted by $\iota_{\g}$ in Section \ref{The algebra A}.
Clearly the diagram
\[\xymatrix{
\g_{(n)} \ar[d]^{q^R_{R_0}}\ar[rr]_{\iota_n} 
& &\iota^{\C}_{R}(\fg)\ar[d]^{q^R_{R_0}}  &\\
\iota_{R_0}^R(\g_{(n)})\ar[rr]^{\iota^R_{R_0}(\iota_n)}&&\iota_{R_0}^R(\iota^{\C}_{R}(\fg) )\ar[r]^{m_0}_{\cong}&\iota^{\C}_{R_0}(\fg)
 }\]
is commutative.
A similar  commutative diagram holds for the universal enveloping algebras.

We set $\psi:=m_0\circ \iota^R_{R_0}(\iota_n):\iota_{R_0}^R(\g_{(n)}) \longrightarrow \iota^{\C}_{R_0}(\fg)$, as a morphism of Lie algebras over $R_0$.
\begin{lemma*}
     The map $\psi$ is a $K$-equivariant  isomorphism of Lie algebras over $R_0$ and it induces an isomorphism of algebraic families of     Harish-Chandra pairs  from $(\iota^{R}_{R_0}(\g_{(n)}),K)$  to $(\iota^{\C}_{R_0}(\fg),K)$.
\end{lemma*}
The proof is by direct calculation.

 \subsubsection{The localized Harish-Chandra modules}\label{Ac}
Let $\boldsymbol{V}$ be   a $(\g_{(n)},K)$-module  with actions of $K$ and  $\g_{(n)}$ that are denoted by $\pi_{K,\boldsymbol{V}}:K\longrightarrow \operatorname{Aut}_{R}(\boldsymbol{V})$ and $\pi_{\g_{(n)},\boldsymbol{V}}:\boldsymbol{\mathfrak{g}}_{(n)} \longrightarrow \operatorname{End}_{R}(\boldsymbol{V})$  respectively. We shall denote the morphism $\g_{(n)}\otimes_R \boldsymbol{V}\longrightarrow \boldsymbol{V} $  associated  with $\pi_{\g_{(n)},\boldsymbol{V}}$ by $\widetilde{\pi}_{\g_{(n)},\boldsymbol{V}}$. That is, for every $X\in \g_{(n)}$ and $v\in \boldsymbol{V}$, we have
$\widetilde{\pi}_{\g_{(n)},\boldsymbol{V}}(X\otimes v)=\pi_{\g_{(n)},\boldsymbol{V}}(X)(v)$. We shall adopt this notation involving $\sim$  for all other relevant actions below.
 
Localization equips the $R_0$-module $\iota^{R}_{R_0}(\boldsymbol{V})$ with the structure of an $(\iota^{R}_{R_0}(\g_{(n)}),K)$-module  with actions of $K$, and $\iota^{R}_{R_0}(\g_{(n)})$ that we denote by 
$\pi_{K,\iota^{R}_{R_0}(\boldsymbol{V})}:K\longrightarrow \operatorname{Aut}_{R_0}(\iota^{R}_{R_0}(\boldsymbol{V}))$, and $\iota^{R}_{R_0}(\pi_{\g_{(n)},\boldsymbol{V}}):\iota^{R}_{R_0}(\g_{(n)}) \longrightarrow \operatorname{End}_{R_0}(\iota^{R}_{R_0}(\boldsymbol{V}))$. 
Using the isomorphism $\psi:\iota_{R_0}^R(\g_{(n)}) \longrightarrow \iota^{\C}_{R_0}(\fg)$ we equip 
$\iota^{R}_{R_0}(\boldsymbol{V})$ with the structure of an $(\iota^{\C}_{R_0}(\fg),K)$-module. All these morphisms fit into the commutative
 diagram below.
  \[\xymatrix{
&& \g_{(n)}\otimes_R \boldsymbol{V}\ar[ld]_{q^{R}_{R_0}\otimes q^{R}_{R_0}} \ar[r]^{\hspace{6mm}\widetilde{\pi}_{\g_{(n)},\boldsymbol{V}}}\ar[d]^{q^{R}_{R_0}}
&\boldsymbol{V}\ar[d]^{q^{R}_{R_0}} \\
\iota^{\C}_{R_0}(\fg)\otimes_{R_0}\iota^{R}_{R_0}(\boldsymbol{V} )\ar[r]^{\hspace{-3mm} \psi^{-1}\otimes \mathbb{I}}_{\cong}    &\iota^{R}_{R_0}(\g_{(n)})\otimes_{R_0}\iota^{R}_{R_0}(\boldsymbol{V} ) \ar[r]^{\cong}&\iota^{R}_{R_0}(\g_{(n)}\otimes_{R} \boldsymbol{V} )  \ar[r]^{\hspace{8mm}\iota^{R}_{R_0}(\widetilde{\pi}_{\g_{(n)},\boldsymbol{V}})} &\iota^{R}_{R_0}(\boldsymbol{V} ) }\]

Let $\fh=\ft\oplus \mathfrak{a}$ be a fundamental Cartan subalgebra of $\fg$ with $\ft=\fh\cap \fk$ and $\mathfrak{a}=\fh \cap \fp$. Let $\boldsymbol{\fh}_{(n)}=\t\oplus \boldsymbol{\mathfrak{a}}_{(n)}$ be the corresponding  fundamental Cartan subfamily of $\g_{(n)}$, and let $\iota^{R}_{R_0}(\boldsymbol{\fh}_{(n)})=\iota^{R}_{R_0}(\t)\oplus \iota^{R}_{R_0}(\boldsymbol{\mathfrak{a}}_{(n)})$ be the corresponding  fundamental Cartan subfamily of $\iota^{R}_{R_0}(\g_{(n)})$. The image of $\iota^{R}_{R_0}(\boldsymbol{\fh}_{(n)})$ under $\psi$ is the Cartan subfamily $\iota_{R_0}^{\C}(\fh)$ of the constant family  $\iota_{R_0}^{\C}(\fg)$.


\begin{lemma*}
    Let $\boldsymbol{V}$  be an algebraic family of  $(\g_{(n)},K)$-modules.
    \begin{enumerate}
     \item If $\boldsymbol{V}$ is generically irreducible then so is $\iota^{R}_{R_0}(\boldsymbol{V})$.
        \item If $\boldsymbol{V}$ is quasi-simple with central infinitesimal character $\chi_{\g_{(n)},\boldsymbol{V}}:\mathcal{Z}(\g_{(n)})\longrightarrow R$, then $\iota^{R}_{R_0}(\boldsymbol{V})$ is a quasi-simple $(\iota^{\C}_{R_0}(\fg),K)$-module  with central infinitesimal character $\iota^{R}_{R_0}(\chi_{\g_{(n)},\boldsymbol{V}})\circ \psi^{-1}:\iota^{\C}_{R_0}(\mathcal{Z}(\fg))\longrightarrow R_0\otimes_{R}R\cong R_0$. 
        \item  If $\boldsymbol{V}$ has 
    an  infinitesimal character $\boldsymbol{\lambda}\in \boldsymbol{\mathfrak{h}}_{(n)}^*$  with respect to $\boldsymbol{\mathfrak{h}}_{(n)}$, then $\iota^{R}_{R_0}(\boldsymbol{V})$ has an infinitesimal character $  \iota^R_{R_0}(\boldsymbol{\lambda})\circ   \psi^{-1} \in \iota_{R_0}^{\C}(\fh^*)\cong \iota_{R_0}^{\C}(\fh)^* $ with respect to $\iota_{R_0}^{\C}(\fh)$.
    \item If $\boldsymbol{V}$ is admissible then so is $\iota^{R}_{R_0}(\boldsymbol{V})$. 
    \end{enumerate}
\end{lemma*}

    \begin{proof}
For any $z\in \C^{\times}=\operatorname{Spec}_m(R_0)$,
\begin{eqnarray}\nonumber
    && \iota^{R}_{R_0}(\g_{(n)})|_{z}\cong \frac{\C[t,t^{-1}]}{(t-z)\C[t,t^{-1}] } \otimes_{\C[t]}  \g_{(n)}\cong  \frac{\C[t]}{(t-z)\C[t] } \otimes_{\C[t]}  \g_{(n)}\cong \g_{(n)}|_z\\ \nonumber
    && \iota^{R}_{R_0}(\boldsymbol{V})|_{z}\cong \frac{\C[t,t^{-1}]}{(t-z)\C[t,t^{-1}] }\otimes_{\C[t]} \boldsymbol{V}\cong    \frac{\C[t]}{(t-z)\C[t]}\otimes_{\C[t]}  \boldsymbol{V}= \boldsymbol{V}|_{z}
\end{eqnarray}
Since for almost all $z$,  $\boldsymbol{V}|_{z}$  is an irreducible $(\g_{(n)}|_z,K)$-module,   then for almost all $z$, $\iota^{R}_{R_0}(\boldsymbol{V})|_{z}$ is an irreducible $( \iota^{R}_{R_0}(\g_{(n)})|_{z},K)$-module and hence also an irreducible $(\iota^{\C}_{R_0}(\fg)|_z,K)$-module. 

Claim 2. about the central infinitesimal character follows from the last commutative diagram above.

 For Claim 3. it suffices to calculate $\iota^{R}_{R_0}({\pi}_{\g_{(n)},\boldsymbol{V}})\psi^{-1}(f\otimes\xi))(g\otimes v)$
for any  $f\otimes \xi \in \iota^{\C}_{R_0}(\mathcal{Z}(\fg))=\mathcal{Z}(\iota^{\C}_{R_0}(\fg))$ and $g\otimes v\in \iota^{R}_{R_0}(\boldsymbol{V})$.
 We first note that for any $m\in \mathbb{N}$ large enough  such that $t^m\otimes \xi \in \g_{(n)}$, we have $\psi^{-1}(f\otimes \xi)=t^{-m}f\otimes t^m \otimes \xi$ and hence
\begin{eqnarray}\nonumber
&& \iota^{R}_{R_0}({\pi}_{\g_{(n)},\boldsymbol{V}})\psi^{-1}(f\otimes\xi))(g\otimes v)= \iota^{R}_{R_0}({\pi}_{\g_{(n)},\boldsymbol{V}})( t^{-m}f\otimes t^m \otimes \xi)(g\otimes v)=\\ \nonumber
   && t^{-m}fg\otimes \left({\pi}_{\g_{(n)},\boldsymbol{V}}  (t^m \otimes\xi)(v)\right)=t^{-m}fg\otimes \left(\left(\widehat{\boldsymbol{\lambda}}\circ \boldsymbol{\operatorname{HC}_{\fh_{(n)}}^{\mathfrak{g}_{(n)}}}(t^m\otimes \xi)\right)v\right)=\\ \nonumber 
&&\left(\iota^R_{R_0}(\widehat{\boldsymbol{\lambda}})\left(t^{-m}f\otimes  \boldsymbol{\operatorname{HC}_{\fh_{(n)}}^{\mathfrak{g}_{(n)}}}(t^m\otimes \xi) \right) \right)(g\otimes v)=\\ \nonumber
    &&  \left(\iota^R_{R_0}(\widehat{\boldsymbol{\lambda}}) \iota^R_{R_0}(\boldsymbol{\operatorname{HC}_{\fh_{(n)}}^{\mathfrak{g}_{(n)}}})(t^{-m}f\otimes t^m\otimes \xi)\right)(g\otimes v)=\\ \nonumber 
    &&  \left(\iota^R_{R_0}(\widehat{\boldsymbol{\lambda}}) \iota^R_{R_0}(\boldsymbol{\operatorname{HC}_{\fh_{(n)}}^{\mathfrak{g}_{(n)}}})\psi^{-1}(f\otimes \xi)\right)(g\otimes v)=\\ \nonumber 
    &&\left(\widehat{\iota^R_{R_0}(\boldsymbol{\lambda}) \psi^{-1}} \circ \boldsymbol{\operatorname{HC}}^{\iota^{\C}_{R_0}(\fg)}_{\iota^{\C}_{R_0}(\fh)}(f\otimes\xi)\right)(g\otimes v).
\end{eqnarray}

Claim 4. follows from the following fact: For every irreducible representation $\tau$ of $K$, $\iota^{R}_{R_0}(\boldsymbol{V})_{\tau}\cong \iota^{R}_{R_0}(\boldsymbol{V}_{\tau})$, where the subscript $\tau$ stands for the $K$-isotypic component corresponding  to $\tau$.
 \end{proof}

\subsubsection{The localized Clifford algebras}
 Recall that $\beta$ is a symmetric bilinear form on $\fg$ satisfying our running  assumptions.   
 To simplify notation, we denote the form $\boldsymbol{\beta}^{t^n}_{\p_{(n)}}$ on $\p_{(n)}$ by $\boldsymbol{\beta}^n$. This is the orthogonalizable form on $\p_{(n)}$ satisfying $1_R\otimes \beta = t^{2n}\boldsymbol{\beta}^n $. 
 
In Section \ref{compform} we introduced   the $K$-equivariant isomorphism of quadratic spaces     $T_r:(\fp,\beta_{\fp})\longrightarrow (\fp_r,\boldsymbol{\beta}_{\p}^r|_{\fp_r})$ given by multiplication by $r$. From now on, for every $n\in \N_0$, we  shall use   $T_n$ as a shorthand notation for $T_{t^n}$.
 By abuse of notation, we shall also use $T_n$ for the isometry $(R\otimes_{\C}\fp,1_R\otimes \beta)\longrightarrow (\p_{(n)}, \boldsymbol{\beta}^n)$, and also for the corresponding isomorphism between the Clifford algebras. 
\begin{lemma*}
   The localized morphism $\iota_{R_0}^R(T_n)$ is a $K$-equivariant  isomorphism between the localized Clifford algebras, and moreover  
   \[\xymatrix{
Cl(\iota^{\C}_{R}(\fp),1_{R}\otimes \beta)\cong \iota^{\C}_{R}(Cl(\fp,\beta)) \ar[rr]_{T_n} \ar[d]^{q_{R_0}^R}
& & Cl(\p_{(n)}, \boldsymbol{\beta}^n))\ar[d]^{q_{R_0}^R}  \\
Cl(\iota^{\C}_{R_0}(\fp),1_{R_0}\otimes \beta)\cong \iota^{\C}_{R_0}(Cl(\fp,\beta))\ar[rr]^{\iota^R_{R_0}(T_n)}  && \iota^{R}_{R_0}(Cl(\p_{(n)}, \boldsymbol{\beta}^n))\cong 
Cl(\iota^{R}_{R_0}(\p_{(n)}),1_{R_0}\otimes \boldsymbol{\beta}^n)
 }\]
 is a commutative diagram.
\end{lemma*}

The proof immediately follows from properties of localization.

\subsubsection{The localized spin module}\label{locspin}
As we saw in  Section \ref{482}, by  fixing a maximal isotropic subspace $\fp^-$ of $\fp$  we obtain  spin modules  $\bigwedge (\iota^{\C}_{R}(\fp^-))\cong \iota^{\C}_{R}(\bigwedge (\fp^-))$ of $Cl(\iota^{\C}_{R}(\fp),1_{R}\otimes \beta)\cong \iota^{\C}_{R}(Cl(\fp,\beta))$ and  $\bigwedge \p_{(n)}^-$ of $Cl(\p_{(n)},\boldsymbol{\beta}^n)$, and 
 a $\widetilde{K}$-equivariant  isomorphism $ T_{n,\fp^-}:\iota^{\C}_{R}(\bigwedge (\fp^-))
  \longrightarrow     \bigwedge \p_{(n)}^-$  such that 
  \[\xymatrix{
& Cl(\p_{(n)},\boldsymbol{\beta}^n)\otimes_R \bigwedge \p_{(n)}^- \ar[r]^{\hspace{12mm}a_{\boldsymbol{\beta} ^n}}
& \bigwedge \p_{(n)}^-  \\
&\iota^{\C}_{R}( Cl(\mathfrak{p},\beta))\otimes_{R} \iota^{\C}_{R}(\bigwedge \fp^-)\cong \iota^{\C}_{R}( Cl(\mathfrak{p},\beta)\otimes_{\C} \bigwedge \fp^-)\ar[u]^{T_{n}\otimes T_{n,\fp^-}}\ar[r]^{\hspace{36mm}\iota^{\C}_{R}(a_{\beta})} &\iota^{\C}_{R} (\bigwedge \fp^-)\ar[u]_{T_{n,\fp^-}}}\]
is a commutative diagram. Upon localization we obtain the commutative diagram 
\[\xymatrix{
& Cl(\p_{(n)},\boldsymbol{\beta}^n)\otimes_R \bigwedge \p_{(n)}^- \ar[r]^{\hspace{12mm}a_{\boldsymbol{\beta} ^n}}\ar[d]^{q^{R}_{R_0}}
& \bigwedge \p_{(n)}^-\ar[d]^{q^{R}_{R_0}}  \\
&  \iota^R_{R_0}(Cl(\p_{(n)},\boldsymbol{\beta}^n)\otimes_R \bigwedge \p_{(n)}^-) \ar[r]^{\hspace{14mm} \iota^R_{R_0}(a_{\boldsymbol{\beta} ^n})}
&  \iota^R_{R_0}(\bigwedge \p_{(n)}^-)  \\
&\iota^{\C}_{R_0}( Cl(\mathfrak{p},\beta))\otimes_{R_0} \iota^{\C}_{R_0}(\bigwedge \fp^-)\cong\iota^{\C}_{R_0}( Cl(\mathfrak{p},\beta)\otimes_{\C} \bigwedge \fp^-)\ar[u]^{ \iota^R_{R_0}(T_{n}\otimes T_{n,\fp^-})}\ar[r]^{\hspace{36mm}\iota^{\C}_{R_0}(a_{\beta})} &\iota^{\C}_{R_0} (\bigwedge \fp^-)\ar[u]_{ \iota^R_{R_0}(T_{n,\fp^-})}}\]
where $\iota^R_{R_0}(T_{n}\otimes T_{n,\fp^-})$ and $ \iota^R_{R_0}(T_{n,\fp^-})$ are $\widetilde{K}$-equivariant isomorphisms. 
Note that we used  the fact that $\iota^{R}_{R_0}\iota^{\C}_{R}(M)\cong\iota^{\C}_{R_0}(M)$, for every $M\in \prescript{}{\C}{\mathcal{M}}$

\subsubsection{The localized generalized pair}\label{526}
We set
\[A(\iota^{R}_{R_0} (\g_{(n)})):=\mathcal{U}(\iota^{R}_{R_0}(\g_{(n)}))\otimes_{R_0}Cl(\iota^{R}_{R_0}(\p_{(n)}),1_{R_0}\otimes \boldsymbol{\beta}^n). \]
Clearly $A(\iota^{R}_{R_0} (\g_{(n)}))\cong  \iota^{R}_{R_0} (A(\g_{(n)},\beta,t^n))$ and 
$A(\iota^{\C}_{R_0}(\fg),\beta,1)\cong \iota^{\C}_{R_0}( A(\fg,\beta))$.


\begin{corollary*}
The map $\iota_{R_0}^{R} (\iota_nT_n^{-1})$ is a $K$-equivariant isomorphism of $R_0$-algebras such that the following diagram is commutative 
\[\xymatrix{
A(\g_{(n)},\beta,t^n)\ar[d]^{q^{R}_{R_0}}\ar[rr]^{\iota_{n}T_n^{-1}}&&R\otimes_{\C} A(\fg,\beta,1)  
    \ar[d]^{q^{R}_{R_0}}\\
 \iota^R_{R_0}(A(\g_{(n)},\beta,t^n))\ar[rr]^{\iota_{R_0}^{R} (\iota_nT_n^{-1})}&& R_0\otimes_{\C} A(\fg,\beta,1) }\]
 and in particular $q^{R}_{R_0}((\iota_{n}T_n^{-1})D(\g_{(n)},\beta,t^n))=t^n\otimes D_{\fg,\beta}\in R_0\otimes_{\C} A(\fg,\beta,1)$.
\end{corollary*}
\begin{proof}
    Commutativity of the diagram follows from properties of localization. Since  $\iota^R_{R_0}(\iota_n)$ and $\iota^R_{R_0}(T_n)$ are $K$-equivariant isomorphisms then so is $\iota_{R_0}^{R} (\iota_nT_n^{-1})$. 
    It follows from Section \ref{DO}  that $\iota_nT_n^{-1}(D(\g,\beta,t^n))=t^n\otimes D_{\fg,\beta}\in R\otimes_{\C} A(\fg,\beta,1)$.
\end{proof}
\begin{definition*}
    We shall call $q^{R}_{R_0}(D(\g_{(n)},\beta,t^n))$ the \textit{Dirac operator of the localized  algebra $ \iota^R_{R_0}(A(\g_{(n)},\beta,t^n))$.}
\end{definition*}
 
\begin{remark*}
It can be shown that the Dirac operator  $D(\iota^{R}_{R_0}(\g_{(n)}),1_{R_0}\otimes \beta,1)$ of $A(\iota^{R}_{R_0} (\g_{(n)}))=\mathcal{U}(\iota^{R}_{R_0}(\g_{(n)}))\otimes_{R_0}Cl(\iota^{R}_{R_0}(\p_{(n)}),1_{R_0}\otimes \boldsymbol{\beta}^n)$ which is obtained in the usual way as in Section \ref{DO} via canonical sections, corresponds under the isomorphism $A(\iota^{R}_{R_0} (\g_{(n)}))\cong  \iota^{R}_{R_0} (A(\g_{(n)},\beta,t^n))$ to the Dirac operator $q^{R}_{R_0}(D(\g_{(n)},\beta,t^n))\in \iota^R_{R_0}(A(\g_{(n)},\beta,t^n))$. 
\end{remark*}

\subsubsection{The localized  modules of  the generalized pair}
Let $\boldsymbol{V}$ be   a $(\g_{(n)},K)-$module. We shall keep the notation for the actions as above  in  Section \ref{Ac}.   We shall denote the algebra  $A(\g_{(n)},\beta,t^n)$ by $\boldsymbol{A}_{(n)}$. 
The $R$-module $\boldsymbol{V}\otimes_R\bigwedge \p_{(n)}^-$ is an 
 $(\boldsymbol{A}_{(n)},\widetilde{K})$-module  with actions that we  denote by 
\[\pi_{\boldsymbol{A}_{(n)},\boldsymbol{V}}:\boldsymbol{A}_{(n)}\longrightarrow  \operatorname{End}_{R}(\boldsymbol{V}\otimes_R\bigwedge \p_{(n)}^-),\]
\[\pi_{\widetilde{K},\boldsymbol{V}}:\widetilde{K}\longrightarrow \operatorname{Aut}_{R}(\boldsymbol{V}\otimes_R\bigwedge \p_{(n)}^-).\]
We shall denote the morphism $\boldsymbol{A}_{(n)}\otimes_R \boldsymbol{V}\otimes_R\bigwedge \p_{(n)}^-\longrightarrow \boldsymbol{V}\otimes_R\bigwedge \p_{(n)}^- $  associated  with $\pi_{\boldsymbol{A}_{(n)},\boldsymbol{V}}$ by $\widetilde{\pi}_{\boldsymbol{A}_{(n)},\boldsymbol{V}}$.

The $R_0$-module 
$\iota^{R}_{R_0}(\boldsymbol{V}\otimes_R\bigwedge \p_{(n)}^-)$ is an 
 $(\iota^{R}_{R_0}(\boldsymbol{A}_{(n)}),\widetilde{K})$-module  with actions arising from localization. We  denote them  by 
\[\pi_{\iota^{R}_{R_0}(\boldsymbol{A}_{(n)}),\iota^{R}_{R_0}(\boldsymbol{V})}:\iota^{R}_{R_0}(\boldsymbol{A}_{(n)})\longrightarrow  \operatorname{End}_{R_0}(\iota^{R}_{R_0}(\boldsymbol{V}\otimes_R\bigwedge \p_{(n)}^-)),\]
\[\pi_{\widetilde{K},\iota^{R}_{R_0}(\boldsymbol{V})}:\widetilde{K}\longrightarrow \operatorname{Aut}_{R_0}(\iota^{R}_{R_0}(\boldsymbol{V})).\]
Note that $\pi_{\iota^{R}_{R_0}(\boldsymbol{A}_{(n)}),\iota^{R}_{R_0}(\boldsymbol{V})}$ is equal to the composition  
\[\xymatrix{
\iota^{R}_{R_0}(\boldsymbol{A}_{(n)})\ar[rr]^{\hspace{-10mm}\iota^{R}_{R_0}(\pi_{\boldsymbol{A}_{(n)},\boldsymbol{V}})}  &&\iota^{R}_{R_0}(\operatorname{End}_{R}(\boldsymbol{V}\otimes_R\bigwedge \p_{(n)}^-))      \ar[r]^{\cong}&\operatorname{End}_{R_0}(\iota^{R}_{R_0}(\boldsymbol{V}\otimes_R\bigwedge \p_{(n)}^-))       }\]
This induces a morphism 
 $ \iota_{R}^{R_0}(\boldsymbol{A}_{(n)})\otimes_{R_0} \iota_{R}^{R_0}(\boldsymbol{V}\otimes_R\bigwedge \p_{(n)}^-)\longrightarrow \iota_{R}^{R_0}(\boldsymbol{V}\otimes_R\bigwedge \p_{(n)}^- )$ that we denote by 
$\widetilde{\pi}_{\iota^{R}_{R_0}(\boldsymbol{A}_{(n)},\boldsymbol{V})}$.

Localization implies that  the diagram below is commutative   

   \[\xymatrix{
\boldsymbol{A}_{(n)}\otimes_R \boldsymbol{V}\otimes_R\bigwedge \p_{(n)}^- \ar[rr]^{\hspace{12mm}\widetilde{\pi}_{\boldsymbol{A}_{(n)},\boldsymbol{V}}}\ar[d]^{q^{R}_{R_0}}
&&\boldsymbol{V}\otimes_R\bigwedge \p_{(n)}^-\ar[d]^{q^{R}_{R_0}} \\
\iota^{R}_{R_0}(\boldsymbol{A}_{(n)})\otimes_{R_0}\iota^{R}_{R_0}(\boldsymbol{V}\otimes_R\bigwedge \p_{(n)}^- )  \ar[rr]^{\hspace{12mm} \widetilde{\pi}_{\iota^{R}_{R_0}(\boldsymbol{A}_{(n)},\boldsymbol{V})}} &&\iota^{R}_{R_0}(\boldsymbol{V}\otimes_R\bigwedge \p_{(n)}^- ) }\]
where all the arrows are $\widetilde{K}$-equivariant. 

Using the isomorphism $\iota_{R_0}^{R} (\iota_nT_n^{-1}):\iota^R_{R_0}(A(\g_{(n)},\beta,t^n))\longrightarrow R_0\otimes_{\C} A(\fg,\beta,1)$ we   equip $\iota^{R}_{R_0}(\boldsymbol{V}\otimes_R\bigwedge \p_{(n)}^- )$ with the structure of a module for the generalized pair  $(R_0\otimes_{\C} A(\fg,\beta,1),\widetilde{K})$.  
The last commutative diagram of Section \ref{locspin} implies that  this $(R_0\otimes_{\C} A(\fg,\beta,1),\widetilde{K})$-module  is the one arising from $\iota^{R}_{R_0}(\boldsymbol{V})$ as a $(R_0\otimes_{\C}  \fg, {K})$-module. This is further explained in the commutative diagram below.

   \[\xymatrix{
\boldsymbol{A}_{(n)}\otimes_R \boldsymbol{V}\otimes_R\bigwedge \p_{(n)}^- \ar[rr]^{\hspace{12mm}\widetilde{\pi}_{\boldsymbol{A}_{(n)},\boldsymbol{V}}}\ar[d]^{q^{R}_{R_0}}
&&\boldsymbol{V}\otimes_R\bigwedge \p_{(n)}^-\ar[d]^{q^{R}_{R_0}} \\
\iota^{R}_{R_0}(\boldsymbol{A}_{(n)})\otimes_{R_0}\iota^{R}_{R_0}(\boldsymbol{V}\otimes_R\bigwedge \p_{(n)}^- )\ar[d]_{\cong}^{\iota_{R_0}^{R} (\iota_nT_n^{-1})\otimes \mathbb{I}}  \ar[rr]^{\hspace{12mm} \widetilde{\pi}_{\iota^{R}_{R_0}(\boldsymbol{A}_{(n)},\boldsymbol{V})}} &&\iota^{R}_{R_0}(\boldsymbol{V}\otimes_R\bigwedge \p_{(n)}^- )\ar[d]_{\cong} \\
\left(R_0\otimes_{\C} A(\fg,\beta,1)\right)\otimes_{R_0}\iota^{R}_{R_0}(\boldsymbol{V})\otimes_{R_0}\iota^{R}_{R_0}(\bigwedge \p_{(n)}^- )\ar[d]_{\cong}^{\mathbb{I}\otimes\mathbb{I}\otimes  \iota_{R_0}^{R} ( T_{n,\fp^-}^{-1}) } && \iota^{R}_{R_0}(\boldsymbol{V})\otimes_{R_0}\iota^{R}_{R_0}(\bigwedge \p_{(n)}^- ) \ar[d]_{\cong}^{\mathbb{I}\otimes  \iota_{R_0}^{R} ( T_{n,\fp^-}^{-1}) }\\
\left(R_0\otimes_{\C} A(\fg,\beta,1)\right)\otimes_{R_0}\iota^{R}_{R_0}(\boldsymbol{V})\otimes_{R_0}\iota^{\C}_{R_0}(\bigwedge \fp^- )\ar@{.>}[rr]^{\hspace{20mm}\widetilde{\pi}_{\iota^{\C}_{R_0}(A),\iota^R_{R_0}(\boldsymbol{V})}}&&\iota^{R}_{R_0}(\boldsymbol{V})\otimes_{R_0}\iota^{\C}_{R_0}(\bigwedge \fp^- ) }\]
Note that the morphism $\widetilde{\pi}_{\iota^{\C}_{R_0}(A),\iota^R_{R_0}(\boldsymbol{V})}$ (the dotted arrow above) is uniquely  defined via the commutativity of the diagram.

\subsubsection{Localization and Dirac cohomology}\label{LDC}
Specializing the last commutative diagram to the case in which we fix the element of $\boldsymbol{A}_{(n)}$ to be the Dirac operator and taking into account Corollary \ref{526}   leads to the following  commutative diagram.

\[\xymatrix{
  \boldsymbol{V}\otimes_R\bigwedge \p_{(n)}^- \ar[rrr]^{\hspace{2mm}\pi_{\boldsymbol{A}_{(n)},\boldsymbol{V}}(D(\g_{(n)},\beta,t^n))}\ar[d]^{q^{R}_{R_0}}
&&&\boldsymbol{V}\otimes_R\bigwedge \p_{(n)}^-\ar[d]^{q^{R}_{R_0}} \\
 \iota^{R}_{R_0}(\boldsymbol{V}\otimes_R\bigwedge \p_{(n)}^- )  \ar[rrr]^{\hspace{2mm}\pi_{\iota^{R}_{R_0}(\boldsymbol{A}_{(n)}),\iota^{R}_{R_0}(\boldsymbol{V})}  (q^{R}_{R_0}D(\g_{(n)},\beta,t^n))}\ar[d]^{\cong} &&&\iota^{R}_{R_0}(\boldsymbol{V}\otimes_R\bigwedge \p_{(n)}^- )\ar[d]^{\cong}  \\
 \iota^{R}_{R_0}(\boldsymbol{V})\otimes_{R_0}\iota^{R}_{R_0}(\bigwedge \p_{(n)}^- )\ar[rrr]^{\hspace{-5mm}{\pi}_{\iota^{\C}_{R_0}(A),\iota^R_{R_0}(\boldsymbol{V})}(t^n\otimes D_{\fg,\beta})} &&& \iota^{R}_{R_0}(\boldsymbol{V})\otimes_{R_0}\iota^{R}_{R_0}(\bigwedge \p_{(n)}^- )
 }\]

\begin{proposition*}
    Let $\boldsymbol{V}$  be a  $(\g_{(n)},K)$-module.
There is a canonical $\widetilde{K}$-equivariant isomorphism of $R_0$-modules
\[ \iota^R_{R_0}(H_{D(\g_{(n)},\beta,t^n)}(\boldsymbol{V}))\cong H_{q^{R}_{R_0}(D(\g_{(n)},\beta,t^n))}(\iota^{R}_{R_0}(\boldsymbol{V}))\]
\end{proposition*}
 \begin{proof}
 From the  last commutative diagram we see that the $\widetilde{K}$-equivariant map $q^{R}_{R_0}:\boldsymbol{V}\otimes_R\bigwedge \p_{(n)}^-\longrightarrow \iota^R_{R_0}(\boldsymbol{V}\otimes_R\bigwedge \p_{(n)}^-)$ 
maps  $\operatorname{Ker}(\pi_{\boldsymbol{A}_{(n)},\boldsymbol{V}}(D(\g_{(n)},\beta,t^n)))$ 
 into $\operatorname{Ker}(\pi_{\iota^{R}_{R_0}(\boldsymbol{A}_{(n)}),\iota^{R}_{R_0}(\boldsymbol{V})}  (q^{R}_{R_0}D(\g_{(n)},\beta,t^n)))$. Since for free modules over a principal ideal domain localization commutes with kernel, image and intersection, the map between the abovementioned kernels induces the 
  required  isomorphism. 
 \end{proof}

 \subsubsection{Proof of Vogan's conjecture for families} 
In this section  we prove Theorem \ref{52}.

\begin{proof}[Proof of Theorem \ref{52} ]
Since $\boldsymbol{V}$ is generically irreducible and admissible it has a central infinitesimal character $\chi_{\g_{(n)},\boldsymbol{V}}:\mathcal{Z}(\g_{(n)})\longrightarrow R$.  By Lemma \ref{Ac},
    $\iota^{R}_{R_0}(\boldsymbol{V})$  is a generically irreducible $(\iota^{\C}_{R_0}(\fg),K)$-module  with central infinitesimal character $\iota^{R}_{R_0}(\chi_{ \g_{(n)},\boldsymbol{V}})\circ \psi^{-1}:\iota^{\C}_{R_0}(\mathcal{Z}(\fg))\longrightarrow R_0\otimes_{R}R\cong R_0$.  We shall denote the isomorphism sending $r\otimes s\in R_0\otimes_{R}R$ to  $rs\in R_0$ by $m_1$.


We note that for every $f\otimes \xi\in  \iota_{R_0}^{\C}(\mathcal{Z}(\fg))  =   \mathcal{Z}(\iota_{R_0}^{\C}(\fg))$, 
and every $(g_1\otimes v)\otimes (g_2\otimes  s)\in \iota^{R}_{R_0}(\boldsymbol{V})\otimes_{R_0}\iota^{\C}_{R_0}(\bigwedge \fp^- ) $,
\begin{eqnarray}\nonumber
&& {\pi}_{\iota^{\C}_{R_0}(A),\iota^R_{R_0}(\boldsymbol{V})}(f\otimes \xi \otimes 1_{ (Cl(\fp,\beta))}) \big((g_1\otimes v)\otimes (g_2\otimes  s)\big)=   \\ \nonumber 
&& \iota^R_{R_0}(\chi_{ \g_{(n)},\boldsymbol{V}})(\psi^{-1}(f\otimes \xi))(g_1\otimes  v)\otimes (g_2\otimes  s).
\end{eqnarray}
Recall that for any $m\in \mathbb{N}$ large enough such that $t^m\otimes \xi\in \mathcal{Z}(\g_{(n)})$, we have $\psi^{-1}(f\otimes \xi)=t^{-m}f\otimes t^m \otimes \xi$. Hence for such $m$, 
\begin{eqnarray}\nonumber
&& {\pi}_{\iota^{\C}_{R_0}(A),\iota^R_{R_0}(\boldsymbol{V})}(f\otimes \xi \otimes 1_{ (Cl(\fp,\beta))}) \big((g_1\otimes v)\otimes (g_2\otimes  s)\big)=   \\ \nonumber 
&& \iota^R_{R_0}(\chi_{ \g_{(n)},\boldsymbol{V}})(t^{-m}f\otimes t^m \otimes \xi)(g_1\otimes  v)\otimes (g_2\otimes  s)=   \\ \nonumber 
&&m_1\left( t^{-m}f  \otimes    \chi_{ \g_{(n)},\boldsymbol{V}}( t^m \otimes \xi )\right)(g_1\otimes  v)\otimes (g_2\otimes  s).
\end{eqnarray}

We now obtain another expression for 
$${\pi}_{\iota^{\C}_{R_0}(A),\iota^R_{R_0}(\boldsymbol{V})}(f\otimes \xi \otimes 1_{ (Cl(\fp,\beta))}) \big((g_1\otimes v)\otimes (g_2\otimes  s)\big).$$ 
 By Proposition \ref{LDC}, $H_{q^{R}_{R_0}(D(\g_{(n)},\beta,t^n))}(\iota^{R}_{R_0}(\boldsymbol{V}))$ contains a $\widetilde{K}$-type of highest weight $\mu\in\mathfrak{t}^*$.   By Theorem \ref{51},
$\mathbb{I}_{R_0}\otimes (\mu+ \rho_\mathfrak{k})^{\ft}_{\fh}\in \iota_{R_0}^{\C}(\fh)^*$ is an infinitesimal character of the $(\iota^{\C}_{R_0}(\fg),K)$-module  $\iota^{R}_{R_0}(\boldsymbol{V})$ 
with respect to the fundamental Cartan subfamily 
$\iota_{R_0}^{\C}(\fh)$. Hence for every $f\otimes \xi\in  \iota_{R_0}^{\C}(\mathcal{Z}(\fg))  =   \mathcal{Z}(\iota_{R_0}^{\C}(\fg))$,  
and every $(g_1\otimes v)\otimes (g_2\otimes  s)\in \iota^{R}_{R_0}(\boldsymbol{V})\otimes_{R_0}\iota^{\C}_{R_0}(\bigwedge \fp^- ) $,
\begin{eqnarray}\nonumber
&& {\pi}_{\iota^{\C}_{R_0}(A),\iota^R_{R_0}(\boldsymbol{V})}(f\otimes \xi \otimes 1_{ (Cl(\fp,\beta))}) \big((g_1\otimes v)\otimes (g_2\otimes  s)\big)=   \\ \nonumber 
&& \left(\widehat{\{ \mathbb{I}_{R_0}\otimes(\mu+\rho_{\fk})^{\ft}_{\fh}  \} } \circ \boldsymbol{\operatorname{HC}}^{\iota^{\C}_{R_0}(\fg)}_{\iota^{\C}_{R_0}(\fh)}(f\otimes \xi)\right)
 (g_1\otimes  v)\otimes (g_2\otimes  s).
\end{eqnarray}
Comparing the two expressions for the scalar operator  ${\pi}_{\iota^{\C}_{R_0}(A),\iota^R_{R_0}(\boldsymbol{V})}(f\otimes \xi \otimes 1_{ (Cl(\fp,\beta))})$ with the choice of $f=t^{-m}$, we see that for every $\xi \in \mathcal{Z}(\fg)$, and $m\in \mathbb{N}$ large enough,

\begin{eqnarray}\nonumber
&& m_1(1 \otimes  \chi_{  \g_{(n)},\boldsymbol{V}}( t^m \otimes \xi))=\widehat{\{ \mathbb{I}_{R_0}\otimes(\mu+\rho_{\fk})^{\ft}_{\fh}  \} } \circ \boldsymbol{\operatorname{HC}}^{\iota^{\C}_{R_0}(\fg)}_{\iota^{\C}_{R_0}(\fh)}(t^m\otimes \xi ).
\end{eqnarray}
We first note the left hand side of the last equation is equal to $\chi_{ \g_{(n)},\boldsymbol{V}}( t^m \otimes \xi)$. On the other hand the right hand side of the same equation is equal to 
\begin{eqnarray}\nonumber
&&
t^m \otimes \widehat{  (\mu+\rho_{\fk})^{\ft}_{\fh}   } \circ \operatorname{HC}^{\fg}_{\fh}(\xi)= m_1\left(1\otimes t^m \otimes \widehat{  (\mu+\rho_{\fk})^{\ft}_{\fh}   } \circ \operatorname{HC}^{\fg}_{\fh}(\xi)\right)= \\ \nonumber
 &&m_1\left(1\otimes    \left(\widehat{\{ \mathbb{I}_{R}\otimes(\mu+\rho_{\fk})^{\ft}_{\fh}  \} } \circ \boldsymbol{\operatorname{HC}_{\fh_{(n)}}^{\mathfrak{g}_{(n)}}}(t^m\otimes \xi) \right)\right)= \\ \nonumber 
&&\widehat{\{ \mathbb{I}_{R}\otimes(\mu+\rho_{\fk})^{\ft}_{\fh}  \} } \circ \boldsymbol{\operatorname{HC}_{\fh_{(n)}}^{\mathfrak{g}_{(n)}}}(t^m\otimes \xi).
\end{eqnarray} 

This shows that $\mathbb{I}_R\otimes (\mu+ \rho_\mathfrak{k})^{\ft}_{\fh}\in  \boldsymbol{\mathfrak{h}}_{(n)}^*$ is an infinitesimal character of $\boldsymbol{V}$  with respect to the fundamental Cartan subfamily $\boldsymbol{\fh}_{(n)}$. 

If  
$\boldsymbol{\lambda}\in \boldsymbol{\fh}_{(n)}^*$ is  an infinitesimal character of $\boldsymbol{V}$ with respect to $\boldsymbol{\fh}_{(n)}$
we must have 
\begin{eqnarray}\nonumber
&&   \widehat{\boldsymbol{\lambda}}(\mathbb{I}_R\otimes \operatorname{HC}_{\fh}^{\mathfrak{g}})|_{\g_{(n)}} =( \widehat{\mathbb{I}_R\otimes(\mu+\rho_{\fk})^{\ft}_{\fh}}) (\mathbb{I}_R\otimes   \operatorname{HC}_{\fh}^{\mathfrak{g}})|_{\g_{(n)}}
\end{eqnarray}
and 
$\boldsymbol{\lambda}$  must be  $W(\fg,\fh)$-conjugate to $\mathbb{I}_{R}\otimes(\mu+\rho_{\fk})^{\ft}_{\fh} $.

\end{proof}

 \section{The  case of $SL(2,\R)$}\label{SL}

In this section, we consider the deformation family  $(\g_d,K)$ of $SL(2,\R)$.  We calculate the Dirac cohomology of any generically irreducible algebraic family of $(\g_d,K)$-modules  and verify Vogan's conjecture explicitly. 

\subsection{Preliminaries and notations}

 In this case, using  the notation of the previous sections, we have

 \begin{equation*}
	\begin{array}{rlrl}
     G\hspace{-2mm}&=SL(2,\C), \hspace{2mm}&G(\R)\hspace{-2.5mm}&=G^{\sigma}=SL(2,\R),\\
     K\hspace{-2.5mm}&=SO(2,\C), \hspace{2mm}&K(\R)\hspace{-2.5mm}&=
     K^{\sigma}=SO(2,\R),\\
     \fg\hspace{-2.5mm}&=\mathfrak{sl}(2,\C),\hspace{2mm}&
    \fg^{\sigma}\hspace{-2mm}&=\mathfrak{sl}(2,\R),\\
    \fk\hspace{-2.5mm}&=\mathfrak{so}(2,\C),\hspace{2mm}&
       \fk^{\sigma}\hspace{-2mm}&=\mathfrak{so}(2,\R),\\
     \fp \hspace{-2.5mm}&=\{X\in \fg|X^t=X\},\hspace{2mm}&
      \fp^{\sigma} \hspace{-2mm}&=\{X\in \fg^{\sigma}|X^t=X\}.
\end{array}
\end{equation*}

We fix an $\mathfrak{sl}_2$-triple $\{x,y,h\}$ in $\fg$ satisfying the usual commutation relations
\[ [x,y]=h,\hspace{2mm} [h,x]=2x, \hspace{2mm}[h,y]=-2y,\]
and such that 
 $\mathfrak{k}= \mathbb{C}h$ and $\mathfrak{p}= \mathbb{C}x\oplus \C y$.

We define three sections in the constant family $\C[t]\otimes_{\C}\fg$ by 
\[\tilde{x}:=t\otimes x,\hspace{2mm} \tilde{y}:=t\otimes y,\hspace{2mm} \tilde{h}:=1\otimes h.\]
Then    $\g_d=\k\oplus \p_d$    where 
\begin{equation*}\k=\C[t]\otimes_{\C}\fk=\C[t]\tilde{h}
,\hspace{3mm}\p_d= \langle t\rangle \otimes_{\C}\fp=\mathbb{C}[t] \tilde{x}\oplus  \mathbb{C}[t] \tilde{y}.
\end{equation*}
The commutation relations of $\{\tilde{x},\tilde{y},\tilde{h}\} $  are given by 
\[[\tilde{h},\tilde{x}]=2\tilde{x}, \hspace{2mm}[\tilde{h},\tilde{y}]=-2\tilde{y},\hspace{2mm}[\tilde{x},\tilde{y}]=t^2\tilde{h}.\]

Let $\boldsymbol{\beta}_d$ be the $\mathbb{C}[t]$-bilinear symmetric $(\boldsymbol{\mathfrak{k}},K)$-invariant and orthogonalizable form on $\boldsymbol{\mathfrak{p}}_d$ given in Lemma \ref{s422}. With respect to the basis $\{\tilde{x},\tilde{y}\}$, the form $\boldsymbol{\beta}_d$ is given by the matrix
	\begin{equation*}
		\begin{pmatrix}
			0&4\\
			4&0
		\end{pmatrix}.
	\end{equation*}
	
The Clifford algebra $Cl(\boldsymbol{\mathfrak{p}}_d)$ of the quadratic space $(\boldsymbol{\mathfrak{p}}_d,\boldsymbol{\beta}_d)$, as a $\mathbb{C}[t]$-module, is generated by the elements $\{1,\gamma_d(\tilde{x}),\gamma_d(\tilde{y}),\gamma_d(\tilde{x})\gamma_d(\tilde{y})\}$ where $\gamma_d:\boldsymbol{\mathfrak{p}}_d\rightarrow Cl(\boldsymbol{\mathfrak{p}}_d)$ is the canonical embedding. 

    With respect to the Killing form, $\mathfrak{p}= \mathbb{C}x\oplus \C y$ is  a decomposition into dual maximal isotropic subspaces. We let $\fp^-=\C y$
and   $\fp^+=\C x$.   The corresponding spin module $\boldsymbol{S}_d$ for the Clifford algebra $Cl(\boldsymbol{\mathfrak{p}}_d)$  is the $\C[t]$-module $\boldsymbol{S}_d:=\bigwedge \boldsymbol{\mathfrak{p}}_d^-$, where $\boldsymbol{\mathfrak{p}}_d^-:=\mathbb{C}[t]\tilde{y}$. It is  equipped with the $Cl(\boldsymbol{\mathfrak{p}}_d)$-action $\gamma_d':Cl(\boldsymbol{\mathfrak{p}}_d)\rightarrow \mathrm{End}(\boldsymbol{S}_d)$ defined by
	\begin{align*}
		\gamma_d'(\tilde{x})1&=0&	\gamma_d'(\tilde{y})1=\tilde{y}\\
		\gamma_d'(\tilde{x})\tilde{y}&=4&	\gamma_d'(\tilde{y})\tilde{y}=\tilde{y}\wedge \tilde{y}=0.
	\end{align*}

    The action of 
    $\boldsymbol{\mathfrak{k}}$  on $\boldsymbol{S}_d$ is determined by 
	\begin{equation*}
		\boldsymbol{\alpha}(\tilde{h}):=\frac{1}{2}\gamma_d'(\tilde{x})\gamma_d'(\tilde{y})-1.
	\end{equation*}
	Under this action there are $2$ weight vectors for $\tilde{h}$ in $\boldsymbol{S}_d$, namely $1$ and $\tilde{y}$ of weights $1$ and $-1$, respectively.
	
	The Dirac element $D_d$ of the algebra $U(\boldsymbol{\mathfrak{g}}_d)\otimes_{\C[t]}	Cl(\boldsymbol{\mathfrak{p}}_d)$ is given by
	\begin{equation*}
		D_d=\frac{1}{4}\{\tilde{x}\otimes \gamma_d({\tilde{y}})+\tilde{y}\otimes\gamma_d(\tilde{x})\}\in U(\boldsymbol{\mathfrak{g}}_d)\otimes_{\C[t]} Cl(\boldsymbol{\mathfrak{p}}_d)
	\end{equation*}
	while
	\begin{equation*}
		D_d^2=\frac{1}{16}(\tilde{x}\tilde{y}\otimes\gamma_d(\tilde{y})\gamma_d(\tilde{x})+\tilde{y}\tilde{x}\otimes\gamma_d(\tilde{x})\gamma_d(\tilde{y})).
	\end{equation*}
	%
	%
	%
	
	In what follows, we compute the Dirac cohomology of all generically irreducible $(\boldsymbol{\mathfrak{g}}_d,K)$-modules $\boldsymbol{V}$. In the study of these $(\boldsymbol{\mathfrak{g}}_d,K)$-modules, we will need the element
		$\Omega_d:=\Omega(\boldsymbol{\mathfrak{g}}_d,\beta,t)=t^2\otimes\Omega_{\mathfrak{g}}\in \mathcal{Z}(\g_d)$
	where $\Omega_\mathfrak{g}$ stands for the Casimir element of $\mathfrak{g}$ (see Subsection 4.4). Explicitly in terms of the chosen basis 
		\begin{align*}
		\Omega_d
		&=
		\frac{1}{8}t^2\tilde{h}^2+\frac{1}{4}\tilde{x}\tilde{y}+\frac{1}{4}\tilde{y}\tilde{x}.
	\end{align*}
The center of the universal enveloping algebra of $\g_d$ is the polynomial ring over $\C[t]$ generated by $\Omega_d$, i.e., 
$\mathcal{Z}(\g_d)=\C[t][\Omega_d]$.

The Lie algebra $\k$ is  also a fundamental (actually $\theta$-fixed) Cartan subfamily of $\g_d$ and 
\[\g_d=\boldsymbol{\mathfrak{n}}^-\oplus \k \oplus \boldsymbol{\mathfrak{n}}^+\] 
is a  $\theta$-stable triangular decomposition where $\boldsymbol{\mathfrak{n}}^-=\C[t]\tilde{y}$, and $\boldsymbol{\mathfrak{n}}^+=\C[t]\tilde{x}$.
The Harish-Chandra homomorphism of $\g_d$ with respect to the abovementioned triangular decomposition satisfies

$$\boldsymbol{\operatorname{HC}_{\fk}^{\mathfrak{g}_d}}(\Omega_d)=\frac{t^2}{8}(\tilde{h}^2-1). $$

	\subsection{Generically irreducible $(\boldsymbol{\mathfrak{g}}_d,K)$-modules}
    In this subsection we recall some facts on generically irreducible $(\boldsymbol{\mathfrak{g}}_d,K)$-modules.  We will recall two basic  invariants of generically irreducible $(\boldsymbol{\mathfrak{g}}_d,K)$-modules, the   set  of $K$-types and  the action of the Casimir.  We shall explain  how these invariants can be used to explicitly describe  generically irreducible $(\boldsymbol{\mathfrak{g}}_d,K)$-modules.  
    We   provide no proofs, these are  easily deduced from the classification given in  \cite{eyallietheory}. 

    \subsubsection{The set of $K$-types}
The set of $K$-types $\mathcal{K}(\boldsymbol{V})$ of a generically irreducible  $(\boldsymbol{\mathfrak{g}}_d,K)$-module $\boldsymbol{V}$ must coincide with the set of  $K$-types  $\mathcal{K}(V)$ of some  irreducible $(\mathfrak{sl}(2,\C)),K)$-module $V$. The various  $\mathcal{K}(V)$  can be identified with the set of eigenvalues (or weights) of ${h}$ in $V$. These are subsets of integers of the same parity.  

We shall freely identify $K$-types and weights for the Cartan subalgebra $\fk$. 

Below we  specify $\mathcal{K}(V)$ for any irreducible $(\mathfrak{sl}(2,\C)),K)$-module $V$.
\begin{enumerate}
    \item Discrete series representations: 
    \begin{enumerate}
        \item Positive (or  lowest weight) Discrete series: For any $m\in \mathbb{N}$, there is an irreducible $(\mathfrak{sl}(2,\C)),K)$-module $DS_m^+$ with $\mathcal{K}(DS_m^+)=\{m,m+2,.... \}$.
        \item Negative (or highest weight) Discrete series: For any $m\in \mathbb{N}$, there is an irreducible $(\mathfrak{sl}(2,\C)),K)$-module $DS_m^-$ with $\mathcal{K}(DS_m^-)=\{-m,-m-2,.... \}$.
    \end{enumerate}
    \item Finite dimensional representations: For any $m\in \mathbb{N}_0$, there is an irreducible $(m+1)$-dimensional  $(\mathfrak{sl}(2,\C)),K)$-module $F_m$ with $\mathcal{K}(F_m)=\{-m,-m+2,...,m\}$.
    \item Principal series representations: 
    \begin{enumerate}
        \item Spherical principal series: $\mathcal{K}(V)=2\Z$
        \item Non-Spherical principal series: $\mathcal{K}(V)=2\Z+1$.
    \end{enumerate}
\end{enumerate}
\begin{remark*}
    As mentioned above, we identify each $K$-type with an eigenvalue of $h$. Below we shall also identify the eigenvalue $m\in \Z$ of $h$ with the linear functional $\mu_m\in \fk^*$ that is defined by $\mu_m(h)=m$. 
\end{remark*}
   \subsubsection{Action of the Casimir and infinitesimal character}
On a generically irreducible  $(\boldsymbol{\mathfrak{g}}_d,K)$-module $\boldsymbol{V}$ the Casimir $\Omega_d$ must act via multiplication by a polynomial function $\omega_{\boldsymbol{V}}(t)\in \C[t]$. 

The generically irreducible $(\boldsymbol{\mathfrak{g}}_d,K)$-module   $\boldsymbol{V}$ has an infinitesimal character with respect to the fundamental   Cartan subfamily $\k$   if and only if there is $\boldsymbol{\lambda}\in \k^*=\operatorname{Hom}_{\C[t]}(\k,\C[t])=\C[t]\otimes_{\C}\fk^*=\C[t]\otimes_{\C}\operatorname{Hom}_{\C}(\fk,\C)$, such that 
\[\omega_{\boldsymbol{V}}(t)=\frac{t^2}{8}(\boldsymbol{\lambda}(\tilde{h})^2-1).\]

\subsubsection{The weight basis}
Let $\boldsymbol{V}$ be a generically irreducible  $(\boldsymbol{\mathfrak{g}}_d,K)$-module. Then as $K$-modules over $\C[t]$ we have a direct sum decomposition 
\[\boldsymbol{V}=\bigoplus_{n\in \mathcal{K}(V)}\boldsymbol{V}_n, \]
where the action of $\tilde{h}$ on each $\boldsymbol{V}_n$ is given by  multiplication by $n$.

Let $v_n$ be a generator of $V_n$ over $\C[t]$.
The commutation relations of  $\{\tilde{x},\tilde{y},\tilde{h}\} $ imply that for every $n\in\mathcal{K}(V) $, there are polynomials $A_n(t),B_n(t)\in \C[t]$   such that 
	\begin{align*}
		\tilde{x} v_n&=A_n(t)v_{n+2}\\
		\tilde{y} v_{n+2}&=B_n(t)v_n,
	\end{align*}
	where by convention, we set $A_{n}(t)\equiv 0$, $v_{n+2}=0$ when $n+2\not \in \mathcal{K}(V)$, and $B_{n}(t)\equiv 0$, $v_n=0$ when $n\not \in \mathcal{K}(V)$.

Since $\Omega_d=\frac{1}{8}t^2\tilde{h}^2+\frac{1}{4}t^2\tilde{h}+\frac{1}{2}\tilde{y}\tilde{x}$, for every $n\in\mathcal{K}(V)$,
\begin{eqnarray}\nonumber
   && \omega_{\boldsymbol{V}}(t)=\frac{1}{8}t^2n^2+\frac{1}{4}t^2n+\frac{1}{2}A_n(t)B_n(t).
\end{eqnarray}
Equivalently $A_n(t)B_n(t)= 2\omega_{\boldsymbol{V}}(t)-\frac{t^2}{4}n\left(n+  2\right)$.

	\subsection{Dirac cohomology of   $(\boldsymbol{\mathfrak{g}}_d,K)$-modules}
In this subsection we calculate the Dirac cohomology of  any generically irreducible  $(\boldsymbol{\mathfrak{g}}_d,K)$-module and we verify Vogan's conjecture.
\begin{theorem*}
		Let $\boldsymbol{V}$ be a generically irreducible $(\boldsymbol{\mathfrak{g}}_d,K)$-module. Then $H_{D_{d}}(\boldsymbol{V})=0$ if  and only if   $\mathcal{K}(\boldsymbol{V})$ is equal to the set of $K$-types of an $(\mathfrak{sl}(2,\C)),K)$-module of a principal series representation. 
        
    Moreover, 
    \begin{enumerate}
    \item     If $\mathcal{K}(\boldsymbol{V})=\mathcal{K}(DS_m^+)$, as $\k$-modules over $\C[t]$   \[H_{D_{d}}(\boldsymbol{V})\cong \mathbb{C}[t]\{v_m \otimes\tilde{y}\},\] 
    and in particular the only weight   of $H_{D_{d}}(\boldsymbol{V})$ is $\mu_{m-1}$.
        \item     If $\mathcal{K}(\boldsymbol{V})=\mathcal{K}(DS_m^-)$, as $\k$-modules over $\C[t]$   
        \[H_{D_{d}}(\boldsymbol{V})\cong \mathbb{C}[t]\{v_{-m} \otimes1\},\]
        and in particular the only weight   of $H_{D_d}(\boldsymbol{V})$ is $\mu_{-m+1}$.
         \item     If $\mathcal{K}(\boldsymbol{V})=\mathcal{K}(F_m)$, as $\k$-modules over $\C[t]$   \[H_{D_d}(\boldsymbol{V})\cong\mathbb{C}[t]\{v_m\otimes 1,v_{-m}\otimes \tilde{y}\},\] and in particular the only weights   of $H_{D_d}(\boldsymbol{V})$ are $\mu_{m+1}$, $\mu_{-m-1}$.
    \end{enumerate}

\end{theorem*}

\begin{proof}
    	Assume that $\mathcal{K}(\boldsymbol{V})=2\Z$ (the $K$-types of spherical principal series). 
        The weights of the $\boldsymbol{\mathfrak{k}}$-module $\boldsymbol{V}\otimes_{\C[t]}\boldsymbol{S}$ are parametrized by all odd integers. 
	The $(2m+1)$-weight subspace  $(\boldsymbol{V}\otimes_{\C[t]}\boldsymbol{S})_{2m+1}$  is a  free rank two $\C[t]$-module with basis 
	$\{v_{2m}\otimes1, v_{2(m+1)}\otimes \tilde{y}\}$. The Dirac operator acts on  $(\boldsymbol{V}\otimes_{\C[t]}\boldsymbol{S})_{2m+1}$ via
\begin{eqnarray}\nonumber   &D_d\big(v_{2m}\otimes1\big)&=\frac{1}{4}A_{2m}(t)v_{2(m+1)}\otimes \tilde{y} \\ \nonumber 
    &D_d\big(v_{2(m+1)}\otimes \tilde{y}\big)&=\frac{1}{4}B_{2m}(t)v_{2m}\otimes 1. 
\end{eqnarray}	
In order to have a non-zero kernel for $D_d$   on   $(\boldsymbol{V}\otimes_{\C[t]}\boldsymbol{S})_{2m+1}$  at least one of the polynomials $A_{2m}(t)$, $B_{2m}(t)$ must be identically zero. This can not be since $\boldsymbol{V}$ is generically irreducible. Hence in this case $H_D(\boldsymbol{V})=0$. A similar argument holds for the case of $\mathcal{K}(\boldsymbol{V})=2\Z+1$ (the $K$-types of non-spherical principal series).

Assume that $\mathcal{K}(\boldsymbol{V})=\mathcal{K}(DS_m^+)=\{m,m+2,.... \}$ for some $m\in \N$.  
The action of the element $\Omega_d$ on the lowest weight vector $v_m$ of $\boldsymbol{V}$ is 
	\begin{equation*}
		\Omega_d v_m=\frac{t^2}{8}m\big(m-2\big)v_m.
	\end{equation*}
	 Hence for every $n\in \mathcal{K}(\boldsymbol{V})$,
\begin{equation}\label{rel1}
	A_n(t)B_n(t)=\frac{t^2}{4}\{m(m-2)-n(n+2)\}.
	\end{equation}
	 
	The weights of the $\boldsymbol{\mathfrak{k}}$-module $\boldsymbol{V}\otimes_{\C[t]}\boldsymbol{S}$ are parametrized by the integers $m-1,m+1,m+3,\ldots$.
    
	For $k\geq 1$, the $(m+2k-1)$-weight subspace $(\boldsymbol{V}\otimes_{\C[t]}\boldsymbol{S})_{m+2k-1}$ is a free $\C[t]$-module of rank two with basis
	$\{v_{m+2(k-1)}\otimes1,v_{m+2k}\otimes \tilde{y}\}$. The $(m-1)$-weight subspace $(\boldsymbol{V}\otimes_{\C[t]}\boldsymbol{S})_{m-1}$ is a free rank one $\C[t]$-module  with basis $\{v_m\otimes \tilde{y}\}$.
   
    For $k\geq1$, the Dirac operator acts on  $(\boldsymbol{V}\otimes_{\C[t]}\boldsymbol{S})_{m+2k-1}$ via
    \begin{eqnarray}\nonumber
    & D_d\big(v_{m+2(k-1)}\otimes1\big)&=\frac{1}{4}A_{m+2(k-1)}(t)v_{m+2k}\otimes \tilde{y}\\ \nonumber
    &  D_d\big(v_{m+2k}\otimes \tilde{y}\big)&=\frac{1}{4}B_{m+2(k-1)}(t)v_{m+2(k-1)}\otimes 1.  
    \end{eqnarray}
 Hence as in the abovementioned case of principal series, on these rank two weight subspaces the kernel of the Dirac operator is trivial.

	On the other hand, $D_d$ annihilates the vector $v_m\otimes \tilde{y}$ so that $\ker D_d=\mathbb{C}[t]\{v_m\otimes \tilde{y}\}$. Clearly  $v_m\otimes \tilde{y}$   does not belong to $\mathrm{im}\hspace{0.5mm}D_{d}$.
	Hence 
	\begin{equation*}
		H_{D_{d}}(\boldsymbol{V})\cong\mathbb{C}[t]\{v_m\otimes \tilde{y}\}.
	\end{equation*}
    The case of $\mathcal{K}(\boldsymbol{V})=\mathcal{K}(DS_m^-)$ and $\mathcal{K}(\boldsymbol{V})=\mathcal{K}(F_m)$ are proven similarly.
\end{proof}
 
\begin{corollary*}
   Let $\boldsymbol{V}$ be a generically irreducible $(\boldsymbol{\mathfrak{g}}_d,K)$-module with nonzero Dirac cohomology. Then $\boldsymbol{V}$ has an infinitesimal character $\boldsymbol{\lambda}\in \operatorname{Hom}_{\C[t]}(\k,\C[t])$
 with respect to the fundamental  Cartan subfamily $\k$.
 
 Moreover, for any highest weight $\mu\in \fk^*$ such that the corresponding  weight space   $H_{D_{d}}(\boldsymbol{V})_{\mu}$ is non-zero  there is $w\in W(\fg,\fk)$ such that 
 \[ \boldsymbol{\lambda}=1\otimes w\cdot \mu \in \C[t]\otimes_{\C} \fk^* =\k^*. \]
\end{corollary*}
\begin{proof}
    From the last theorem  we know that there is some $m\in \N_0$, such that 
    $\mathcal{K}(\boldsymbol{V})$ is either $\mathcal{K}(DS_m^+)$ or $ \mathcal{K}(DS_m^-)$ or $\mathcal{K}(F_m)$. We shall consider the case with $\mathcal{K}(\boldsymbol{V})=\mathcal{K}(F_m)$. The other cases are handled similarly. 
    
By direct calculations \begin{equation*}
		\omega_{\boldsymbol{V}}(t)=\frac{t^2}{8}m(m+2)= \frac{t^2}{8}((m+1)^2 -1).
	\end{equation*}
     
    Since  $H_{D_d}(\boldsymbol{V})\cong\mathbb{C}[t]\{v_m\otimes 1,v_{-m}\otimes \tilde{y}\}$, the highest weights that appear in $H_{D_d}(\boldsymbol{V})$ are $\mu_{\pm(m+1)}\in \fk^*$. 
The Weyl group $W(\fg,\fk)= \{1,w_0\}$ is the two-element  group where the non-trivial  element $w_0$ acts via multiplication by $-1$ on $\fk^*$. In particular $w_0\cdot \mu_{\pm(m+1)}=\mu_{\mp(m+1)}$.

Clearly for $\boldsymbol{\lambda}_{\pm}:=1\otimes \mu_{\pm(m+1)}$ we have 
 \[\omega_{\boldsymbol{V}}(t)=\frac{t^2}{8}(\boldsymbol{\lambda}_{\pm}(\tilde{h})^2-1),\]
 and each of the functionals $\boldsymbol{\lambda}_{\pm}$ is a an infinitesimal character of  $\boldsymbol{V}$
 with respect to the fundamental $\theta$-stable Cartan subfamily $\k$.

Any other such infinitesimal character  
    $\boldsymbol{\lambda}\in \operatorname{Hom}_{\C[t]}(\k,\C[t])$ must satisfy $\boldsymbol{\lambda}(\tilde{h})^2=\boldsymbol{\lambda}_{\pm}(\tilde{h})^2$. Hence $\boldsymbol{\lambda}$ must belong to the orbit of $W(\fg,\fk)$ that is given by 
$\{\boldsymbol{\lambda}_{+},\boldsymbol{\lambda}_{-}\}=\{1\otimes \mu_{(m+1)},1\otimes \mu_{-(m+1)}\}$.
 
\end{proof}

\printbibliography
\Addresses
\end{document}